%% file: main.tex
\documentclass[11pt]{article}

\input{ex_shared}

\begin{document}
	
\maketitle	

\begin{abstract}
We consider the problem of density estimation in the context of multiscale Langevin diffusion processes, where a single-scale homogenized surrogate model can be derived. In particular, our aim is to learn the density of the invariant measure of the homogenized dynamics from a continuous-time trajectory generated by the full multiscale system. We propose a spectral method based on a truncated Fourier expansion with Hermite functions as orthonormal basis. The Fourier coefficients are computed directly from the data owing to the ergodic theorem. We prove that the resulting density estimator is robust and converges to the invariant density of the homogenized model as the scale separation parameter vanishes, provided the time horizon and the number of Fourier modes are suitably chosen in relation to the multiscale parameter. The accuracy and reliability of this methodology is further demonstrated through a series of numerical experiments.
\end{abstract}

\textbf{AMS subject classifications.} 42C10, 60H10, 60J60, 62M20, 62G05.

\textbf{Key words.} Ergodicity, Fourier expansion, Hermite functions, homogenization, invariant measure, multiscale Langevin dynamics.

\section{Introduction} \label{sec:intro}

Stochastic systems with multiscale structure arise naturally in a wide range of applications across the physical, biological, and engineering sciences. A typical example is the overdamped Langevin diffusion with both slow and fast scales, where the dynamics evolve under the influence of a potential landscape with multiple scales. For such systems, direct analysis or simulation of the full multiscale model often becomes computationally prohibitive. A widely used strategy in this context is homogenization theory, which enables replacing the original multiscale model with an effective single-scale surrogate model that accurately approximates the slow dynamics in the limit of vanishing scale separation parameter \cite{BLP78,PaS08}.

In many practical situations, data are collected from the full multiscale dynamics, while the modeling objective is to infer effective (or homogenized) properties that emerge asymptotically. This discrepancy gives rise to the problem of model misspecification, originally investigated in \cite{PaS07} in the context of the maximum likelihood estimator (MLE) for Langevin dynamics. It was shown that the MLE becomes (even asymptotically) biased when applied directly to multiscale data without appropriate preprocessing. Since then, various model calibration methods have been proposed to mitigate this issue, including subsampling strategies \cite{PaS07,ABJ13}, techniques based on martingale properties and conditional expectations \cite{KPK13,KKP15,Kru18}, combinations of filtered data and MLEs \cite{AGP23,GaZ23}, eigenfunction-based estimators \cite{APZ22,Zan23}, stochastic gradient descent in continuous time \cite{HiZ24}, and more recently, minimum distance estimators \cite{BKP25}.

In this work, we focus on the long-time behavior of stochastic differential equations (SDEs), specifically on their invariant measures, which encode the equilibrium statistics of the system. The problem of estimating the invariant density from data has previously been studied for multidimensional ergodic SDEs; see, e.g., \cite{Bia07}. Related approaches include high-order numerical schemes \cite{AVZ14}, adaptive methods under anisotropic Hölder regularity \cite{Str18}, and estimation from discrete-time observations with controlled discretization error \cite{AmG23}. Moreover, density estimation for jump diffusions is treated in \cite{AmN22}, and a central limit theorem for one-dimensional kernel estimators is established in \cite{Pri01}. However, to the best of our knowledge, the problem of learning the invariant density in a multiscale setting has not been addressed yet. 

Our goal is to estimate the invariant density of a homogenized one-dimensional overdamped Langevin diffusion, given a single long trajectory generated by the original multiscale system. Although the homogenized model can often be derived analytically via homogenization theory, its invariant measure is not directly observable from multiscale data, especially when the scale separation is finite. In fact, naïve density estimation from such data captures fast-scale oscillations, resulting in a biased estimate that fails to reflect the coarse-grained behavior of interest. To overcome this challenge, we propose a nonparametric spectral estimator based on a truncated Fourier expansion using Hermite functions as an orthonormal basis. By leveraging the ergodicity of the multiscale system, the estimator approximates the Fourier coefficients through time averages computed along the trajectory. Crucially, our method does not require explicit knowledge of either the homogenized SDE or the multiscale one.

We show that the estimator is robust to model misspecification and converges to the true invariant density of the homogenized model under suitable conditions on the time horizon and the number of retained Fourier modes, relative to the multiscale parameter. In particular, we provide a convergence result establishing the asymptotic unbiasedness of the estimator as the scale separation parameter tends to zero, provided that the number of Fourier modes grows slowly enough with respect to the scale separation parameter and the observation time is large enough. Our theoretical analysis relies on properties of Hermite functions, as well as recent limit theorems for multiscale diffusions with time horizons depending on the scale parameter \cite{BK25}. Numerical experiments confirm that, when appropriately tuned, the estimator accurately recovers the invariant density of the homogenized model while ignoring irrelevant fine-scale fluctuations. We believe that this methodology not only provides a systematic solution to the problem of invariant density estimation in multiscale systems but also opens the door to novel approaches for nonparametric estimation of the homogenized surrogate model in the multiscale data setting.

\paragraph{Outline.} The remainder of the paper is structured as follows. In \cref{sec:problem}, we present the model and the density estimation problem and outline our approach based on Fourier series with Hermite functions. In \cref{sec:numerics}, we describe various numerical experiments that illustrate the effectiveness of our method. In \cref{sec:analysis}, we provide a theoretical analysis of the proposed estimator, establishing its asymptotic unbiasedness under certain conditions. Finally, in \cref{sec:conclusion} we discuss potential directions for future work.

\section{Problem setting} \label{sec:problem}

We consider the one-dimensional multiscale overdamped Langevin diffusion defined by the following SDE in the time interval $[0,T]$ with deterministic initial condition $x_0 \in \R$
\begin{equation} \label{eq:SDE_multiscale}
\d X_t^\epl = - V'(X_t^\epl) \dd t - \frac1\epl p' \left( \frac{X_t^\epl}\epl \right) \dd t + \sqrt{2\sigma^2} \dd W_t, \quad X_0^\epl = x_0,
\end{equation}
where $V,p \colon \R \to \R$ are the confining slow-scale and $\period$-periodic fast-scale potentials, respectively, and $W := \left( W(t), \mathcal{F}_t \right)_{t \in [0, \infty)}$ is a standard one-dimensional Brownian motion defined on an underlying probability space $(\Omega, \mathcal F, \Pr)$ equipped with the natural filtration $\left( \mathcal{F}_t \right)_{t \in [0, \infty)}$. Moreover, $\epl>0$ is the scale separation parameter and $\sigma > 0$ is the diffusion coefficient. The confining slow-scale and fast-scale potentials satisfy the following conditions.
\begin{assumption} \label{as:potentials}
The potentials $V$ and $p$ are such that:
\begin{enumerate}[leftmargin=*,label=\roman*)]
\item $V, p \in C^2(\R)$ and, without loss of generality, it holds $V(0) = 0$ and $p(0) = 0$;
\item $V'$ is globally Lipschitz on $\R$ with Lipschitz constant $L_V> 0 $;
\item there exist $\beta > 0$ and $R \ge 1$ such that for all $x \in \R$ with $\abs{x} \ge R$
\begin{equation}    
- \sign(x) V'(x) \le -\beta \abs{x}.
\end{equation}
\end{enumerate}
\end{assumption}

Due to the theory of homogenization (see, e.g., \cite[Chapter 3]{BLP78} or more recently \cite{RoL21}), there exists an effective SDE 
\begin{equation} \label{eq:SDE_homogenized}
\d X_t = - \V'(X_t) \dd t + \sqrt{2\Sigma} \dd W_t, \quad X_0 = x_0,
\end{equation}
such that $X_t^\epl$ converges to $X_t$ as $\epl \to 0$ in law as random processes in $\mathcal C^0([0,T])$. The homogenized drift term and diffusion coefficient are defined as $\V'(x) = \mathcal K V'(x)$ and $\Sigma = \mathcal K \sigma^2$, where $\mathcal K > 0$ is given by $\mathcal K = \period^2/(\Pi \cdot \widehat \Pi)$, with
\begin{equation} \label{eq:Pi_def}
\Pi = \int_0^\period e^{-\frac1\sigma p(y)} \dd y \qquad \text{and} \qquad \widehat \Pi = \int_0^\period e^{\frac1\sigma p(y)} \dd y.
\end{equation}
In the dissipative setting of \cite[Assumptions 3.1]{PaS07}, which is implied by our \cref{as:potentials}, it is proved in \cite[Propositions 5.1 and 5.2]{PaS07} that the processes $X_t^\epl$ and $X_t$ in the equations \eqref{eq:SDE_multiscale} and \eqref{eq:SDE_homogenized} are geometrically ergodic with unique invariant measures
\begin{equation} \label{eq:invariant_density}
\rho^\epl(x) = \frac1{Z^\epl} e^{- \frac1{\sigma^2} \left( V(x) + p \left( \frac{x}\epl \right) \right)} \qquad \text{and} \qquad \rho(x) = \frac1Z e^{- \frac1\Sigma \V(x)} = \frac1Z e^{- \frac1{\sigma^2} V(x)},
\end{equation}
where 
\begin{equation}
Z^\epl = \int_\R e^{- \frac1\sigma \left( V(x) + p \left( \frac{x}\epl \right) \right)} \dd x  \qquad \text{and} \qquad Z = \int_\R e^{- \frac1\Sigma \V(x)} \dd x = \int_\R e^{- \frac1{\sigma^2} V(x)} \dd x.
\end{equation}
We are interested in the problem of inferring the invariant density $\rho$ of the homogenized SDE \eqref{eq:SDE_homogenized} given a continuous-time trajectory $(X_t^\epl)_{t\in[0,T]}$ originated from the multiscale dynamics \eqref{eq:SDE_multiscale}.

\begin{remark} \label{rem:multidimensional}
In this work, we focus on one-dimensional stochastic processes in order to present our method more clearly and to carry out a rigorous theoretical analysis. However, the proposed approach can, in principle, be extended to multidimensional Langevin diffusions by employing tensor products of Hermite functions as basis functions. As an illustration, we include a numerical example for a two-dimensional model in \cref{sec:num2D}, although a convergence analysis of the estimator in higher dimensions is left for future work. We note that for $d$-dimensional multiscale diffusions with a periodic fast-scale potential $p$, where $p$ has period $\period_i$ in the $i$-th coordinate direction, the corresponding homogenization constant $\mathcal{K}$ becomes a matrix in $\R^{d \times d}$. It is given by
\begin{equation}
\mathcal{K} = \int_{\bigtimes_{i=1}^d [0, \period_i]} (I_d + \nabla \Phi(y))(I_d + \nabla \Phi(y))^\top \nu(y) \dd y,
\end{equation}
where $I_d$ is the identity matrix and $\nabla \Phi$ denotes the Jacobian of the solution $\Phi \colon \R^d \to \R^d$ to the vector-valued cell problem
\begin{equation}
- \nabla \Phi \nabla p + \sigma \Delta \Phi = \nabla p,
\end{equation}
defined on the domain $\bigtimes_{i=1}^d [0, \period_i]$ with periodic boundary conditions. Finally, the function $\nu$ is the associated invariant density and takes the form
\begin{equation}
\nu(y) = \frac{e^{-\frac{1}{\sigma} p(y)}}{\int_{\bigtimes_{i=1}^d [0, \period_i]} e^{-\frac{1}{\sigma} p(z)} \dd z}.
\end{equation}
\end{remark}

\subsection{Nonparametric estimation of the invariant density}

Inspired by \cite[Section 4]{CGL24}, we consider an orthonormal basis $\{ \psi_n \}_{n=0}^{\infty}$ of $L^2(\R)$ and represent the invariant density $\rho$ via its Fourier expansion. We then construct an estimator by truncating this expansion and approximating the corresponding Fourier coefficients using the available data $(X_t^\epl)_{t\in[0,T]}$. 

In particular, we focus on the orthonormal basis of Hermite functions defined by
\begin{equation}
\psi_n(x) = \frac1{\sqrt{\sqrt\pi 2^n n!}} e^{- \frac{x^2}2} H_n(x),
\end{equation}
where $H_n$ denotes the $n$-th physicist's Hermite polynomial; see, e.g., \cite[Chapter V]{Sze75} and \cite{CGO21}. The Fourier expansion of $\rho$ with respect to this basis is
\begin{equation}
\rho(x) = \sum_{n=0}^\infty \alpha_n \psi_n(x), \qquad \text{with} \qquad \alpha_n = \int_\R \psi_n(x) \rho(x) \dd x.
\end{equation}
Recall that our goal is to approximate $\rho$ from a trajectory $(X_t^\epl)_{t\in[0,T]}$ solving equation \eqref{eq:SDE_multiscale}. Hence, due to ergodicity and homogenization, each Fourier coefficient $\alpha_n$ can be approximated by $\widehat \alpha_n^{T, \epl}$ given by
\begin{equation}
\widehat \alpha_n^{T, \epl} = \frac1T \int_0^T \psi_n(X_t^\epl) \dd t,
\end{equation}
since it holds
\begin{equation}
\lim_{\epl \to 0} \lim_{T \to \infty} \frac1T \int_0^T \psi_n(X_t^\epl) \dd t = \lim_{\epl \to 0} \int_\R \psi_n(x) \rho^\epl(x) \dd x = \int_\R \psi_n(x) \rho(x) \dd x, \qquad a.s.,
\end{equation}
which implies
\begin{equation}
\lim_{\epl \to 0} \lim_{T \to \infty} \widehat \alpha_n^{T,\epl} = \alpha_n, \qquad a.s.
\end{equation}
Hence, we introduce the estimator $\widehat \rho_N^{T,\epl}$ for the invariant density $\rho$, constructed by truncating the Fourier series and retaining only the first $N$ approximated coefficients
\begin{equation} \label{eq:estimator_density}
\widehat \rho_N^{T,\epl}(x) = \sum_{n=0}^{N-1} \widehat \alpha_n^{T,\epl} \psi_n(x).
\end{equation}
Note that the estimator $\widehat \rho_N^{T,\epl}$ depends on three key parameters: the scale separation $\epl$, the trajectory length $T$, and the number of Fourier coefficients $N$. Among these, $\epl$ is typically unknown, $T$ reflects the amount of available data, and $N$ is a tunable parameter of the method. Intuitively, one might expect the estimator to converge to the true invariant measure $\rho$ in the limit of infinite data ($T \to \infty$), infinitely many Fourier modes ($N \to \infty$), and vanishing scale separation ($\epl \to 0$). However, this convergence does not always hold. In fact, the estimator is asymptotically unbiased only under specific conditions. In particular, the number of Fourier modes $N$ should not grow too quickly relative to the scale separation parameter $\epl$. That is, $N$ should be carefully chosen as a function of $\epl$. On the other hand, the final time $T$ must diverge sufficiently fast. These requirements are made precise in the next theorem, which constitutes the main result of this work and whose proof is postponed to \cref{sec:analysis}.

\begin{theorem} \label{thm:convergence}
Let $\widehat \rho_N^{T,\epl}$ be the estimator defined in equation \eqref{eq:estimator_density}, and let $\rho$ denote the invariant density given in equation \eqref{eq:invariant_density}. Let \cref{as:potentials} hold and assume that the number of Fourier modes and the final time of observation scale with $\epl$ as
\begin{equation}
N = N(\epl) = \left\lfloor \frac{\pi^2}{\gamma \period^2 \epl^2} \right\rfloor, \quad\text{and}\quad T = T(\epl) = \kappa\epl^{-\zeta},
\end{equation}
respectively, for some $\kappa > 0$, where the parameters $\gamma$ and $\zeta$ satisfy
\begin{equation}
\gamma > \begin{cases}
3 + \log(8), & \text{if } \sigma = 1, \\
c^2 + \max\left\{16e^{\frac32}, \frac{c}{4}\left(\log\abs{\frac{2}{c}-1}+2\log(4)\right)\right\}, & \text{if } \sigma \ne 1,
\end{cases}
\qquad
\zeta > \begin{cases}
5, & \text{if } r \ge l, \\
5\frac{l}{r}, & \text{if } r < l,
\end{cases}
\end{equation}
with
\begin{equation}
c = \frac{\sigma^2 + 1}{\sigma^2}, \qquad l = \frac{L_V + \abs{V'(0)}}{\sigma^2}, \qquad r = \frac{\beta}{\sigma^2}.
\end{equation}
Then, it holds that
\begin{equation}
\lim_{\epl \to 0} \E \left[ \norm{\widehat \rho_{N(\epl)}^{T(\epl),\epl} - \rho}_{L^2(\R)}^2 \right] = 0.
\end{equation}
\end{theorem}

\begin{remark} \label{rem:inference_wavelength}
\cref{thm:convergence} shows that the optimal number of Fourier coefficients $N$ depends on the wavelength $\period\epl$, which is typically unknown. A heuristic approach to estimate this wavelength involves analyzing the Fourier transform of the estimator $\widehat \rho_N^{T,\epl}$ in the regime where $N \gg 1$ and the condition in \cref{thm:convergence} is not satisfied. In this case, fast-scale oscillations emerge in the estimator, and the dominant wavelength $\period\epl$ can be inferred by identifying the most significant non-zero frequency $\bar\xi \ne 0$ in the frequency domain \cite{Bra78}. In particular, the Fourier transform of $\widehat \rho_N^{T,\epl}$ reads
\begin{equation}
\mathcal F(\widehat \rho_N^{T,\epl})(\xi) = \int_\R \widehat \rho_N^{T,\epl}(x) e^{-2\pi i\xi x} \dd x = \sum_{n=0}^{N-1} \widehat \alpha_n^{T,\epl} \int_\R \psi_n(x) e^{-2\pi i\xi x} \dd x,
\end{equation}
and, using the fact that the Hermite functions $\psi_n$ are eigenfunctions of the Fourier transform, we obtain
\begin{equation}
\mathcal F(\widehat \rho_N^{T,\epl})(\xi) = \sqrt{2\pi} \sum_{n=0}^{N-1} (-i)^n \widehat \alpha_n^{T,\epl} \psi_n(2\pi\xi).
\end{equation}
Then, by plotting the magnitude of the Fourier transform $\abs{\mathcal F(\widehat \rho_N^{T,\epl})(\xi)}$, we can identify the dominant frequency $\bar\xi \ne 0$. From this, the wavelength can be estimated as $\period\epl \simeq 1/\bar\xi$. Finally, we note that knowing the wavelength $\period\epl$ also enables selecting a suitable observation time $T$, as prescribed in \cref{thm:convergence}.
\end{remark}

\section{Numerical illustration} \label{sec:numerics}

In this section, we present a series of numerical experiments to assess the performance of the proposed spectral estimator $\widehat \rho_N^{T,\epl}$ in approximating the invariant density $\rho$ of the homogenized dynamics \eqref{eq:SDE_homogenized}. We begin by illustrating the influence of the trajectory length $T$ and the truncation level $N$ on the accuracy of the estimation when the scale separation parameter $\epl$ is small but finite. Then, following \cref{rem:inference_wavelength}, we show how the multiscale parameter $\epl$ can be inferred from our spectral estimator by considering a large number of Fourier modes. Finally, we demonstrate that the methodology extends naturally to multidimensional diffusion processes.

Throughout all experiments, synthetic trajectories are generated by numerically solving the multiscale SDE \eqref{eq:SDE_multiscale} with deterministic initial condition $X_0^\epl = 0$. The SDEs are discretized using the Euler--Maruyama scheme (see, e.g., \cite{KlP92}) with a fine time step $h = \epl^3$. The ground-truth invariant densities $\rho$ and $\rho^\epl$ are computed analytically from equation \eqref{eq:invariant_density}, where the normalization constants $Z$ and $Z^\epl$ are evaluated via numerical integration using the Python function \texttt{scipy.integrate.quad} from the SciPy library \cite{VGO20}. Finally, the physicist’s Hermite polynomials are computed using the Python function \texttt{scipy.special.hermite} from the same library.

\subsection{Density estimator}

\begin{figure}[th!]
\centering
\begin{tabular}{ccccc}
\includegraphics{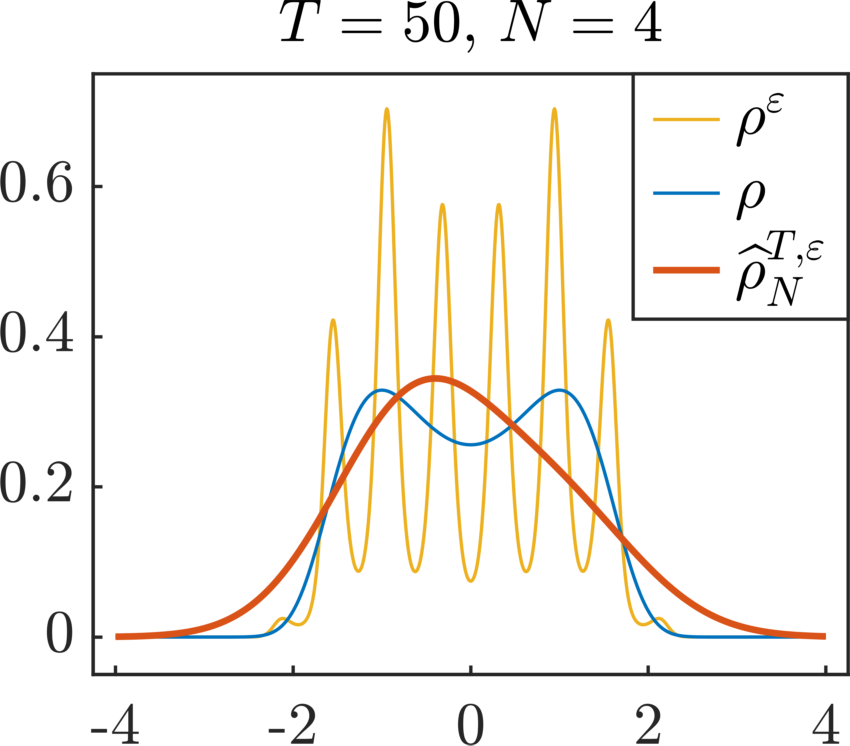} &&
\includegraphics{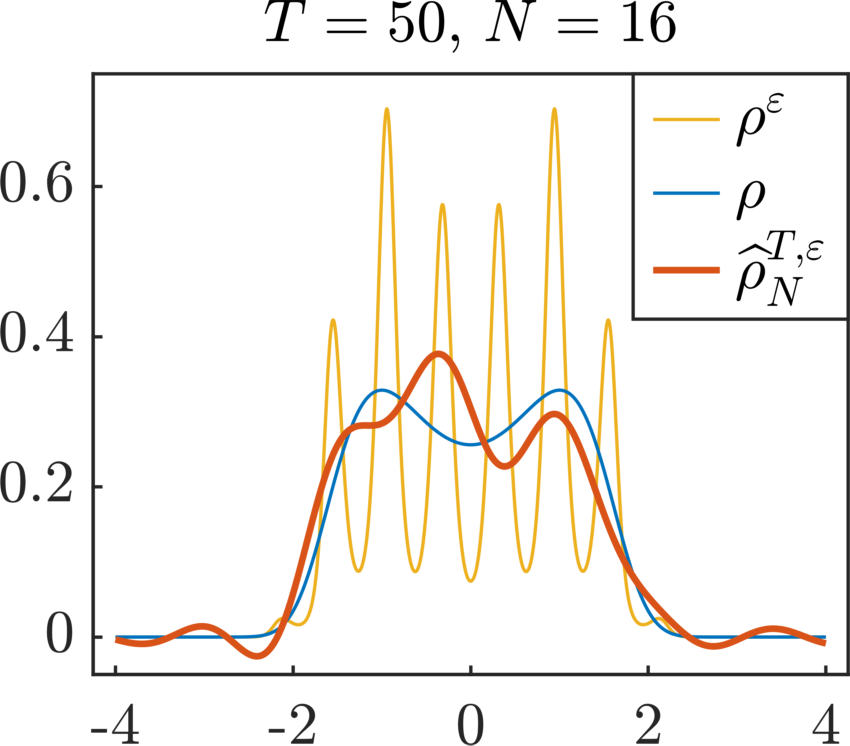} &&
\includegraphics{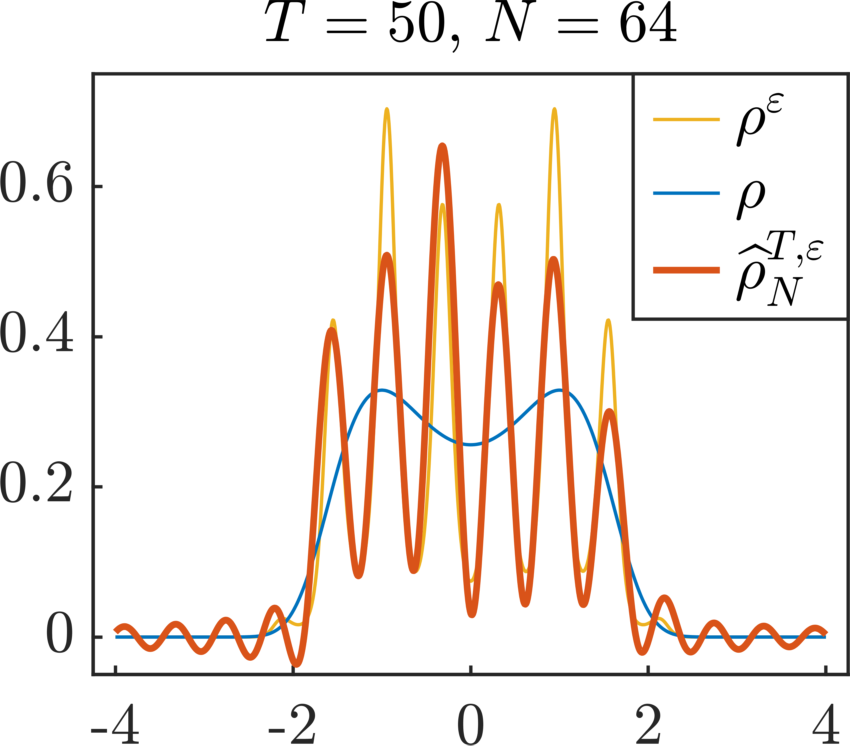} \\[0.25cm]
\includegraphics{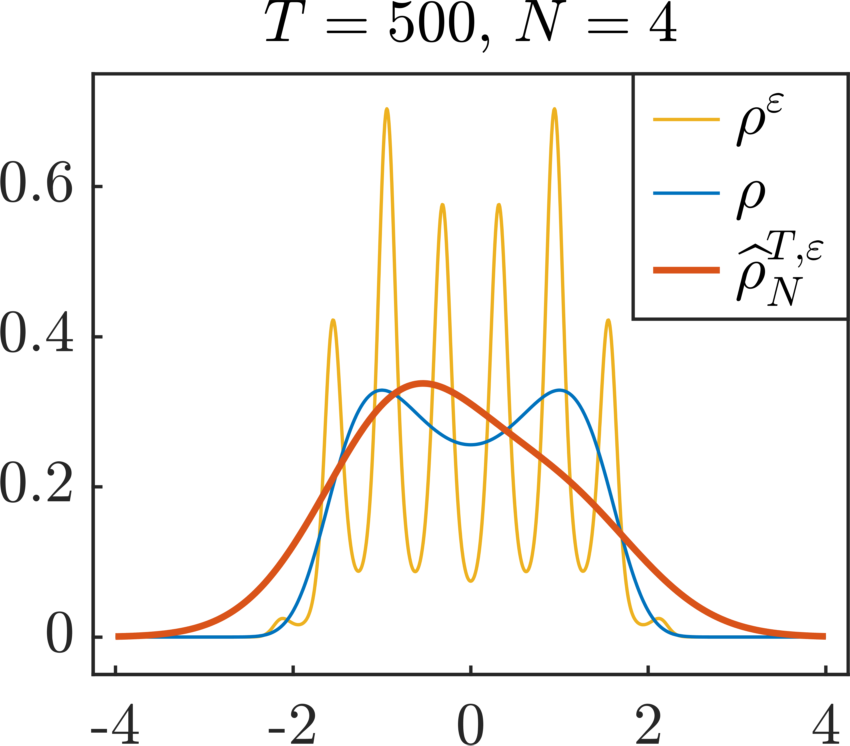} &&
\includegraphics{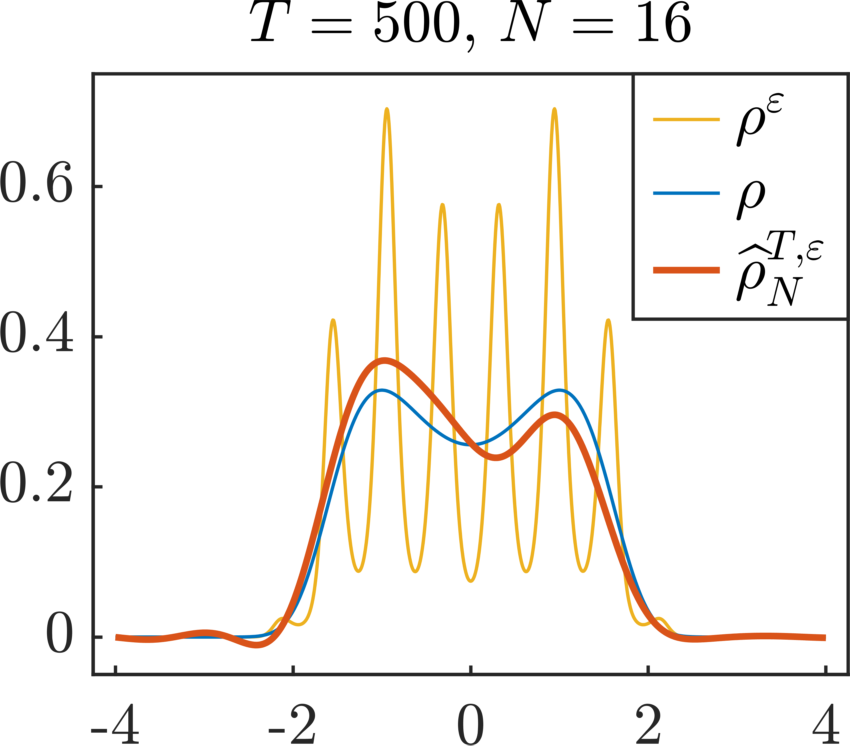} &&
\includegraphics{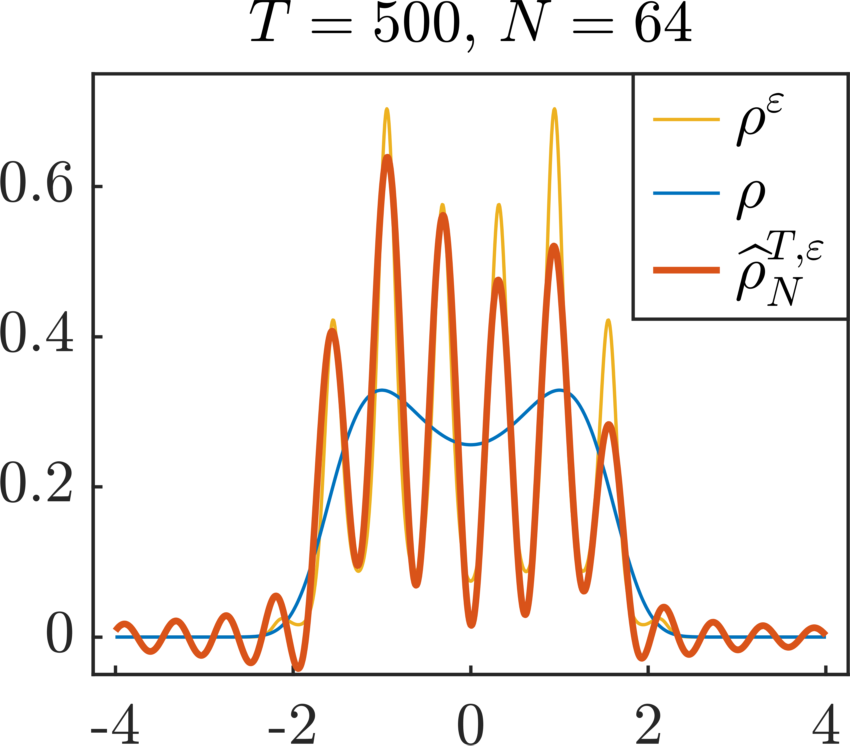} \\[0.25cm]
\includegraphics{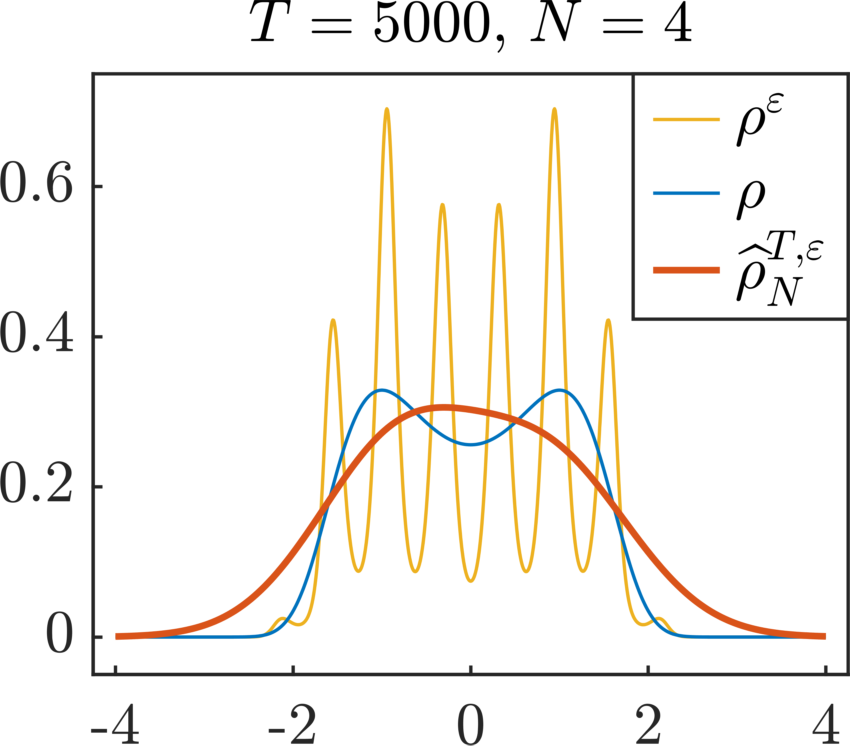} &&
\includegraphics{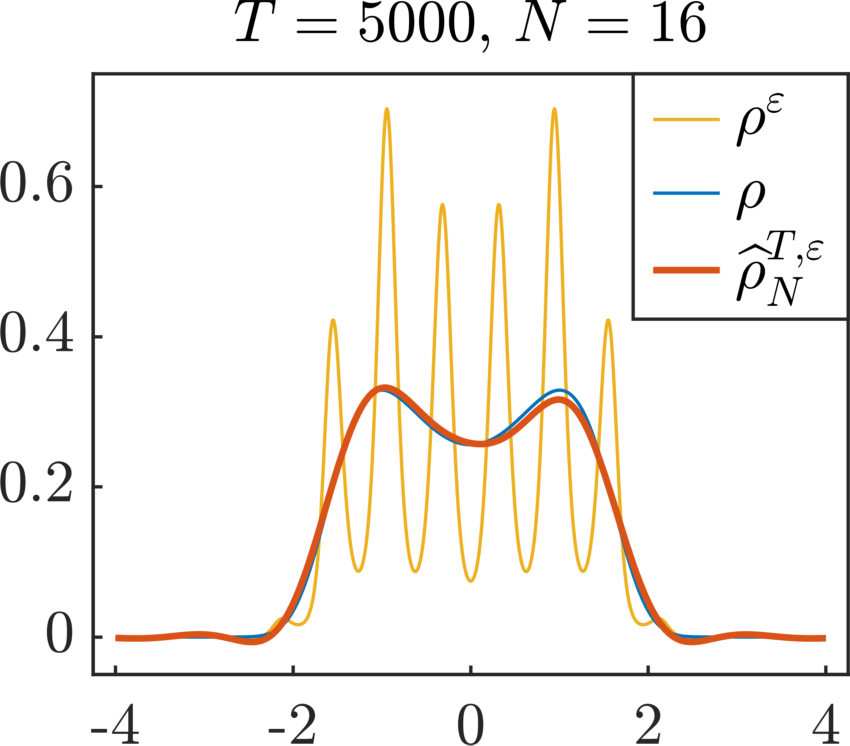} &&
\includegraphics{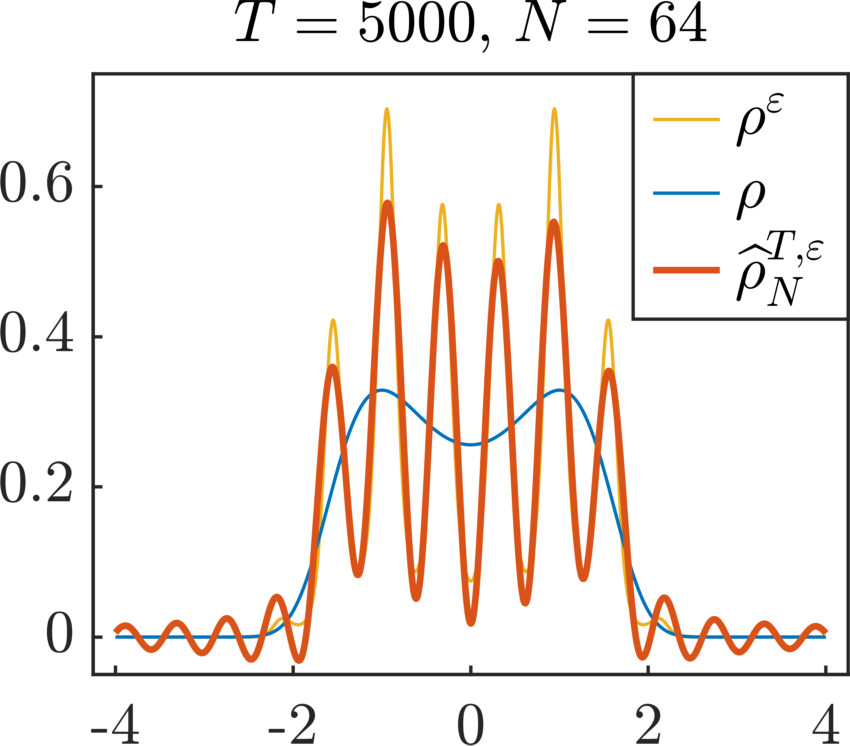}
\end{tabular}
\caption{Performance of the estimator $\widehat \rho_N^{T,\epl}$ across varying values of $T = 50, 500, 5000$ and $N = 4, 16, 64$ with fixed $\epl = 0.1$, for the double-well potential.}
\label{fig:varyTeN}
\end{figure}

We consider the multiscale dynamics \eqref{eq:SDE_multiscale} with a double-well slow-scale potential and a periodic fast-scale potential with period $\period = 2\pi$, given by
\begin{equation}
V(x) = \frac{x^4}4 - \frac{x^2}2, \qquad p(y) = \cos(y).
\end{equation}
We set the diffusion coefficient to $\sigma = 1$ and the scale separation parameter to $\epl = 0.1$. The density estimator $\widehat \rho_N^{T,\epl}$ is then computed for varying trajectory lengths $T \in \{50, 500, 5000\}$ and numbers of Fourier modes $N \in \{4, 16, 64\}$.

In \cref{fig:varyTeN}, we compare the resulting estimators with both the target homogenized density $\rho$ and the invariant measure of the multiscale dynamics $\rho^\epl$. We observe that when the number of Fourier coefficients $N$ is too small, the estimator fails to accurately approximate the invariant density. This is due to the truncation error in the spectral expansion, as characterized by the remainder term in \cref{lem:rate_alpha_n}. On the other hand, when $N$ is chosen appropriately, the estimator $\widehat \rho_N^{T,\epl}$ accurately captures the homogenized density $\rho$, provided the trajectory is sufficiently long. However, an excessive number of Fourier modes causes the estimator to break down, consistent with the theoretical limitations established in \cref{thm:convergence}. In particular, $\widehat \rho_N^{T,\epl}$ begins to capture the fine-scale oscillations of the multiscale dynamics instead of the coarse-grained behavior of the homogenized system, and this leads the estimator to approximate $\rho^\epl$ rather than $\rho$. Conversely, the estimation quality improves as the observation time $T$ increases.

\subsection{Inference of the scale separation parameter}

\begin{figure}[th!]
\centering
\begin{tabular}{ccccc}
\includegraphics{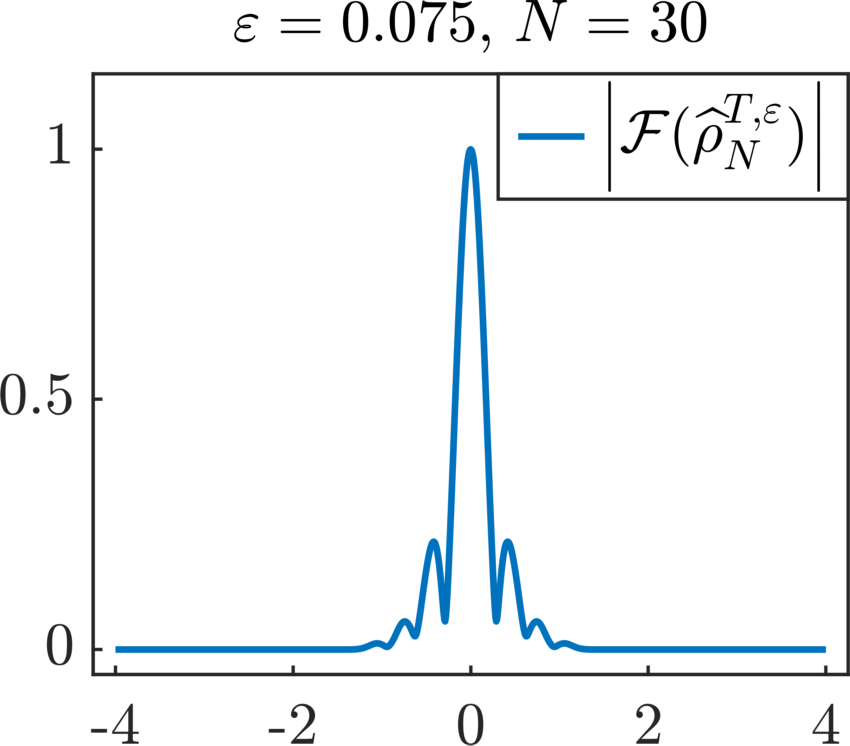} &&
\includegraphics{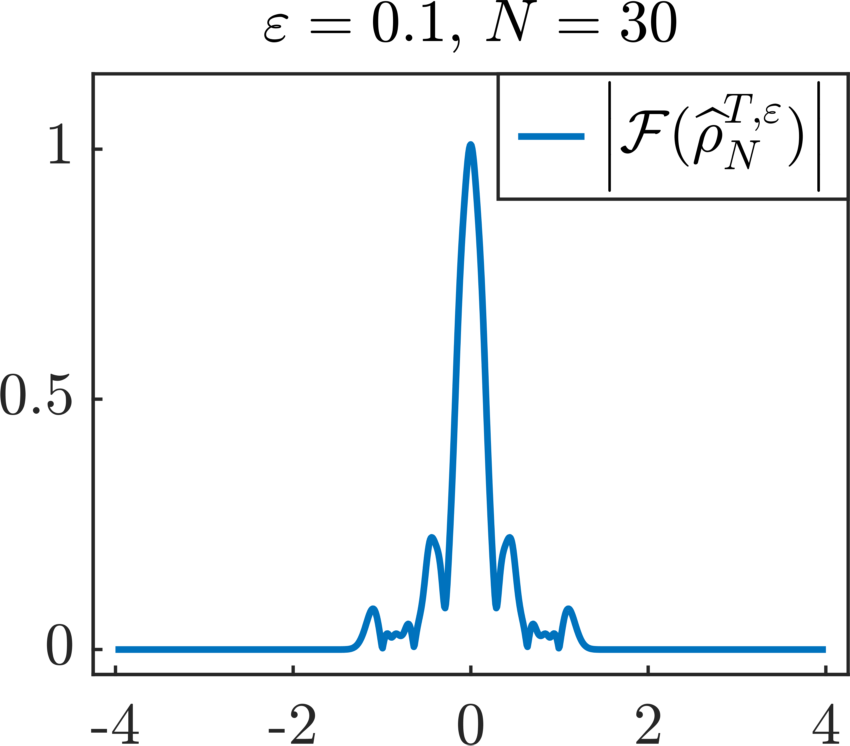} &&
\includegraphics{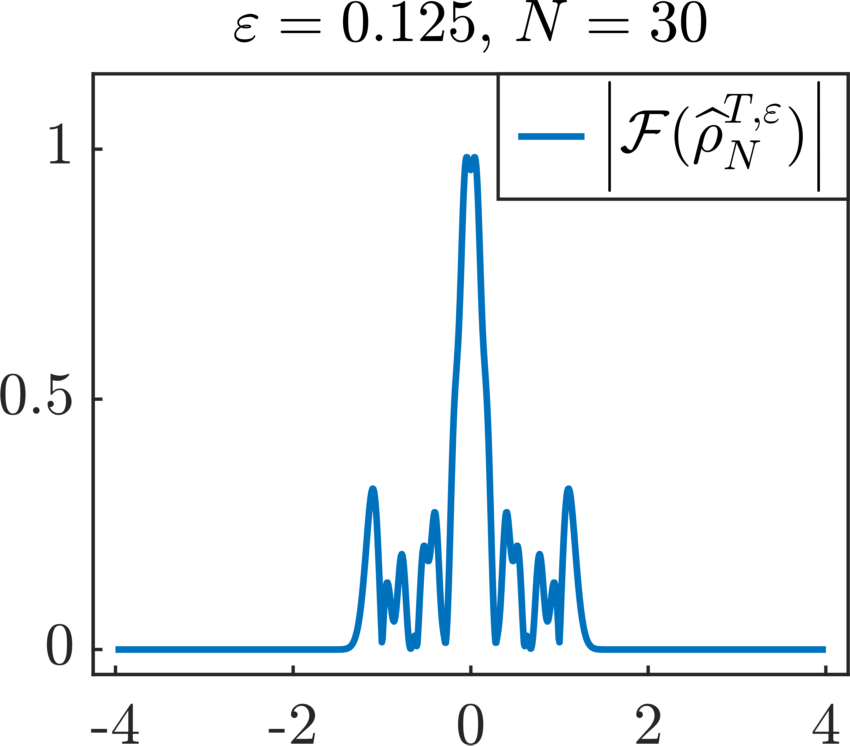} \\[0.25cm]
\includegraphics{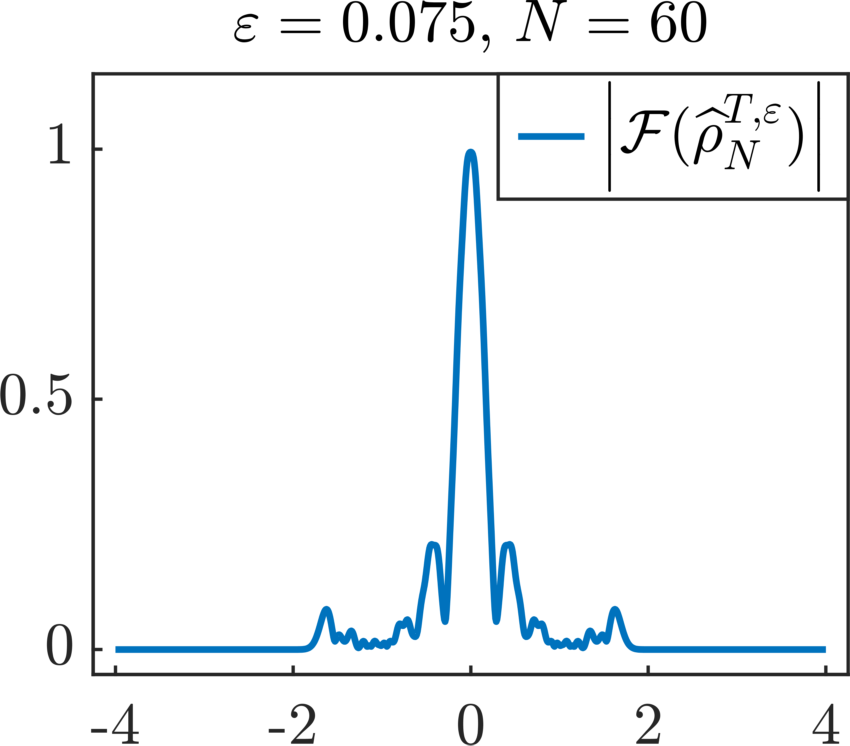} &&
\includegraphics{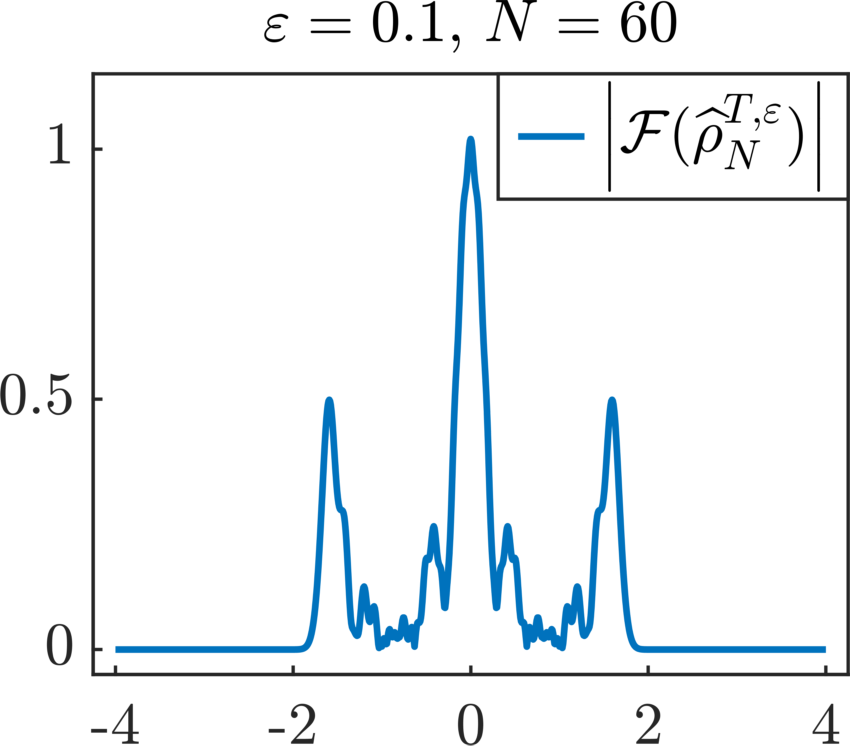} &&
\includegraphics{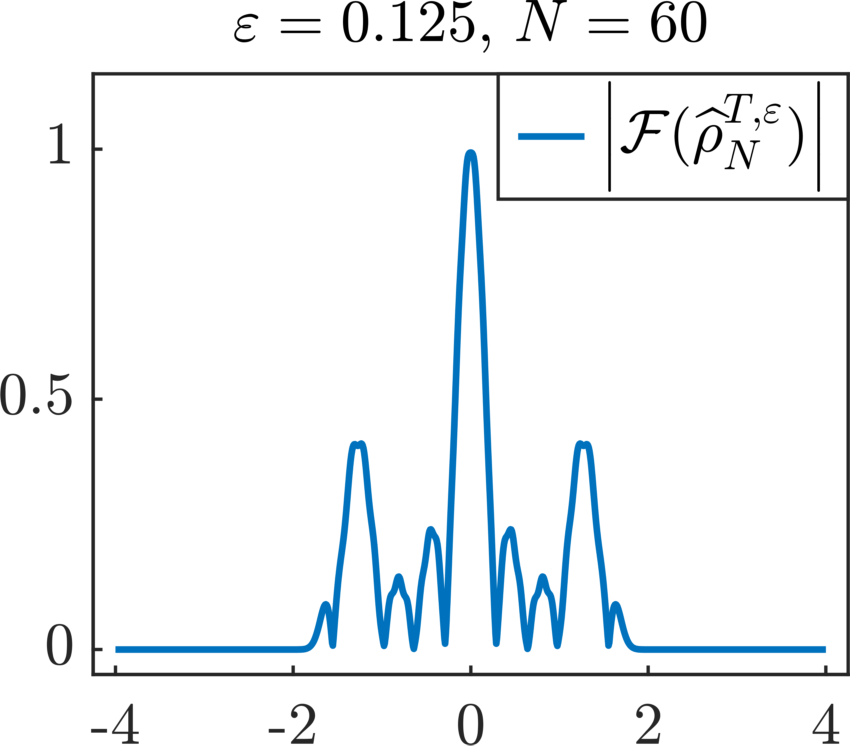} \\[0.25cm]
\includegraphics{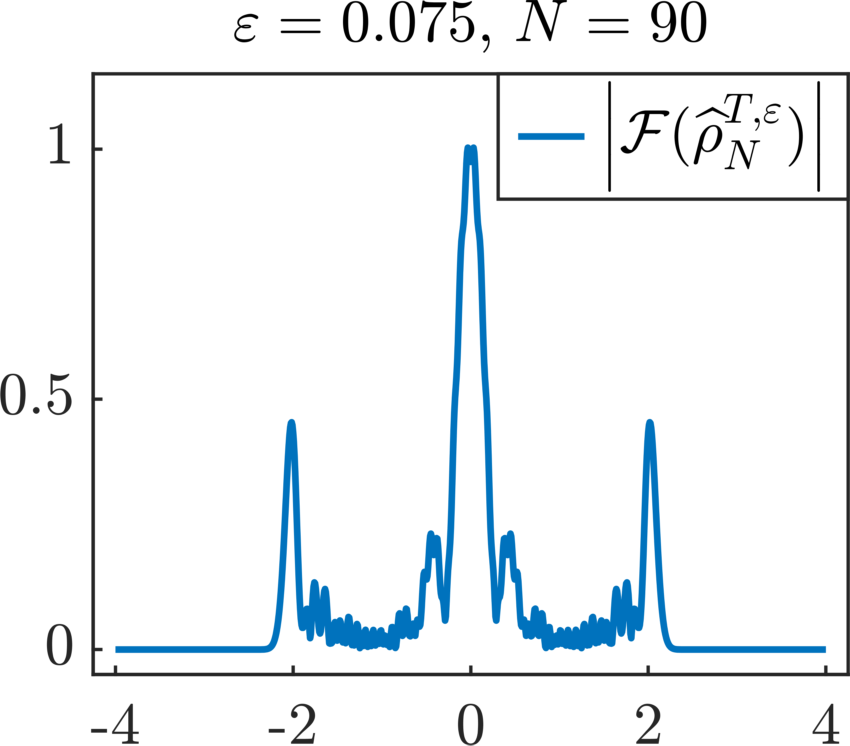} &&
\includegraphics{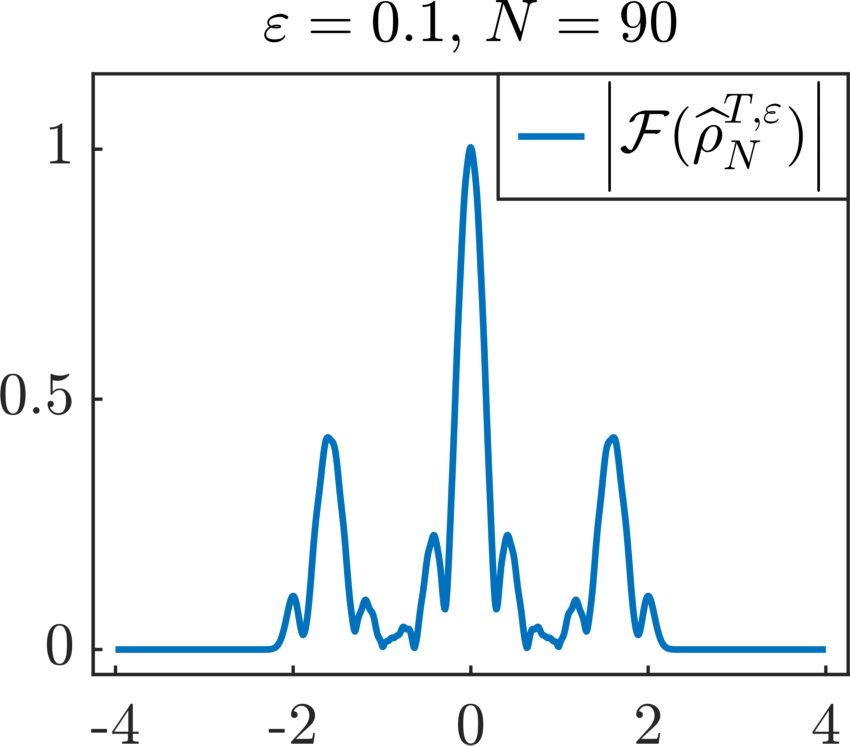} &&
\includegraphics{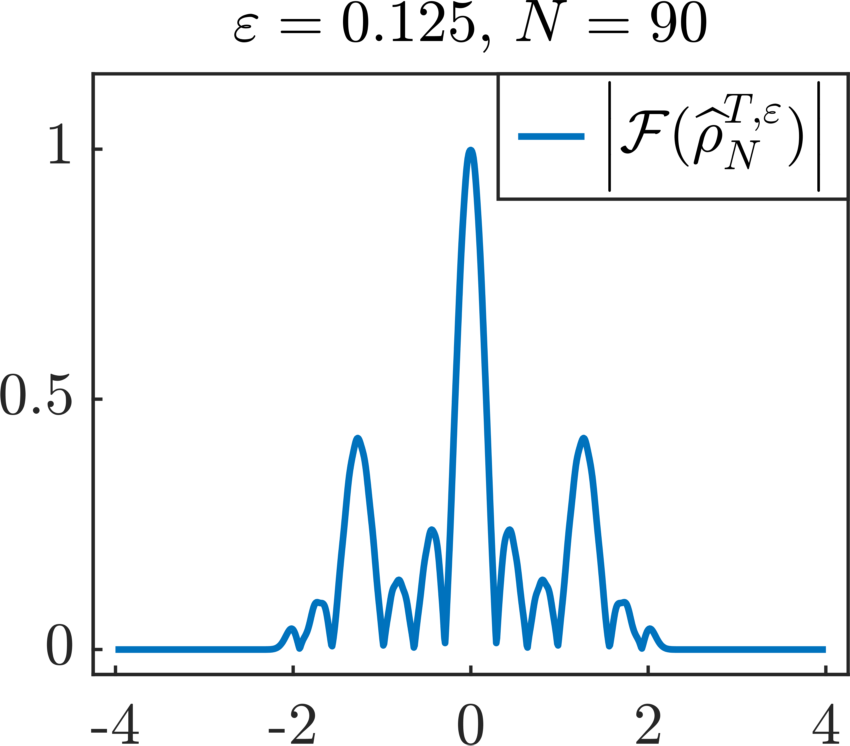}
\end{tabular}

\vspace{0.5cm}
\begin{tabular}{ccccc}
\toprule
$\epl$ && $0.075$ & $0.1$ & $0.125$ \\
\midrule
$\widehat \epl = 1/(\period \bar\xi)$ && $\sim 0.079$ & $\sim 0.099$ & $\sim 0.125$ \\
\bottomrule
\end{tabular}
\caption{Top: magnitude of the Fourier transform of the density estimator $\mathcal F(\widehat \rho_N^{T,\epl})$ across varying values of $\epl = 0.075, 0.1, 0.125$ and $N = 30, 60, 90$ with fixed $T = 1000$, for the double-well potential. Bottom: inference of the scale separation parameter $\epl$ from the dominant frequency $\bar\xi \neq 0$.}
\label{fig:varyTeN_Fourier}
\end{figure}

Let us consider the same setting as in the previous section. As noted above, appropriately selecting the number of Fourier modes is crucial for accurately approximating the homogenized invariant measure $\rho$. While \cref{thm:convergence} provides a quantitative guide for choosing $N$, its formula depends on the wavelength $\period \epl$. Therefore, to apply this result in practice, we must first estimate this quantity.

By analyzing the Fourier transform of our estimator, we aim to identify the dominant nonzero frequency component, which in turn allows us to determine the wavelength $\period\epl$. Following the procedure outlined in \cref{rem:inference_wavelength}, we calculate the magnitude of the Fourier transform of the estimator, $\abs{\mathcal F(\widehat \rho_N^{T,\epl})}$, for various values of $\epl \in \{ 0.075, 0.1, 0.125 \}$ and $N \in \{ 30, 60, 90 \}$. We observe that the dominant frequency $\bar\xi \neq 0$ emerges only when $N$ is sufficiently large. In particular, smaller values of $\epl$ require larger values of $N$ for this frequency to become visible.

From the final set of plots with $N = 90$ Fourier coefficients, we extract the dominant nonzero frequency $\bar\xi$ and, assuming the fast-scale period $\period = 2\pi$ is known, estimate the scale separation parameter via $\widehat\epl = 1/(\period \bar\xi)$. Comparing $\widehat\epl$ with the true value of $\epl$, we find strong agreement, demonstrating the effectiveness of this approach for inferring the wavelength $\period \epl$. This, in turn, provides a practical method for guiding the selection of $N$ in our spectral estimator.

\subsection{A two-dimensional example} \label{sec:num2D}

\begin{figure}[th!]
\centering
\begin{tabular}{ccc}
\includegraphics{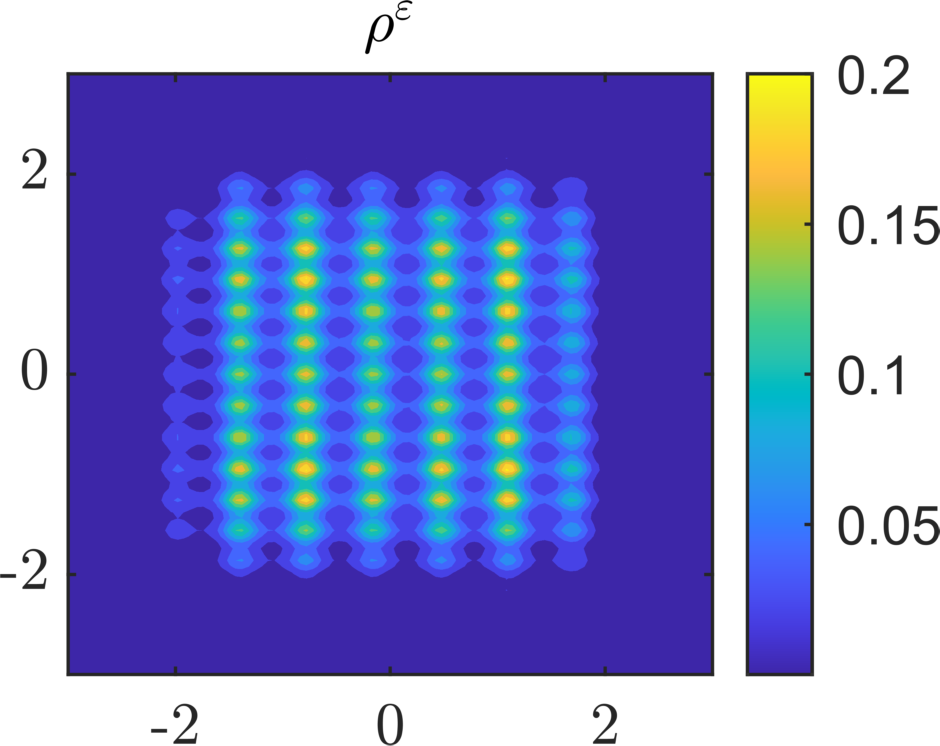} &
\includegraphics{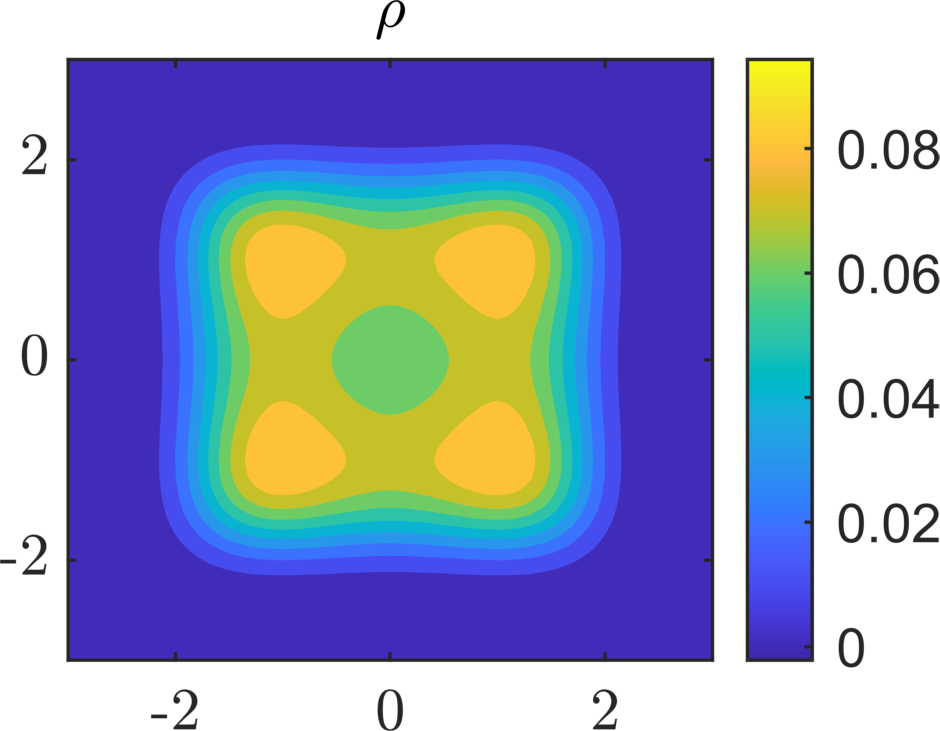} &
\includegraphics{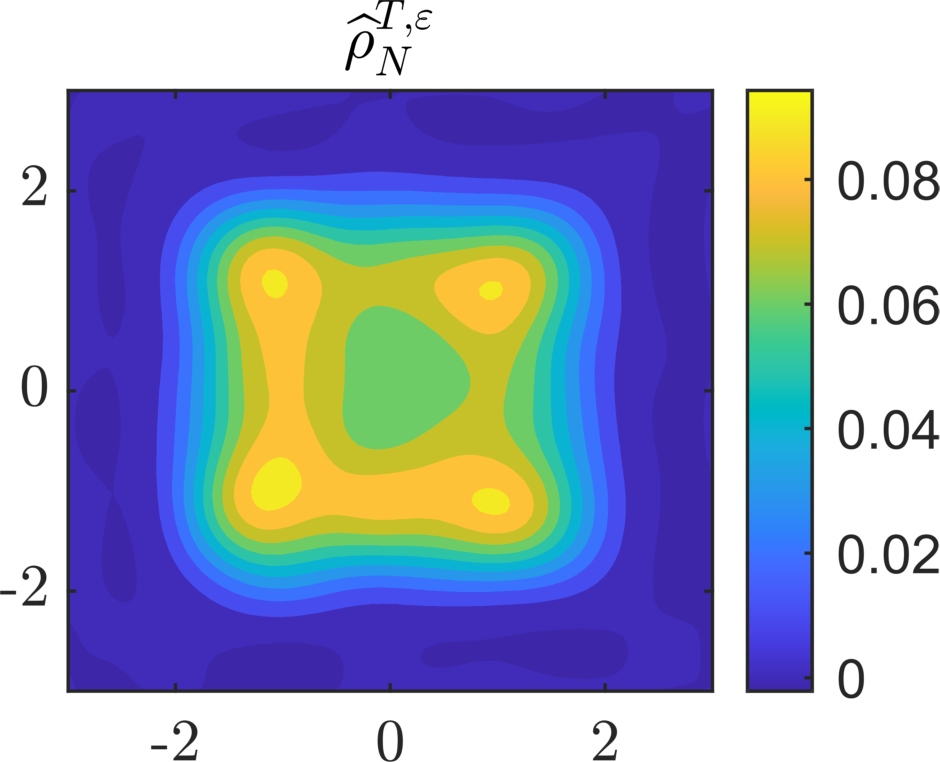} 
\end{tabular}
\caption{Performance of the estimator $\widehat \rho_N^{T,\epl}$ for the two-dimensional test case.}
\label{fig:example2D}
\end{figure}

In this section, we examine a two-dimensional test case to demonstrate that the methodology proposed in this work extends beyond the one-dimensional setting, as discussed in \cref{rem:multidimensional}. Here, the basis functions are constructed as tensor products of one-dimensional Hermite functions. Specifically, our estimator $\widehat \rho_N^{\epl,T}$ for $\mathbf x = (x_1, x_2) \in \R^2$ is given by
\begin{equation}
\widehat \rho_N^{T,\epl}(\mathbf{x}) = \sum_{m=0}^{N-1} \sum_{n=0}^{N-1} \widehat \alpha_{mn}^{T,\epl} \Psi_{mn}(\mathbf{x}),
\end{equation}
where the basis functions $\Psi_{mn}\colon\R^2\to \R$ are defined as
\begin{equation}
\Psi_{mn}(\mathbf{x}) = \psi_m(x_1) \psi_n(x_2),
\end{equation}
and the coefficients $\alpha_{mn}^{T,\epl}$ are computed as before
\begin{equation}
\widehat \alpha_{mn}^{T,\epl} = \frac{1}{T} \int_0^T \Psi_{mn}(\mathbf{X}_t^\epl) \dd t.
\end{equation}
We consider the following potential functions
\begin{equation}
V(\mathbf{x}) = \frac{x_1^4 + x_2^4}{4} - \frac{x_1^2 + x_2^2}{2}, \qquad p(\mathbf{y}) = \sin(y_1) + \sin^2(y_2),
\end{equation}
and set the diffusion coefficient and the multiscale parameter to $\sigma = 1.5$ and $\epl = 0.1$, respectively. A trajectory of the multiscale SDE is simulated over a time interval of length $T = 2000$, and the estimator $\widehat \rho_N^{\epl,T}$ is computed using $N = 16$ Fourier modes in each dimension. In \cref{fig:example2D}, we compare the estimated density with the invariant densities of both the homogenized and multiscale dynamics, and observe that $\widehat \rho_N^{\epl,T}$ provides a reasonably accurate approximation of $\rho$ by capturing its main features.

\section{Convergence analysis} \label{sec:analysis}

This section is devoted to the proof of \cref{thm:convergence}. We recall that our goal is to show that
\begin{equation}
\E \left[ \norm{\widehat \rho_N^{T,\epl} - \rho}_{L^2(\R)}^2 \right] \to 0,
\end{equation}
under suitable conditions on $N$, $T$, and $\epl$. Defining
\begin{equation} \label{eq:rho_N_e_def}
\rho_N^\epl(x) \defeq \sum_{n=0}^{N-1} \alpha_n^\epl \psi_n(x), \qquad \text{with} \qquad \alpha_n^\epl \defeq \int_\R \psi_n(x) \rho^\epl(x) \dd x,
\end{equation}
and using the triangle inequality, we first find that
\begin{equation} \label{eq:triangle}
\E \left[ \norm{\widehat \rho_N^{T,\epl} - \rho}_{L^2(\R)}^2 \right] \le 2 \left( \mathfrak q_N^\epl + \mathcal Q_N^{T,\epl} \right),
\end{equation}
where 
\begin{equation} \label{eq:deterministic_stochastic_terms}
\mathfrak q_N^\epl = \norm{\rho_N^\epl - \rho}_{L^2(\R)}^2 \qquad \text{and} \qquad \mathcal Q_N^{T,\epl} = \E \left[ \norm{\widehat \rho_N^{T,\epl} - \rho_N^\epl}_{L^2(\R)}^2 \right].
\end{equation}
The term $\mathfrak q_N^\epl$ represents the deterministic error due to truncating the Fourier expansion and using multiscale coefficients, while $\mathcal Q_N^{T,\epl}$ captures the stochastic error arising from the finite-time, data-driven approximation of the coefficients. In the following sections, we analyze each term separately and then combine the results to eventually establish \cref{thm:convergence}.

\subsection{Analysis of the deterministic component $\mathfrak q_N^\epl$}

The convergence analysis of the first term $\mathfrak q_N^\epl$ is carried out in three main steps. We begin by establishing technical results concerning Hermite functions. Next, we demonstrate convergence in the Gaussian setting, where the potential $V$ is quadratic. Finally, we extend the analysis to more general potentials by leveraging approximation theory based on mixtures of Gaussians.

\subsubsection{Properties of Hermite polynomials and Hermite functions}

In the following technical result, we derive an expression for the Fourier coefficients of the product of Hermite functions and Gaussian densities, formulated in terms of Hermite polynomials.

\begin{lemma} \label{lem:bound_Fourier_psi}
For all $n\in\N\cup\{0\}$, it holds that
\begin{equation}
\int_\R \psi_n(x) e^{-\frac{(x-\mu)^2}{2\sigma^2} \pm i \frac{2\pi}{\period} k \frac{x}\epl} \dd x = \frac{\pi^{1/4}}{\sqrt{c\cdot 2^{n-1} n!}}e^{-\frac{\mu^2}{2\sigma^2} + \frac{1}{2c}\left(\frac{\mu}{\sigma^2}\mp\frac{2\pi ki}{\period\epl}\right)^2} \widetilde{H}_n\left(\mp \frac{2\pi k}{\period\epl}; \sigma^2\right),
\end{equation}
where $c=\frac{\sigma^2+1}{\sigma^2}$, and
\begin{equation}
\widetilde{H}_n(x; \sigma^2) 
= \begin{cases}
(-1)^n \left(1 - \frac2c\right)^{n/2}H_n\left(\frac{ix-\mu/\sigma^2}{c\sqrt{1-2/c}}\right) &\text{if } \sigma^2 < 1, \\
(\mu-ix)^n &\text{if } \sigma^2 = 1, \\
(-i)^n\left(\frac{2}{c}-1\right)^{n/2}H_n\left(\frac{x+i\mu/\sigma^2}{c\sqrt{2/c-1}}\right) &\text{if } \sigma^2 > 1.
\end{cases}
\end{equation}
\end{lemma}
\begin{proof}
First, using the generating function for $H_n$ \cite{Sto27}, we have
\begin{equation}
e^{-x^2/2 - (x-\mu)^2/(2\sigma^2)+2xs-s^2} = \sum_{n=0}^{\infty} e^{-x^2/2 - (x-\mu)^2/(2\sigma^2)}H_n(x)\frac{s^n}{n!}.
\end{equation}
The Fourier transform of the right hand side is
\begin{equation}
\frac{1}{\sqrt{2\pi}}\sum_{n=0}^\infty \frac{s^n}{n!} \int_\R e^{-x^2/2} H_n(x) e^{-\frac{(x - \mu)^2}{2\sigma^2} -ixk} \dd x = \frac{\sqrt{\sqrt{\pi}2^n n!}}{\sqrt{2\pi}}\sum_{n=0}^\infty \frac{s^n}{n!}\int_\R \psi_n(x) e^{-\frac{(x - \mu)^2}{2\sigma^2} - ixk} \dd x.
\end{equation}
Then taking the Fourier transform of the left hand side, we have
\begin{equation}
\frac{1}{\sqrt{2\pi}}\int_\R e^{-ixk}\cdot e^{-\frac{x^2}{2}\left(1+\frac{1}{\sigma^2}\right) + \frac{x\mu}{\sigma^2} - \frac{\mu^2}{2\sigma^2}+2xs-s^2} \dd x = \frac{1}{\sqrt{c}}e^{-\frac{\mu^2}{2\sigma^2}-s^2+\frac{c}{2}\left(\frac{\mu/\sigma^2-ik+2s}{c}\right)^2},
\end{equation}
where $c = \frac{\sigma^2+1}{\sigma^2}$. When $c < 2$, the $n$-th derivative of this with respect to $s$ is
\begin{equation}
\begin{split}
&\frac{\partial^n}{\partial s^n}\frac{1}{\sqrt{c}}e^{-\frac{\mu^2}{2\sigma^2}-s^2+\frac{c}{2}\left(\frac{\mu/\sigma^2-ik+2s}{c}\right)^2}\\ 
&\qquad = \frac{1}{\sqrt{c}}e^{-\frac{\mu^2}{2\sigma^2} + \frac{1}{2c}\left(\frac{\mu}{\sigma^2}-ik\right)^2}(-i)^n\left(\frac{2}{c}-1\right)^{n/2}e^{\left(\frac{2}{c}-1\right)s^2 + \frac{2}{c}\left(\frac{\mu}{\sigma^2}-ik\right)s} H_n\left(\frac{k - i(c-2)s+i\mu/\sigma^2}{c\sqrt{2/c-1}}\right).
\end{split}
\end{equation}
Evaluating at $s=0$ gives
\begin{equation}
\left.\frac{\partial^n}{\partial s^n}\frac{1}{\sqrt{c}}e^{-\frac{\mu^2}{2\sigma^2}-s^2+\frac{c}{2}\left(\frac{\mu/\sigma^2-ik+2s}{c}\right)^2}\right|_{s=0} = \frac{1}{\sqrt{c}}e^{-\frac{\mu^2}{2\sigma^2} + \frac{1}{2c}\left(\frac{\mu}{\sigma^2}-ik\right)^2}(-i)^n\left(\frac{2}{c}-1\right)^{n/2}H_n\left(\frac{k+i\mu/\sigma^2}{c\sqrt{2/c-1}}\right).
\end{equation}
It follows that when $c<2$, the Fourier transform of the left hand side has a Taylor expansion
\begin{equation}
\begin{split}
&\frac{1}{\sqrt{2\pi}}\int_\R e^{-ixk}\cdot e^{-\frac{x^2}{2}\left(1+\frac{1}{\sigma^2}\right) + \frac{x\mu}{\sigma^2} - \frac{\mu^2}{2\sigma^2}+2xs-s^2} \dd x\\ 
&\qquad = \frac{1}{\sqrt{c}}e^{-\frac{\mu^2}{2\sigma^2} + \frac{1}{2c}\left(\frac{\mu}{\sigma^2}-ik\right)^2}\sum_{n=0}^\infty (-i)^n\left(\frac{2}{c}-1\right)^{n/2}H_n\left(\frac{k+i\mu/\sigma^2}{c\sqrt{2/c-1}}\right)\frac{s^n}{n!}.
\end{split}
\end{equation}
Thus, for $c<2$, equating coefficients of $s^n$ in the series expansions of the Fourier transforms, we have 
\begin{equation}
\int_\R \psi_n(x) e^{-\frac{(x-\mu)^2}{2\sigma^2}-ixk} \dd x = \frac{\pi^{1/4}\sqrt{2}}{\sqrt{c\cdot 2^n n!}} e^{-\frac{\mu^2}{2\sigma^2} + \frac{1}{2c}\left(\frac{\mu}{\sigma^2}-ik\right)^2} (-i)^n\left(\frac{2}{c}-1\right)^{n/2}H_n\left(\frac{k+i\mu/\sigma^2}{c\sqrt{2/c-1}}\right).
\end{equation}
The proof for $c > 2$ is similar. Now when $c=2$, we have that the Fourier transform of the left hand side simplifies to
\begin{equation}
\frac{1}{\sqrt{2}}e^{-\frac{\mu^2}{2}-s^2+\frac{1}{4}\left(\mu-ik+2s\right)^2} = \frac{1}{\sqrt{2}} e^{-\frac{\mu^2}{2}+\frac{1}{4}(\mu-ik)^2}\sum_{n=0}^\infty \frac{(\mu-ik)^ns^n}{n!}.
\end{equation}
Again equating coefficients of Taylor expansions yields
\begin{equation}
\int_\R \psi_n(x) e^{-\frac{x^2}{2\sigma^2}-ixk} \dd x = \frac{\pi^{1/4}}{\sqrt{2^{n}n!}}e^{-\frac{\mu^2}{2}+\frac{1}{4}(\mu-ik)^2}(\mu-ik)^n
\end{equation}
when $c=2$, that is, $\sigma^2 = 1$.
\end{proof}

\begin{remark}
Note that $\widetilde{H}_n$ is continuous in $\sigma^2$ at $\sigma^2=1$. Indeed, around $\sigma^2=1$, we have
\begin{equation}
\begin{split}
(-i)^n\left(\frac{2}{c}-1\right)^{n/2}H_n\left(\frac{x+i\mu/\sigma^2}{c\sqrt{2/c-1}}\right)
&\sim (-i)^{n}\left(\frac{2}{c}-1\right)^{n/2}\cdot \frac{2^n (x+i\mu/\sigma^2)^n}{(c\sqrt{2/c-1})^n}\\ 
&= \frac{(-i)^n2^n(x+i\mu/\sigma^2)^n}{c^n} \to (\mu-ix)^n
\end{split}
\end{equation}
as $\sigma^2\to1^+$, or equivalently, as $c = \frac{\sigma^2+1}{\sigma^2}\to 2^-$. When $\sigma^2 \to 1^-$, the calculation follows in a similar way.
\end{remark}

The next result provides an upper bound for the Hermite polynomials.

\begin{lemma}\label{lem:hermite-bound}
For all $n\in\N\cup\{0\}$, we have the following bound on the Hermite polynomials
\begin{equation}
\abs{H_n(x)} \le \begin{cases}
n^{n/2}(4\sqrt{2})^n\left(1+\frac{n}{2}\right) &\text{if } \abs{x} \le \sqrt{2n},\\
2^{2n}\abs{x}^n\left(1 + \frac{n}{2}\right) &\text{if } \abs{x} > \sqrt{2n},
\end{cases}
\end{equation}
where we define $0^0 := 1$. In particular, it holds
\begin{equation}
\abs{H_n(x)} \le 4^n\left(1 + \frac{n}{2}\right)\left(2^{n/2}n^{n/2} + \abs{x}^n\right).
\end{equation}
\end{lemma}
\begin{proof}
Recall from \cite[Chapter V]{Sze75} that 
\begin{equation}
H_n(x) = \sum_{m=0}^{\lfloor n/2\rfloor} \frac{(-1)^m n!}{m!(n-2m)!} (2x)^{n-2m} = \sum_{m=0}^{\lfloor n/2\rfloor} \frac{(-1)^m n!}{m!(n-m)!}\frac{(n-m)!}{(n-2m)!}(2x)^{n-2m}.
\end{equation}
If $\abs{x} \le \sqrt{2n}$, then we have
\begin{equation}
\abs{H_n(x)} \le \sum_{m=0}^{\lfloor n/2\rfloor} \binom{n}{m} n^m(2\sqrt{2n})^{n-2m} = n^{n/2}(2\sqrt{2})^n \sum_{m=0}^{\lfloor n/2\rfloor} \binom{n}{m} (2\sqrt{2})^{-2m} \le n^{n/2}(4\sqrt{2})^n\left(1+\frac{n}{2}\right),
\end{equation}
where the last inequality uses $\binom{n}{m} \le 2^n$. Note that we use $0^0 = 1$ when we bound $\frac{(n-m)!}{(n-2m)!} \le n^m$ for $n=0$. Now similarly, if $\abs{x} > \sqrt{2n}$, we have
\begin{equation}
\abs{H_n(x)} \le 2^n\abs{x}^n \sum_{m=0}^{\lfloor n/2\rfloor} \frac{n!}{m!(n-2m)!}\frac{1}{(2x)^{2m}} \le 2^n\abs{x}^n \sum_{m=0}^{\lfloor n/2\rfloor} \binom{n}{m} \frac{n^m}{(8n)^m} \le 2^{2n}\abs{x}^n\left(1 + \frac{n}{2}\right),
\end{equation}
as desired. The final bound in the statement of the lemma is obtained by adding the bounds for $\abs{x}\le \sqrt{2n}$ and $\abs{x} > \sqrt{2n}$.
\end{proof}

Combining the previous lemmas, we obtain a bound for the sum of the Fourier coefficients of the product of Hermite functions and Gaussian densities.

\begin{lemma}\label{lem:psi-fourier-series-bound}
Let $n\in\N\cup\{0\}$ with $n < \frac{2\pi^2}{\period^2\epl^2 c}$, where $c = \frac{\sigma^2+1}{\sigma^2}$. If $\sigma^2\ne 1$, it holds that
\begin{equation}
\begin{split}
\sum_{k=1}^\infty \left|\int_\R \psi_n(x)e^{-\frac{(x-\mu)^2}{2\sigma^2}\pm i\frac{2\pi}{\period}k\frac{x}{\epl}}\dd x\right| &\le C(n,\sigma^2)e^{-\frac{2\pi^2}{\period^2\epl^2 c}}\left[\left(2^{n/2}n^{n/2}\abs{\frac{2}{c}-1}^{n/2} + 2^{n-1}\left(\frac{\abs{\mu}}{c\sigma^2}\right)^n\right) \right. \\
&\hspace{3.1cm}\times \left. \left(1 + \frac{\period^2\epl^2 c}{4\pi^2}\right) + 2^{n-1}\left(\frac{2\pi}{\period\epl c}\right)^n\left(1 + \frac12 e^{n-1}\right)\right],
\end{split}
\end{equation}
where $C(n,\sigma^2) = \frac{\pi^{1/4}4^n\left(1 + \frac{n}{2}\right)e^{-\frac{\mu^2}{2\sigma^2}+\frac{\mu^2}{2c\sigma^4}}}{\sqrt{c\cdot 2^{n-1}n!}}$. Otherwise, if $\sigma^2=1$, it holds that
\begin{equation}
\begin{split}
\sum_{k=1}^\infty \left|\int_\R \psi_n(x)e^{-\frac{(x-\mu)^2}{2\sigma^2}\pm i\frac{2\pi}{\period}k\frac{x}{\epl}}\dd x\right|
&\le \frac{\pi^{1/4}2^{n-1}e^{-\frac{\mu^2}{4}}}{\sqrt{2^n n!}} e^{-\frac{\pi^2}{\period^2\epl^2}}\left(\abs{\mu}^n\left(1 + \frac{\period^2\epl^2}{2\pi^2}\right) \right. \\
&\hspace{4.5cm} \left. + \left(\frac{2\pi}{\period\epl}\right)^n\left(1 + \frac12 e^{n-1} \right)\right).
\end{split}
\end{equation}
\end{lemma}
\begin{proof}
From \cref{lem:bound_Fourier_psi}, we have
\begin{equation}
\int_\R \psi_n(x) e^{-\frac{(x-\mu)^2}{2\sigma^2} \pm i \frac{2\pi}{\period} k \frac{x}\epl} \dd x = \frac{\pi^{1/4}}{\sqrt{c\cdot 2^{n-1} n!}}e^{-\frac{\mu^2}{2\sigma^2} + \frac{1}{2c}\left(\frac{\mu}{\sigma^2}\mp\frac{2\pi ki}{\period\epl}\right)^2} \widetilde{H}_n\left(\mp \frac{2\pi k}{\period\epl}; \sigma^2\right).
\end{equation}
We consider the two cases $\sigma^2\ne 1$ and $\sigma^2 = 1$ separately. \\
\textbf{Case $\sigma^2\ne 1$.} By \cref{lem:hermite-bound}, we have
\begin{equation}
\abs{H_n(x)} \le 4^n\left(1 + \frac{n}{2}\right)\left(2^{n/2}n^{n/2} + \abs{x}^n\right).
\end{equation}
Thus, 
\begin{equation}
\abs{\widetilde{H}_n(\mp 2\pi k/(\period\epl); \sigma^2)} \le \abs{\frac{2}{c}-1}^{n/2}4^n\left(1 + \frac{n}{2}\right)\left(2^{n/2}n^{n/2} + \left(\frac{\frac{2\pi k}{\period\epl} + \abs{\mu}/\sigma^2}{c\sqrt{\abs{2/c-1}}}\right)^n\right),
\end{equation}
and
\begin{equation}
\abs{e^{-\frac{\mu}{2\sigma^2}+\frac{1}{2c}\left(\frac{\mu}{\sigma^2}\mp \frac{2\pi ki}{\period\epl}\right)^2}} = e^{-\frac{\mu^2}{2\sigma^2}+\frac{\mu^2}{2c\sigma^4}-\frac{2\pi^2 k^2}{\period^2\epl^2 c}}.
\end{equation}
It therefore follows that
\begin{equation}
\begin{split}
&\abs{\int_\R \psi_n(x)e^{-\frac{(x-\mu)^2}{2\sigma^2}\pm i\frac{2\pi}{\period}k\frac{x}{\epl}}\dd x}\\
&\hspace{2cm}\le \frac{\pi^{1/4}e^{-\frac{\mu^2}{2\sigma^2}+\frac{\mu^2}{2c\sigma^4}}}{\sqrt{c\cdot 2^{n-1}n!}}e^{-\frac{2\pi^2 k^2}{\period^2\epl^2 c}}\abs{\frac{2}{c}-1}^{n/2}4^n\left(1 + \frac{n}{2}\right)\left(2^{n/2}n^{n/2} + \left(\frac{\frac{2\pi k}{\period\epl} + \abs{\mu}/\sigma^2}{c\sqrt{\abs{2/c-1}}}\right)^n\right)\\
&\hspace{2cm}\le \frac{\pi^{1/4}4^n\left(1 + \frac{n}{2}\right)e^{-\frac{\mu^2}{2\sigma^2}+\frac{\mu^2}{2c\sigma^4}}}{\sqrt{c\cdot 2^{n-1}n!}}\left(2^{n/2}n^{n/2}\abs{\frac{2}{c}-1}^{n/2}e^{-\frac{2\pi^2k^2}{\period^2\epl^2 c}} \right. \\
&\hspace{7.25cm}\left. + 2^{n-1}\left(\frac{2\pi k}{\period\epl c}\right)^ne^{-\frac{2\pi^2k^2}{\period^2\epl^2 c}} + 2^{n-1}\left(\frac{\abs{\mu}}{c\sigma^2}\right)^ne^{-\frac{2\pi^2k^2}{\period^2\epl^2 c}}\right),
\end{split}
\end{equation}
where in the last line we used $(a+b)^n \le 2^{n-1}(a^n+b^n)$. Summing over $k$ gives
\begin{equation}
\begin{split}
&\sum_{k=1}^\infty \abs{\int_\R \psi_n(x)e^{-\frac{(x-\mu)^2}{2\sigma^2}\pm i\frac{2\pi}{\period}k\frac{x}{\epl}}\dd x}\\
&\quad \le C(n,\sigma^2) \left[ \left(2^{n/2}n^{n/2}\abs{\frac{2}{c}-1}^{n/2} + 2^{n-1}\left(\frac{\abs{\mu}}{c\sigma^2}\right)^n\right)\sum_{k=1}^\infty e^{-\frac{2\pi^2 k^2}{\period^2\epl^2 c}} + 2^{n-1}\left(\frac{2\pi}{\period\epl c}\right)^n\sum_{k=1}^\infty k^n e^{-\frac{2\pi^2 k^2}{\period^2\epl^2 c}} \right],
\end{split}
\end{equation}
where $C(n,\sigma^2) = \frac{\pi^{1/4}4^n\left(1 + \frac{n}{2}\right)e^{-\frac{\mu^2}{2\sigma^2}+\frac{\mu^2}{2c\sigma^4}}}{\sqrt{c\cdot 2^{n-1}n!}}$. First consider the first sum. Because $e^{-\frac{2\pi^2 k^2}{\period^2\epl^2 c}}$ is decreasing in $k$, we have
\begin{equation}
\sum_{k=1}^\infty e^{-\frac{2\pi^2 k^2}{\period^2\epl^2 c}} \le e^{-\frac{2\pi^2}{\period^2\epl^2 c}} + \int_1^\infty e^{-\frac{2\pi^2 y^2}{\period^2\epl^2 c}}\dd y \le e^{-\frac{2\pi^2}{\period^2\epl^2 c}} + \int_1^\infty ye^{-\frac{2\pi^2 y^2}{\period^2\epl^2 c}}\dd y = e^{-\frac{2\pi^2}{\period^2\epl^2 c}} + \frac{\period^2\epl^2c}{4\pi^2}e^{-\frac{2\pi^2}{\period^2\epl^2 c}}.
\end{equation}
Now consider the second sum, $\sum_{k=1}^\infty k^n e^{-\frac{2\pi^2 k^2}{\period^2\epl^2 c}}$. Note that $y\mapsto y^n e^{-Ay^2}$ is maximized at $y=\sqrt{\frac{n}{2A}}$ and is decreasing for $y > \sqrt{\frac{n}{2A}}$. Defining $A = \frac{2\pi^2}{\period^2\epl^2c}$, we have that $\sqrt{\frac{n}{2A}}<1$ holds by our assumption. Combining these facts, we obtain
\begin{equation}
\sum_{k=1}^\infty k^n e^{-\frac{2\pi^2 k^2}{\period^2\epl^2c}} \le e^{-A} + \int_1^\infty y^n e^{-Ay^2}\dd y.
\end{equation}
Then since $n < A$ by our assumption, we have
\begin{equation}
\int_1^\infty y^n e^{-Ay^2}\dd y = \int_1^\infty y e^{(n-1)\log(y)} e^{-Ay^2}\dd y \le \int_1^\infty y e^{(n-1)y^2} e^{-Ay^2}\dd y,
\end{equation}
which implies
\begin{equation}
\int_1^\infty y^n e^{-Ay^2}\dd y \le \left[\frac{1}{2(n-1-A)}e^{-(A+1-n)y^2}\right]_{1}^\infty < \frac{1}{2}e^{-(A+1-n)}.
\end{equation} 
Thus,
\begin{equation}
\begin{split}
\sum_{k=1}^\infty \abs{\int_\R \psi_n(x)e^{-\frac{(x-\mu)^2}{2\sigma^2}\pm i\frac{2\pi}{\period}k\frac{x}{\epl}}\dd x} &\le C(n,\sigma^2) \left[  \left(2^{n/2}n^{n/2}\abs{\frac{2}{c}-1}^{n/2} + 2^{n-1}\left(\frac{\abs{\mu}}{c\sigma^2}\right)^n\right)\left(1 + \frac{\period^2\epl^2 c}{4\pi^2}\right) \right. \\ 
&\hspace{2.5cm} + \left. 2^{n-1}\left(\frac{2\pi}{\period\epl c}\right)^n \left(1 + \frac12 e^{n-1}\right) \right]e^{-\frac{2\pi^2}{\period^2\epl^2 c}},
\end{split}
\end{equation}
which implies the desired result for $\sigma^2 \neq 1$. \\
\textbf{Case $\sigma^2=1$.} In this case, we have
\begin{equation}
|\widetilde{H}_n(\mp 2\pi k/(\period\epl); \sigma^2)| = \abs{\mu \pm \frac{2\pi ki}{\period\epl}}^n \le 2^{n-1}\abs{\mu}^n + 2^{n-1}\left(\frac{2\pi k}{\period\epl}\right)^n.
\end{equation}
Thus,
\begin{equation}
\sum_{k=1}^\infty \left|\int_\R \psi_n(x)e^{-\frac{(x-\mu)^2}{2\sigma^2}\pm i\frac{2\pi}{\period}k\frac{x}{\epl}}\dd x\right| = \frac{\pi^{1/4}e^{-\frac{\mu^2}{4}}}{\sqrt{2^n n!}}2^{n-1} \left[ \abs{\mu}^n\sum_{k=1}^\infty e^{-\frac{\pi^2 k^2}{\period^2\epl^2}} + \left(\frac{2\pi}{\period\epl}\right)^n\sum_{k=1}^\infty k^ne^{-\frac{\pi^2 k^2}{\period^2\epl^2}} \right].
\end{equation}
As computed in the previous case, we have
\begin{equation}
\sum_{k=1}^\infty e^{-\frac{\pi^2 k^2}{\period^2\epl^2}} \le e^{-\frac{\pi^2}{\period^2\epl^2}} + \frac{\period^2\epl^2}{2\pi^2}e^{-\frac{\pi^2}{\period^2\epl^2}} \qquad \text{and} \qquad \sum_{k=1}^\infty k^n e^{-\frac{\pi^2k^2}{\period^2\epl^2}} \le e^{-\frac{\pi^2}{\period^2\epl^2}} + \frac{1}{2}e^{n-1 - \frac{\pi^2}{\period^2\epl^2}},
\end{equation}
which yield the desired result for $\sigma^2 = 1$ and concludes the proof.
\end{proof}

\subsubsection{Quadratic potential and the Gaussian setting}

In this section, we restrict the analysis to the quadratic potential $V(x) = (x-\mu)^2/2$ for any $\mu\in\R$, which implies that $\rho$ is the density of a Gaussian random variable $\mathcal N(\mu, \sigma^2)$, and therefore the normalization constant is $Z = \sqrt{2\pi\sigma^2}$. Using its characteristic function, we first derive a closed-form expression for the Fourier coefficients of the Gaussian density.

\begin{lemma} \label{lem:bound_Fourier}
Let $\mu \in \R$. Then, for every $k\in\N$, it holds that
\begin{equation}
\int_\R e^{-\frac{(x-\mu)^2}{2\sigma^2} \pm i \frac{2\pi}{\period} k \frac{x}\epl} \dd x = \sqrt{2\pi\sigma^2}e^{\pm\frac{2\pi k i\mu}{\period\epl}-\frac12\sigma^2\left(\frac{2\pi k}{\period\epl}\right)^2}.
\end{equation}
\end{lemma}
\begin{proof}
Consider a random variable $X\sim N(\mu, \sigma^2)$. The characteristic function of $X$ is given by
\begin{equation}
\phi_X(t) = \E \left[ e^{itX} \right] = e^{it\mu-\frac12 \sigma^2 t^2}.
\end{equation}
Thus, we have
\begin{equation} \label{eq:normal-characteristic-fn}
\int_\R \frac{1}{\sqrt{2\pi\sigma^2}} e^{-\frac{(x-\mu)^2}{2\sigma^2}+itx}\dd x = e^{it\mu-\frac12 \sigma^2 t^2}.
\end{equation}
The result follows from letting $t=\pm \frac{2\pi k}{\period\epl}$ and multiplying \eqref{eq:normal-characteristic-fn} by $\sqrt{2\pi\sigma^2}$.
\end{proof}

In the following result, we derive the rate of convergence of some quantities related to the expectations of the Hermite functions with respect to the multiscale invariant measure towards their homogenized limit.

\begin{lemma} \label{lem:convergence_Z_psi}
Let $Z$ and $Z^\epl$ be the normalization constants of the invariant densities $\rho$ and $\rho^\epl$ in equation \eqref{eq:invariant_density} with potential $V = (x-\mu)^2/2$ and $\mu \in \R$. Then, it holds that
\begin{equation}
\abs{\frac{\Pi Z}\period - Z^\epl} \le \frac{4\Pi}{\period}\sqrt{2\pi\sigma^2}\left(1 + \frac{\period^2\epl^2}{4\pi^2\sigma^2}\right)e^{-\frac{\sigma^2}{2}\left(\frac{2\pi}{\period\epl}\right)^2}.
\end{equation}
where $\Pi$ is defined in equation \eqref{eq:Pi_def}. Now, let $n\in\N$ with $n < \frac{2\pi^2}{\period^2\epl^2 c}$, where $c = \frac{\sigma^2+1}{\sigma^2}$. Then, if $\sigma^2\ne 1$, it holds that
\begin{equation}
\begin{split}
&\abs{\int_\R \psi_n(x) e^{- \frac{(x-\mu)^2}{2\sigma^2}} \left( e^{- \frac{1}{\sigma^2} p \left( \frac{x}\epl \right)} - \frac{\Pi}{\period} \right) \dd x} \\
&\hspace{3cm}\le \frac{4\Pi C(n,\sigma^2)e^{-\frac{2\pi^2}{\period^2\epl^2 c}}}{\period}\left[\left(2^{n/2}n^{n/2}\abs{\frac{2}{c}-1}^{n/2} + 2^{n-1}\left(\frac{\abs{\mu}}{c\sigma^2}\right)^n\right)\left(1 + \frac{\period^2\epl^2 c}{4\pi^2}\right) \right. \\
&\hspace{7cm}\left. + 2^{n-1}\left(\frac{2\pi}{\period\epl c}\right)^n\left(1 + \frac12 e^{n-1}\right)\right],
\end{split}
\end{equation}
where $C(n,\sigma^2) = \frac{\pi^{1/4}4^n\left(1 + \frac{n}{2}\right)e^{-\frac{\mu^2}{2\sigma^2}+\frac{\mu^2}{2c\sigma^4}}}{\sqrt{c\cdot 2^{n-1}n!}}$. Otherweise, if $\sigma^2=1$, it holds that
\begin{equation}
\begin{split}
\abs{\int_\R \psi_n(x) e^{- \frac{(x-\mu)^2}{2\sigma^2}} \left( e^{- \frac{1}{\sigma^2} p \left( \frac{x}\epl \right)} - \frac{\Pi}{\period} \right) \dd x} \le \frac{4\Pi\pi^{1/4}2^{n-1}e^{-\frac{\mu^2}{4}}}{\period\sqrt{2^n n!}} e^{-\frac{\pi^2}{\period^2\epl^2}}&\left(\abs{\mu}^n\left(1 + \frac{\period^2\epl^2}{2\pi^2}\right) \right. \\
&\quad\left. + \left(\frac{2\pi}{\period\epl}\right)^n\left(1 + \frac12 e^{n-1}\right)\right).
\end{split}
\end{equation}
\end{lemma}
\begin{proof}
Let
\begin{equation}
u(y) = \frac{\Pi}{\period} -  e^{- \frac{1}{\sigma^2} p(y)},
\end{equation}
and notice that
\begin{equation}
\frac1\period \int_0^\period u(y) \dd y = 0.
\end{equation}
Therefore, the Fourier series of $u$ reads
\begin{equation}
u(y) = \sum_{k=1}^\infty \left( c_k^+ e^{i \frac{2\pi}{\period} ky} + c_k^- e^{- i \frac{2\pi}{\period} ky} \right), \qquad c_k^\pm = \frac1\period \int_0^\period u(y) e^{\mp i \frac{2\pi}\period ky} \dd y,
\end{equation}
where it holds
\begin{equation}
\abs{c_k^\pm} \le \frac1\period \int_0^\period \abs{u(y)} \dd y \le \frac{2\Pi}{\period}.
\end{equation}
Then, we have
\begin{equation}
\abs{\frac{\Pi Z}{\period} - Z^\epl} = \abs{\int_\R e^{- \frac{(x-\mu)^2}{2\sigma^2}} \left( \frac\Pi{\period} -  e^{- \frac{1}{\sigma^2} p \left( \frac{x}\epl \right)} \right) \dd x} = \abs{\int_\R e^{- \frac{(x-\mu)^2}{2\sigma^2}} u \left( \frac{x}\epl \right) \dd x},
\end{equation}
which implies
\begin{equation}
\begin{split}
\abs{\frac{\Pi Z}\period - Z^\epl} &\le \sum_{k=1}^\infty \left( \abs{c_k^+} \abs{\int_\R e^{-\frac{(x-\mu)^2}{2\sigma^2} + i \frac{2\pi}\period k \frac{x}\epl} \dd x} + \abs{c_k^-} \abs{\int_\R e^{-\frac{(x-\mu)^2}{2\sigma^2} - i \frac{2\pi}\period k \frac{x}\epl} \dd x} \right) \\
&\le \frac{2\Pi}\period \sum_{k=1}^\infty \left( \abs{\int_\R e^{-\frac{(x-\mu)^2}{2\sigma^2} + i \frac{2\pi}\period k \frac{x}\epl} \dd x} + \abs{\int_\R e^{-\frac{(x-\mu)^2}{2\sigma^2} - i \frac{2\pi}\period k \frac{x}\epl} \dd x} \right).
\end{split}
\end{equation}
Using \cref{lem:bound_Fourier}, we obtain
\begin{equation}
\label{eq:exponential-sq-sum-bound}
\abs{\frac{\Pi Z}\period - Z^\epl} \le \frac{4\Pi}\period \sqrt{2\pi\sigma^2} \sum_{k=1}^\infty e^{-\frac12 \sigma^2\left(\frac{2\pi k}{\period\epl}\right)^2}.
\end{equation}
We now bound the sum by
\begin{equation}
\begin{split}
\sum_{k=1}^\infty e^{-\frac12 \sigma^2\left(\frac{2\pi k}{\period\epl}\right)^2} 
&\le e^{-\frac12\sigma^2\left(\frac{2\pi}{\period\epl}\right)^2} + \int_1^\infty e^{-\frac12 \sigma^2\left(\frac{2\pi y}{\period\epl}\right)^2}\dd y\\
&\le e^{-\frac12\sigma^2\left(\frac{2\pi}{\period\epl}\right)^2} + \int_1^\infty y e^{-\frac12 \sigma^2\left(\frac{2\pi y}{\period\epl}\right)^2}\dd y\\
&= e^{-\frac{\sigma^2}{2}\left(\frac{2\pi}{\period\epl}\right)^2} + \frac{\period^2\epl^2}{4\pi^2\sigma^2}e^{-\frac{\sigma^2}{2}\left(\frac{2\pi}{\period\epl}\right)^2},
\end{split}
\end{equation}
which combines with \eqref{eq:exponential-sq-sum-bound} to obtain the first desired result. Let us now focus on the second estimate. Using again the definition of $u$ and its Fourier series, we have
\begin{equation}
\begin{split}
\abs{\int_\R \psi_n(x) e^{- \frac{(x-\mu)^2}{2\sigma^2}} \left( e^{- \frac{1}{\sigma^2} p \left( \frac{x}\epl \right)} - \frac\Pi{\period} \right) \dd x} &= \abs{\int_\R \psi_n(x) e^{- \frac{(x-\mu)^2}{2\sigma^2}} u \left( \frac{x}\epl \right) \dd x} \\
&\hspace{-3cm}\le \frac{2\Pi}\period \sum_{k=1}^\infty \left( \abs{\int_\R \psi_n(x) e^{-\frac{(x-\mu)^2}{2\sigma^2} + i \frac{2\pi}\period k \frac{x}\epl} \dd x} + \abs{\int_\R \psi_n(x) e^{-\frac{(x-\mu)^2}{2\sigma^2} - i \frac{2\pi}\period k \frac{x}\epl} \dd x} \right),
\end{split}
\end{equation}
which, due to \cref{lem:psi-fourier-series-bound}, implies the second and third results and concludes the proof.
\end{proof}

\begin{remark} \label{rem:polynomial_rate}
Similar estimates for a general potential $V$ can be obtained using either Cesàro mean approximations or repeated integration by parts; see \cite[Corollary A.4]{BKP25} and \cite[Theorem 4.14]{Zan22}, respectively. However, both approaches yield only a polynomial convergence rate $\mathcal O(\epl^{\mathscr \ell})$, where the speed, i.e., the power $\ell$ in the decay, depends on the regularity of the potential and the test functions. In \cref{lem:convergence_Z_psi}, by considering the specific case where the test functions are the analytic Hermite functions $\psi_n$, we improve upon these estimates by establishing the optimal exponential convergence rate $\mathcal O(e^{-1/\epl^2})$ in the Gaussian setting.
\end{remark}

The previous lemma allows us to compute the convergence rate of the coefficients $\alpha_n^\epl$ in the Fourier expansion of $\rho^\epl$ towards their homogenized counterparts $\alpha_n$. This is quantified next.

\begin{lemma} \label{lem:bound_alpha}
Let $\alpha_n^\epl$ and $\alpha_n$ with $n\in\N\cup\{0\}$ be the Fourier coefficients of the invariant densities $\rho$ and $\rho^\epl$ in equation \eqref{eq:invariant_density} with potential $V = (x-\mu)^2/2$ and $\mu \in \R$. Moreover, set $n<\frac{2\pi^2}{\period^2\epl^2c}$, where $c=\frac{\sigma^2+1}{\sigma^2}$, and $\epl \le \frac{\pi\sqrt{\sigma^2}}{\period\sqrt{\log(4)}}$. Then, if $\sigma^2\ne 1$, it holds
\begin{equation}
\begin{split}
\abs{\alpha_n^\epl - \alpha_n}
&\le \frac{8\pi^{-1/4}\left(1 + \frac{\sigma^2c}{4\log(4)}\right)(1+\frac{n}{2})}{\sqrt{\sigma^2+1}} e^{-\frac{2\pi^2}{\period^2\epl^2 c} + \frac{n}{2} + \frac{n}{2}\log\abs{\frac{2}{c}-1} + n\log(4)}\\
&\; + \frac{4\pi^{-1/4}\left(1 + \frac{\sigma^2c}{4\log(4)}\right)\left(1 + \frac{n}{2}\right)}{\sqrt{(\sigma^2+1)n!}}\left(\frac{4\sqrt{2}\abs{\mu}}{\sigma^2+1}\right)^ne^{-\frac{2\pi^2}{\period^2\epl^2 c}} + \frac{6\pi^{-1/4}8^n\left(1 + \frac{n}{2}\right)}{\sqrt{(\sigma^2+1)2^nn!}}\left(\frac{2\pi e}{\period\epl c}\right)^ne^{-\frac{2\pi^2}{\period^2\epl^2 c}}\\ 
&\;+ 16\pi^{-1/4}e^{-\frac{\sigma^2}{2}\left(\frac{2\pi}{\period\epl}\right)^2},
\end{split}
\end{equation}
and, if $\sigma^2 = 1$, we have
\begin{equation}
\abs{\alpha_n^\epl - \alpha_n} \le \frac{16\pi^{-1/4}\abs{\mu}^n2^{n-1}}{\sqrt{2^{n+1} n!}} e^{-\frac{\mu^2}{4}-\frac{\pi^2}{\period^2\epl^2}} + \frac{12\pi^{-1/4}2^{n-1}}{\sqrt{2^{n+1} n!}} e^{-\frac{\mu^2}{4}-\frac{\pi^2}{\period^2\epl^2}}\left(\frac{2\pi e}{\period\epl}\right)^n + 16\pi^{-1/4}e^{-\frac{\sigma^2}{2}\left(\frac{2\pi}{\period\epl}\right)^2}.
\end{equation}
\end{lemma}
\begin{proof}
By definition of $\alpha_n^\epl$ and $\alpha_n$, we have
\begin{equation}
\abs{\alpha_n^\epl - \alpha_n} = \abs{\int_\R \psi_n(x) (\rho^\epl(x) - \rho(x)) \dd x} = \abs{\int_\R \psi_n(x) e^{- \frac{(x-\mu)^2}{2\sigma^2}} \left( \frac1{Z^\epl} e^{- \frac{1}{\sigma^2} p \left( \frac{x}\epl \right)} - \frac1{Z} \right) \dd x},
\end{equation}
which, due to the triangle inequality, implies
\begin{equation}
\abs{\alpha_n^\epl - \alpha_n} \le \frac1{Z^\epl} \abs{\int_\R \psi_n(x) e^{- \frac{(x-\mu)^2}{2\sigma^2}} \left( e^{- \frac{1}{\sigma^2} p \left( \frac{x}\epl \right)} - \frac\Pi{\period} \right) \dd x} + \frac1{Z^\epl} \abs{\frac{\Pi Z}\period - Z^\epl} \abs{\int_\R \psi_n(x) \frac1Z e^{- \frac{(x-\mu)^2}{2\sigma^2}} \dd x}.
\end{equation}
Notice that, by the first estimate in \cref{lem:convergence_Z_psi}, we have $Z^\epl \ge (Z\Pi)/(2\period)$ since for $\epl \le \frac{\pi\sqrt{\sigma^2}}{\period\sqrt{\log(4)}}$,
\begin{equation}
\abs{\frac{\Pi Z}{\period} - Z^\epl} \le \frac{4\Pi Z}{\period}\left(1 + \frac{\period^2\epl^2}{4\pi^2\sigma^2}\right)e^{-\frac{\sigma^2}{2}\left(\frac{2\pi}{\period\epl}\right)^2} \le \frac{4\Pi Z}{\period}\left(1 + \frac{1}{4\log(4)}\right) e^{-2\log(4)} \le \frac{\Pi Z}{2\period},
\end{equation}
where we recall that $Z = \sqrt{2\pi\sigma^2}$ here. Then by Cramér's inequality \cite{Ind61}, it holds for all $x \in \R$
\begin{equation} \label{eq:Cramer}
\abs{\psi_n(x)} \le \pi^{-1/4}.
\end{equation}
Therefore, we deduce
\begin{equation}
\abs{\alpha_n^\epl - \alpha_n} \le \frac{2\period}{\Pi Z} \abs{\int_\R \psi_n(x) e^{- \frac{(x-\mu)^2}{2\sigma^2}} \left( e^{- \frac{1}{\sigma^2} p \left( \frac{x}\epl \right)} - \frac\Pi{\period} \right) \dd x} + \frac{2\period \pi^{-\frac14}}{\Pi Z} \abs{\frac{\Pi Z}\period - Z^\epl}.
\end{equation}
Using \cref{lem:convergence_Z_psi}, we obtain that if $\sigma^2\ne 1$, then
\begin{equation}
\begin{split}
\abs{\alpha_n^\epl - \alpha_n}
&\le \frac{8 C(n,\sigma^2)e^{-\frac{2\pi^2}{\period^2\epl^2 c}}}{\sqrt{2\pi\sigma^2}}\left[\left(2^{n/2}n^{n/2}\abs{\frac{2}{c}-1}^{n/2} + 2^{n-1}\left(\frac{\abs{\mu}}{c\sigma^2}\right)^n\right)\left(1 + \frac{\period^2\epl^2 c}{4\pi^2}\right)\right.\\ 
&\hspace{10ex}\left.+ 2^{n-1}\left(\frac{2\pi}{\period\epl c}\right)^n\left(1 + \frac12 e^{n-1}\right)\right] + 8\pi^{-1/4}\left(1 + \frac{\period^2\epl^2}{4\pi^2\sigma^2}\right)e^{-\frac{\sigma^2}{2}\left(\frac{2\pi}{\period\epl}\right)^2}\\
&\le \frac{8e^{-\frac{2\pi^2}{\period^2\epl^2 c}}\pi^{-1/4}4^n\left(1 + \frac{n}{2}\right)}{\sqrt{(\sigma^2+1)2^n n!}}\left[\left(1 + \frac{\sigma^2c}{4\log(4)}\right)\left(2^{n/2}n^{n/2}\abs{\frac{2}{c}-1}^{n/2} + 2^{n-1}\left(\frac{\abs{\mu}}{\sigma^2+1}\right)^n\right)\right.\\ 
&\hspace{10ex}\left.+ 2^{n-1}\left(\frac{2\pi}{\period\epl c}\right)^n\left(1 + \frac12 e^{n-1}\right)\right] + 16\pi^{-1/4}e^{-\frac{\sigma^2}{2}\left(\frac{2\pi}{\period\epl}\right)^2},
\end{split}
\end{equation}
where we used that $e^{-\frac{\mu^2}{2\sigma^2} + \frac{\mu^2}{2c\sigma^4}} \le 1$ and $1 + e^{n-1}/2 \le 3 e^n /2$. If $n = 0$, then by convention, $n^{n/2} = 1$, so $\frac{n^{n/2}}{\sqrt{n!}} = 1$. Otherwise, when $n>0$, we have, by Stirling's approximation, $n! > \sqrt{2\pi n}\left(\frac{n}{e}\right)^n$ that $\frac{n^{n/2}}{\sqrt{n!}} \le e^{n/2}(2\pi n)^{-1/4} \le e^{n/2}$. In both cases, we can bound $\frac{n^{n/2}}{\sqrt{n!}} \le e^{n/2}$, so the desired result follows. The estimate for $\sigma = 1$ is obtained similarly from \cref{lem:convergence_Z_psi}.
\end{proof}

In the truncated sum $\rho^\epl_N$, we are neglecting the higher-order coefficients for $n \ge N$. This is justified by the following lemma, where we compute the rate of convergence to zero of the remainder $\sum_{n=N}^\infty \alpha_n^2$.

\begin{lemma} \label{lem:rate_alpha_n}
Let $\alpha_n$ with $n\in\N\cup\{0\}$ be the Fourier coefficients of the invariant density $\rho$ in equation \eqref{eq:invariant_density} with potential $V = (x-\mu)^2/2$ and $\mu \in \R$. If $N$ satisfies
\begin{equation}
N \ge \begin{cases}
\frac{32 \mu^2}{\abs{\sigma^4 - 1}} \left( \log \abs{\frac{\sigma^2 + 1}{\sigma^2 - 1}} \right)^{-2}, & \text{if } \sigma^2 \neq 1, \\
\frac{e^{3/2} \mu^2}2, & \text{if } \sigma^2 = 1,
\end{cases}
\end{equation}
then it holds
\begin{equation}
\left( \sum_{n=N}^\infty \alpha_n^2 \right)^{\frac12} \le \frac1{\pi^{1/4} \sqrt{(\sigma^2 + 1)(1 - e^{-2\lambda})}} e^{- \frac{\mu^2}{2(\sigma^2 + 1)} - \lambda N},
\end{equation}
where 
\begin{equation} \label{eq:lambda_def}
\lambda = \begin{cases}
\frac14 \log \abs{\frac{\sigma^2 + 1}{\sigma^2 - 1}}, & \text{if } \sigma^2 \neq 1, \\
\frac14, & \text{if } \sigma^2 = 1.
\end{cases}
\end{equation}
\end{lemma}
\begin{proof}
Recalling that
\begin{equation}
\alpha_n = \frac1{\sqrt{2\pi\sigma^2}} \int_\R \psi_n(x) e^{- \frac{(x - \mu)^2}{2\sigma^2}} \dd x,
\end{equation}
and applying \cref{lem:bound_Fourier_psi} with $k=0$, we obtain 
\begin{equation}
\alpha_n = \begin{cases}
\frac{(-1)^n}{\pi^{1/4} \sqrt{(\sigma^2 + 1) 2^n n!}} e^{- \frac{\mu^2}{2(\sigma^2 + 1)}} \left( \frac{1 - \sigma^2}{1 + \sigma^2} \right)^{n/2} H_n \left( - \frac{\mu}{\sqrt{1 - \sigma^4}} \right) & \text{if } \sigma^2 < 1, \\
\frac1{\pi^{1/4} \sqrt{(\sigma^2 + 1) 2^n n!}} e^{- \frac{\mu^2}{2(\sigma^2 + 1)}} \mu^n  & \text{if } \sigma^2 = 1, \\
\frac{(-i)^n}{\pi^{1/4} \sqrt{(\sigma^2 + 1) 2^n n!}} e^{- \frac{\mu^2}{2(\sigma^2 + 1)}} \left( \frac{\sigma^2 - 1}{\sigma^2 + 1} \right)^{n/2} H_n \left( \frac{i\mu}{\sqrt{\sigma^4 - 1}} \right) & \text{if } \sigma^2 > 1.
\end{cases}
\end{equation}
Using the following bound from \cite{EiM90}
\begin{equation}
\abs{H_n(z)} \le \sqrt{2^n n!} e^{\sqrt{2n}\abs{z}} \quad \text{for all } n \in \N \text{ and } z \in \C,
\end{equation}
we get
\begin{equation}
\abs{\alpha_n} \le
\begin{cases}
\frac1{\pi^{1/4} \sqrt{\sigma^2 + 1}} e^{- \frac{\mu^2}{2(\sigma^2 + 1)}} \abs{\frac{\sigma^2 - 1}{\sigma^2 + 1}}^{n/2} e^{\frac{\sqrt{2n}\abs{\mu}}{\sqrt{\abs{\sigma^4 - 1}}}} & \text{if } \sigma^2 \neq 1, \\
\frac1{\pi^{1/4} \sqrt{\sigma^2 + 1}} e^{- \frac{\mu^2}{2(\sigma^2 + 1)}} \frac{|\mu|^n}{\sqrt{2^n n!}} & \text{if } \sigma^2 = 1.
\end{cases}
\end{equation}
Let us now distinguish the two cases. If $\sigma^2 \neq 1$, then we have
\begin{equation}
\abs{\alpha_n} \le \frac1{\pi^{1/4} \sqrt{\sigma^2 + 1}} e^{- \frac{\mu^2}{2(\sigma^2 + 1)}} \exp \left\{ -\frac{n}2 \log \abs{\frac{\sigma^2 + 1}{\sigma^2 - 1}} + \frac{\sqrt{2n}\abs{\mu}}{\sqrt{\abs{\sigma^4 - 1}}} \right\},
\end{equation}
which, due to the assumption on $N \le n$, implies
\begin{equation} \label{eq:bound_alpha_sigmaNOT1}
\abs{\alpha_n} \le \frac1{\pi^{1/4} \sqrt{\sigma^2 + 1}} e^{- \frac{\mu^2}{2(\sigma^2 + 1)}} \exp \left\{ -\frac{n}4 \log \abs{\frac{\sigma^2 + 1}{\sigma^2 - 1}} \right\}.
\end{equation}
On the other hand, if $\sigma^2 = 1$, using Stirling's approximation, we obtain
\begin{equation}
\abs{\alpha_n} \le \frac1{\pi^{1/4} \sqrt{\sigma^2 + 1}} e^{- \frac{\mu^2}{2(\sigma^2 + 1)}} \frac{|\mu|^n}{\sqrt{2^n n^n e^{-n} \sqrt{2\pi n}}} \le \frac1{\pi^{1/4} \sqrt{\sigma^2 + 1}} e^{- \frac{\mu^2}{2(\sigma^2 + 1)}} \left( \frac{e \mu^2}{2n} \right)^{n/2},
\end{equation}
which, due to the assumption on $N \le n$, implies
\begin{equation} \label{eq:bound_alpha_sigma1}
\abs{\alpha_n} \le \frac1{\pi^{1/4} \sqrt{\sigma^2 + 1}} e^{- \frac{\mu^2}{2(\sigma^2 + 1)}} \exp \left\{ - \frac{n}2 \log \left( \frac{2n}{e\mu^2} \right) \right\} \le \frac1{\pi^{1/4} \sqrt{\sigma^2 + 1}} e^{- \frac{\mu^2}{2(\sigma^2 + 1)}} \exp \left\{ - \frac{n}4 \right\}.
\end{equation}
Finally, from equations \eqref{eq:bound_alpha_sigmaNOT1} and \eqref{eq:bound_alpha_sigma1}, we deduce that
\begin{equation}
\sum_{n=N}^\infty \alpha_n^2 \le \frac1{\sqrt\pi (\sigma^2 + 1)} e^{- \frac{\mu^2}{\sigma^2 + 1}} \sum_{n=N}^\infty \left( e^{-2\lambda} \right)^n = \frac1{\sqrt\pi (\sigma^2 + 1) (1 - e^{-2\lambda})} e^{- \frac{\mu^2}{\sigma^2 + 1} - 2\lambda N},
\end{equation}
where $\lambda > 0$ is defined in equation \eqref{eq:lambda_def}, and which implies the desired result.
\end{proof}

By combining the above results, we can finally state and prove the main result of this section, which is the convergence of the approximation $\rho^\epl_N$ to the invariant measure $\rho$ in the Gaussian setting with quadratic potential.

\begin{proposition} \label{pro:Gaussian-Case}
Consider the potential $V(x) = (x - \mu)^2/2$. Let $\{\epl_k\}_k$ be such that $\epl_k\to0$ as $k\to\infty$ with $\epl_k \le \frac{\pi\sqrt{\sigma^2}}{\period\sqrt{\log(4)}}$ and small enough that $\left\lfloor \frac{\pi^2}{\gamma\period^2\epl_k^2}\right\rfloor \ge 2$, where $\gamma > 3+\log{8}$ if $\sigma^2 = 1$ and $\gamma > c^2 + \max\left\{16e^{3/2}, \frac{c}{4}\left(\log\abs{\frac{2}{c}-1}+2\log(4)\right)\right\}$ if $\sigma^2\ne 1$, where $c = \frac{\sigma^2+1}{\sigma^2}$. Let $N_k = N_k(\epl_k) \in \N$ be any sequence such that $N_k\to\infty$ as $k\to\infty$ and
\begin{equation} \label{eq:bound_N}
2 \le N_k \le \left\lfloor \frac{\pi^2}{\gamma\period^2\epl_k^2}\right\rfloor.
\end{equation}
Then $\rho^{\epl_k}_{N_k}\to\rho$ in $L^2(\R)$ as $k\to\infty$. Furthermore, there are constants $C, \lambda_1,\lambda_2 > 0$ independent of $k$ such that
\begin{equation}
\norm{\rho^{\epl_k}_{N_k} - \rho}_{L^2(\R)} \le Ce^{-\frac{\lambda_1}{\period^2\epl_k^2}} + Ce^{-\lambda_2 N_k}
\end{equation}
for $k$ sufficiently large.
\end{proposition}
\begin{proof}
For simplicity, we will suppress the subscript $k$ in $\epl_k$ and $N_k$. By the triangle inequality and since $\psi_n$ forms an orthonormal basis, we have
\begin{equation}
\norm{\rho^\epl_N - \rho}_{L^2(\R)} \le \left( \sum_{n=0}^{N-1} (\alpha_n^\epl - \alpha_n)^2 \right)^{\frac12} + \left( \sum_{n=N}^\infty \alpha_n^2 \right)^{\frac12},
\end{equation}
where we note that the second term in the right-hand side vanishes as $N \to \infty$ by \cref{lem:rate_alpha_n}. We now consider the cases $\sigma^2\ne 1$ and $\sigma^2 = 1$ separately.

\textbf{Case $\sigma^2\ne 1$.} Using \cref{lem:bound_alpha}, we obtain
\begin{equation}
\begin{split}
\sum_{n=0}^{N-1} (\alpha_n^\epl - \alpha_n)^2 
&\le C\sum_{n=0}^{N-1}\left[\left(1+\frac{n}{2}\right)^2 e^{-\frac{4\pi^2}{\period^2\epl^2 c} + n + n\log\abs{\frac{2}{c}-1} + 2n\log(4)} + \frac{\left(1 + \frac{n}{2}\right)^2}{n!}\left(\frac{4\sqrt{2}\abs{\mu}}{\sigma^2+1}\right)^{2n}e^{-\frac{4\pi^2}{\period^2\epl^2 c}}\right.\\ 
&\hspace{10ex}\left.+ \frac{8^{2n}\left(1 + \frac{n}{2}\right)^2}{2^nn!}\left(\frac{2\pi e}{\period\epl c}\right)^{2n}e^{-\frac{4\pi^2}{\period^2\epl^2 c}} + e^{-\sigma^2\left(\frac{2\pi}{\period\epl}\right)^2}\right]\\
&\le CN^2e^{-\frac{4\pi^2}{\period^2\epl^2 c}}\sum_{n=0}^{N-1} e^{n+n\log\abs{\frac{2}{c}-1}+2n\log(4)} + CNe^{-\frac{4\pi^2\sigma^2}{\period^2\epl^2}}\\ 
&\hspace{10ex}+ CN^2 e^{-\frac{4\pi^2}{\period^2\epl^2 c}}\sum_{n=0}^{N-1} \frac{1}{n!}\left(\frac{4\sqrt{2}\abs{\mu}}{\sigma^2+1}\right)^{2n} + CN^2 e^{-\frac{4\pi^2}{\period^2\epl^2 c}} \sum_{n=0}^{N-1} \frac{1}{n!}\left(\frac{128\pi^2 e^2}{\period^2\epl^2 c^2}\right)^n,
\end{split}
\end{equation}
where $C>0$ is some constant. Note that the first sum satisfies
\begin{equation}
\sum_{n=0}^{N-1} e^{n+n\log\abs{\frac{2}{c}-1}+2n\log(4)} = \sum_{n=0}^{N-1} \left(e^{1 + \log\abs{\frac{2}{c}-1}+2\log(4)}\right)^n \le N\left(1 + e^{N(1+\log\abs{\frac{2}{c}-1}+2\log(4))}\right).
\end{equation}
For the second sum, we have
\begin{equation}
\sum_{n=0}^{N-1} \frac{1}{n!}\left(\frac{4\sqrt{2}\abs{\mu}}{\sigma^2+1}\right)^{2n} \le \exp\left(\frac{32\mu^2}{(\sigma^2+1)^2}\right).
\end{equation}
For the third sum, we have
\begin{equation}
\sum_{n=0}^{N-1} \frac{1}{n!}\left(\frac{128\pi^2 e^2}{\period^2\epl^2 c^2}\right)^n = \frac{1}{(N-1)!}e^{\frac{128\pi^2 e^2}{\period^2\epl^2 c^2}}\Gamma\left(N, \frac{128\pi^2 e^2}{\period^2\epl^2 c^2}\right) \le \frac{1}{(N-1)!} e^{\frac{128\pi^2 e^2}{\period^2\epl^2 c^2}}\Gamma\left(N, \frac{128 e^2\gamma}{c^2}N\right),
\end{equation}
where $\Gamma(N, x)$ is the upper incomplete gamma function \cite[Section 6.5]{AbS64}. Thus, defining $A = 128 e^2\gamma/c^2$ and using $1+z \le e^z$ for $z\in\R$, we have
\begin{equation}
\begin{split}
\Gamma\left(N, AN\right)
&= \int_{AN}^\infty t^{N-1}e^{-t}\dd t = \int_0^\infty (AN+u)^{N-1}e^{-AN-u}\dd u \le (AN)^{N-1}e^{-AN}\int_0^\infty e^{\frac{(N-1)u}{AN}}e^{-u}\dd u,
\end{split}
\end{equation}
which yields
\begin{equation}
\Gamma\left(N, \frac{128 e^2\gamma}{c^2}N\right) \le A^{N-1}N^{N-1}e^{-AN}\int_0^\infty e^{u(1/A-1)}\dd u = A^{N-1}N^{N-1}e^{-AN}\frac{1}{1-A^{-1}},
\end{equation}
since $A > 1$ by assumption.
This implies 
\begin{equation}
\begin{split}
\left( \sum_{n=0}^{N-1} (\alpha_n^\epl - \alpha_n)^2 \right)^{\frac12} 
&\le CN^{3/2}e^{-\frac{2\pi^2}{\period^2\epl^2 c}}\left(1 + e^{\frac{N}{2}(1 + \log\abs{\frac{2}{c}-1}+2\log(4))}\right) \\
&\quad + C\sqrt{N}e^{-\frac{2\pi^2\sigma^2}{\period^2\epl^2}} + CNe^{-\frac{2\pi^2}{\period^2\epl^2 c}+ \frac{32\mu^2}{(\sigma^2+1)^2}}\\ 
&\quad+ \frac{CN\cdot N^{\frac{N-1}{2}}}{\sqrt{1 - \frac{c^2}{128 e^2\gamma}}\sqrt{(N-1)!}}e^{-\frac{2\pi^2}{\period^2\epl^2c}+\frac{64\pi^2e^2}{\period^2\epl^2c^2}-\frac{64e^2\gamma}{c^2}N + \frac{N-1}{2}\log\left(\frac{128e^2\gamma}{c^2}\right)}.
\end{split}
\end{equation}
Then using Stirling's approximation, we have
\begin{equation}
\frac{N^{\frac{N-1}{2}}}{\sqrt{(N-1)!}} \le \frac{N^{\frac{N-1}{2}}}{\sqrt{2\pi(N-1)}}\frac{e^{\frac{N-1}{2}}}{(N-1)^{\frac{N-1}{2}}} \le \frac{1}{\sqrt{2\pi}}\left(1 + \frac{1}{N-1}\right)^{\frac{N-1}{2}}e^{\frac{N-1}{2}} \le \frac{1}{\sqrt{2\pi}}(2e)^{\frac{N-1}{2}},
\end{equation}
so that
\begin{equation}
\begin{split}
\left( \sum_{n=0}^{N-1} (\alpha_n^\epl - \alpha_n)^2 \right)^{\frac12} 
&\le CN^{3/2}e^{-\frac{2\pi^2}{\period^2\epl^2 c}}\left(1 + e^{\frac{N}{2}(1 + \log\abs{\frac{2}{c}-1}+2\log(4))}\right) \\
&\quad+ C\sqrt{N}e^{-\frac{2\pi^2\sigma^2}{\period^2\epl^2}} + CNe^{-\frac{2\pi^2}{\period^2\epl^2 c}+ \frac{32\mu^2}{(\sigma^2+1)^2}}\\ 
&\quad+ \frac{CN}{\sqrt{2\pi\left(1 - \frac{c^2}{128 e^2\gamma}\right)}}e^{-\frac{2\pi^2}{\period^2\epl^2c}+\frac{64\pi^2e^2}{\period^2\epl^2c^2}-\frac{64e^2\gamma}{c^2}N + \frac{N-1}{2}\log\left(\frac{128e^2\gamma}{c^2}\right) + \frac{N-1}{2}\left(1 + \log(2)\right)}\\
&\eqdef I + II + III + IV.
\end{split}
\end{equation}
We have that $II\to 0$ and $III\to 0$ as $\epl\to 0$ for any $\gamma>0$. For $I$, if $c$ is close enough to $2$, then $1+\log\abs{\frac{2}{c}-1}+2\log(4) \le 0$, and $I\to 0$ as $\epl\to 0$. Otherwise, $1+\log\abs{\frac{2}{c}-1}+2\log(4)>0$, but we then have
\begin{equation}
\frac12\left\lfloor \frac{\pi^2}{\gamma \period^2\epl^2}\right\rfloor \left(1 + \log\abs{\frac{2}{c}-1}+2\log(4)\right) -\frac{2\pi^2}{\period^2\epl^2 c} \le -\frac{\pi^2}{\period^2\epl^2}\left(\frac{2}{c} - \frac{1}{2\gamma}\left(1 + \log\abs{\frac{2}{c}-1}+2\log(4)\right)\right),
\end{equation}
so, since $\gamma > \frac{c}{4}\left(1 + \log\abs{\frac{2}{c}-1}+2\log(4)\right)$ holds by assumption, we have that $I\to0$ as $\epl\to0$. For $IV$, since $\gamma > \frac{c^2}{128e^2}$, we have that
\begin{equation}
\begin{split}
&\exp\left(-\frac{2\pi^2}{\period^2\epl^2 c} + \frac{64\pi^2e^2}{\period^2\epl^2 c^2} - \frac{64e^2\gamma}{c^2}\left\lfloor \frac{\pi^2}{\gamma \period^2\epl^2}\right\rfloor\right. \\
&\hspace{20ex}\left.+ \frac{1}{2}\left(\left\lfloor \frac{\pi^2}{\gamma \period^2\epl^2}\right\rfloor-1\right)\log\left(\frac{128e^2\gamma}{c^2}\right) + \frac{1}{2}\left(\left\lfloor \frac{\pi^2}{\gamma \period^2\epl^2}\right\rfloor-1\right)(1+\log(2))\right)\\
&\qquad \le C\exp\left(-\frac{\pi^2}{\period^2\epl^2}\left[\frac{2}{c} - \frac{1}{2\gamma}\log\left(\frac{128e^2\gamma}{c^2}\right) - \frac{1}{2\gamma}(1+\log(2))\right]\right),
\end{split}
\end{equation}
so we see that $IV\to0$ when $\epl\to 0$ as long as $\gamma > \frac{c}{4}\left(\log\left(\frac{128e^2\gamma}{c^2}\right) + 1 + \log(2)\right)$. To obtain a lower bound on $\gamma$ independent of $\gamma$, we claim that $\frac{c}{4}\log(\gamma) \le \frac{\gamma}{2}$ for $\gamma\ge c^2$. Indeed, if $\gamma=Bc^2$ for any $B\ge 1$, we have using $\log(x) \le x$ that
\begin{equation}
\frac{\gamma}{2} - \frac{c}{4}\log(\gamma) = \frac{B^2c^2}{2}-\frac{c}{2}\log(Bc) \ge \frac{B^2c^2}{2}-\frac{Bc^2}{2} \ge 0
\end{equation}
since $B\ge 1$. As this holds for any $B\ge 1$, the claim holds. It follows that when $\gamma \ge c^2$, 
\begin{equation}
\frac{c}{4}\left(\log\left(\frac{128e^2\gamma}{c^2}\right) + 1 + \log(2)\right) \le \frac{\gamma}{2} + \frac{c}{2}\log\left(\frac{16e^{3/2}}{c}\right) \le \frac{\gamma}{2}+8e^{3/2},
\end{equation}
so the condition that $\gamma > c^2 + 16e^{3/2}$ will imply $\gamma > \frac{c}{4}\left(\log\left(\frac{128e^2\gamma}{c^2}\right) + 1 + \log(2)\right)$.
Since this holds by assumption, we obtain
\begin{equation}
\lim_{\epl \to 0} \norm{\rho^\epl_N - \rho}_{L^2(\R)} = 0.
\end{equation}
\textbf{Case $\sigma^2=1$.} Using \cref{lem:bound_alpha}, we have for $N \le \frac{\pi^2}{\period^2\epl^2}$
\begin{equation}
\begin{split}
\sum_{n=0}^{N-1} (\alpha_n^\epl - \alpha_n)^2 
&\le C\sum_{n=0}^{N-1} \left[\frac{\abs{\mu}^{2n}2^{2n}}{2^{n} n!} e^{-\frac{\mu^2}{2}-\frac{2\pi^2}{\period^2\epl^2}} + \frac{2^{2n}}{2^{n} n!} e^{-\frac{\mu^2}{2}-\frac{2\pi^2}{\period^2\epl^2}}\left(\frac{2\pi e}{\period\epl}\right)^{2n} + e^{-\left(\frac{2\pi}{\period\epl}\right)^2}\right]\\
&= CNe^{-\frac{4\pi^2}{\period^2\epl^2}} + Ce^{-\frac{\mu^2}{2}-\frac{2\pi^2}{\period^2\epl^2}}\sum_{n=0}^{N-1} \frac{1}{n!}\left[(2\mu^2)^n + \left(\frac{8\pi^2 e^2}{\period^2\epl^2}\right)^{n}\right].
\end{split}
\end{equation}
Now we have as before
\begin{equation}
\sum_{n=0}^{N-1} \frac{1}{n!}\left(\frac{8\pi^2 e^2}{\period^2\epl^2}\right)^{n} = \frac{1}{(N-1)!}e^{\frac{8\pi^2 e^2}{\period^2\epl^2}}\Gamma\left(N, \frac{8\pi^2 e^2}{\period^2\epl^2}\right) \le \frac{1}{(N-1)!}e^{\frac{8\pi^2 e^2}{\period^2\epl^2}}\Gamma\left(N, 8e^2\gamma N\right).
\end{equation}
Thus, defining $A = 8e^2\gamma$, we again have
\begin{equation}
\Gamma\left(N, 8e^2\gamma N\right) = \Gamma(N, AN) \le A^{N-1}N^{N-1}e^{-AN}\frac{1}{1-A^{-1}},
\end{equation}
since $A > 1$ by assumption. Then we also have
\begin{equation}
\sum_{n=0}^{N-1} \frac{(2\mu^2)^n}{n!} \le e^{2\mu^2},
\end{equation}
and it follows that
\begin{equation}
\begin{split}
\sum_{n=0}^{N-1} (\alpha_n^\epl - \alpha_n)^2
&\le CNe^{-\frac{4\pi^2}{\period^2\epl^2}} + Ce^{\frac{3}{2}\mu^2 - \frac{2\pi^2}{\period^2\epl^2}} + Ce^{-\frac{\mu^2}{2}-\frac{2\pi^2}{\period^2\epl^2}} \frac{1}{(N-1)!}e^{\frac{8\pi^2 e^2}{\period^2\epl^2}} A^{N-1}N^{N-1}e^{-AN}\frac{1}{1-A^{-1}}\\
&\le CNe^{-\frac{4\pi^2}{\period^2\epl^2}} + Ce^{-\frac{2\pi^2}{\period^2\epl^2}} + Ce^{-\frac{2\pi^2}{\period^2\epl^2}}\frac{N^{N-1}}{(N-1)!} A^{N}.
\end{split}
\end{equation}
Now using Stirling's approximation $n! > \sqrt{2\pi n}(n/e)^n$, we have
\begin{equation}
\frac{N^{N-1}}{(N-1)!} \le \frac{N^{N-1}e^{N-1}}{\sqrt{2\pi(N-1)}(N-1)^{N-1}} = \frac{\left(1 + \frac{1}{N-1}\right)^{N-1}e^{N-1}}{\sqrt{2\pi(N-1)}} \le \frac{e^{N}}{\sqrt{2\pi(N-1)}},
\end{equation}
where in the last inequality, we used $1+z \le e^z$ for $z\in\R$. Thus, we have
\begin{equation}
\begin{split}
\left(\sum_{n=0}^{N-1} (\alpha_n^\epl - \alpha_n)^2\right)^{1/2} 
&\le C\sqrt{N}e^{-\frac{2\pi^2}{\period^2\epl^2}} + Ce^{-\frac{\pi^2}{\period^2\epl^2}} + Ce^{-\frac{\pi^2}{\period^2\epl^2}}(8e^3\gamma)^{N/2}\\
&\le \frac{C}{\period\epl}e^{-\frac{2\pi^2}{\period^2\epl^2}} + Ce^{-\frac{\pi^2}{\period^2\epl^2}} + Ce^{-\frac{\pi^2}{\period^2\epl^2}+\frac{\pi^2}{2\gamma \period^2\epl^2}\log(8e^3\gamma)},
\end{split}
\end{equation}
so because we have
\begin{equation}
\frac{\log(8e^3\gamma)}{2\gamma} = \frac{3}{2\gamma} + \frac{\log{8}}{2\gamma} + \frac{\log{\gamma}}{2\gamma} \le \frac{3}{2\gamma} + \frac{\log{8}}{2\gamma} + \frac12 < 1,
\end{equation}
as $\gamma > 3+\log{8}$, it follows that
\begin{equation}
\lim_{\epl \to 0} \norm{\rho^\epl_N - \rho}_{L^2(\R)} = 0,
\end{equation}
which gives the desired result. 

Next, to obtain the rate, recall that
\begin{equation}
\norm{\rho^\epl_N - \rho}_{L^2(\R)} \le \left( \sum_{n=0}^{N-1} (\alpha_n^\epl - \alpha_n)^2 \right)^{\frac12} + \left( \sum_{n=N}^\infty \alpha_n^2 \right)^{\frac12}.
\end{equation}
The second term is bounded by $Ce^{-\lambda_2N}$ for some $C,\lambda_2>0$ by \cref{lem:rate_alpha_n}. Then in each case $\sigma^2 = 1$ and $\sigma^2\ne 1$ above, using $N \le \frac{C}{\period^2\epl^2}$, we have that
\begin{equation}
\left( \sum_{n=0}^{N-1} (\alpha_n^\epl - \alpha_n)^2 \right)^{\frac12} \le Ce^{-\frac{\lambda_3}{\period^2\epl^2} + \lambda_4\log\left(\frac{1}{\period\epl}\right)},
\end{equation}
for some $C, \lambda_3,\lambda_4 > 0$. Then using $\log(x) \le \frac{\lambda_3}{2\lambda_4}x^2$ for $x$ sufficiently large, we have
\begin{equation}
\left( \sum_{n=0}^{N-1} (\alpha_n^\epl - \alpha_n)^2 \right)^{\frac12} \le Ce^{-\frac{\lambda_3}{2\period^2\epl^2}},
\end{equation}
for $\epl$ sufficiently small. Taking $\lambda_1 = \lambda_3/2$ and combining our estimates yields the desired result.
\end{proof}

\subsubsection{Extension to a general potential}

In this section, we consider the general case where the potential $V$ does not have a specific form. The main idea is to approximate the invariant measure $\rho$ using a mixture of Gaussians, and then apply the convergence result established in the previous section. Before presenting the main theorem, we also need to approximate the multiscale invariant measure $\rho^\epl$. The following result provides a way to do so and follows as a consequence of \cref{pro:Gaussian-Case}.

\begin{corollary} \label{cor:Gaussian-mixture}
Let $M\in\N$ and set $\rho^\epl = G^\epl$ and $\rho = G$, where
\begin{equation}
\begin{aligned}
G(x) &= \sum_{m=1}^M \theta_m \frac{1}{\sqrt{2\pi\sigma_m^2}}e^{-\frac{(x-\mu_m)^2}{2\sigma_m^2}}\quad\text{with}\quad \sum_{m=1}^M\theta_m = 1, \\
G^\epl(x) &= \sum_{m=1}^M \theta_m \frac{1}{\widetilde{Z}_m^\epl} e^{-\frac{(x-\mu_m)^2}{2\sigma_m^2}-\frac{1}{\sigma^2}p(\frac{x}{\epl})}\quad\text{with}\quad \widetilde{Z}_m^\epl = \int_\R e^{-\frac{(x-\mu_m)^2}{2\sigma_m^2}-\frac{1}{\sigma^2}p(\frac{x}{\epl})}\dd x.
\end{aligned}
\end{equation}
Finally, let $G_{N}^\epl$ be the projection of $G^\epl$ onto the span of the first $N$ Hermite functions. Then, under the hypotheses of \cref{pro:Gaussian-Case}, we have that $G_{N_k}^{\epl_k} \to G$ in $L^2(\R)$ as $k\to\infty$. Furthermore, there are constants $C,\lambda_1^{(m)},\lambda_2^{(m)}>0$ with $m=1,2,\dots,M$ such that
\begin{equation}
\norm{G_{N_k}^{\epl_k} - G}_{L^2(\R)} \le C\max_{m\in\{1,2,\dots,M\}} \left(e^{-\frac{\lambda_1^{(m)}}{\period^2\epl_k^2}} + e^{-\lambda_2^{(m)}N_k}\right).
\end{equation}
\end{corollary}
\begin{proof}
We will drop the subscript $k$ in $\epl_k$ and $N_k$. As in the proof of \cref{pro:Gaussian-Case}, by the triangle inequality and since $\psi_n$ forms an orthonormal basis, we have
\begin{equation} \label{eq:Gaussian-mixture-triangle}
\norm{G^\epl_N - G}_{L^2(\R)} \le \left( \sum_{n=0}^{N-1} (\alpha_n^\epl - \alpha_n)^2 \right)^{\frac12} + \left( \sum_{n=N}^\infty \alpha_n^2 \right)^{\frac12},
\end{equation}
where
\begin{equation}
\alpha_n^\epl = \int_\R G_N^\epl(x)\psi_n(x)\dd x,\quad\text{and}\quad \alpha_n = \int_\R G(x)\psi_n(x)\dd x.
\end{equation}
For the second sum, we have by the triangle inequality in $\ell^2$ that
\begin{equation}
\begin{split}
\left( \sum_{n=N}^\infty \alpha_n^2 \right)^{\frac12} 
&= \left(\sum_{n=N}^\infty \left(\sum_{m=1}^M \theta_m \int_\R \frac{1}{\sqrt{2\pi\sigma_m^2}}e^{-\frac{(x-\mu_m)^2}{2\sigma_m^2}}\psi_n(x)\dd x\right)^2\right)^{1/2}\\ 
&\le \sum_{m=1}^M \theta_m \left(\sum_{n=N}^\infty\left(\int_\R \frac{1}{\sqrt{2\pi\sigma_m^2}}e^{-\frac{(x-\mu_m)^2}{2\sigma_m^2}} \psi_n(x)\dd x\right)^2\right)^{1/2}.
\end{split}
\end{equation}
Thus, we have reduced the sum to the Gaussian base case and can apply \cref{lem:rate_alpha_n} to find that
\begin{equation}
\left( \sum_{n=N}^\infty \alpha_n^2 \right)^{\frac12} \le C\sum_{m=1}^M \theta_m e^{-\lambda_2^{(m)}N}
\end{equation}
for some $C,\lambda_2^{(m)}>0$ for $m=1,2,\dots,M$. 

Then for the first sum in \eqref{eq:Gaussian-mixture-triangle}, we note that
\begin{equation}
\begin{split}
\abs{\alpha_n^\epl - \alpha_n} &= \abs{\int_\R \psi_n(x) (G^\epl_{N}(x) - G(x)) \dd x} \\
&= \abs{\int_\R \psi_n(x) \left( \sum_{m=1}^M \theta_m \frac1{\widetilde Z^\epl_m} e^{-\frac{(x - \mu_m)^2}{2\sigma_m^2} - \frac{1}{\sigma^2} p \left( \frac{x}\epl \right)} - \sum_{m=1}^M \theta_m \frac1{\sqrt{2\pi \sigma_m^2}} e^{-\frac{(x - \mu_m)^2}{2\sigma_m^2}} \right)} \\
&\le \sum_{m=1}^M \theta_m \abs{\int_\R \psi_n(x) \left( \frac1{\widetilde Z^\epl_m} e^{-\frac{(x - \mu_m)^2}{2\sigma_m^2} - \frac{1}{\sigma_m^2} \widetilde{p}_m \left( \frac{x}\epl \right)} - \frac1{\sqrt{2\pi \sigma_m^2}} e^{-\frac{(x - \mu_m)^2}{2\sigma_m^2}} \right)},
\end{split}
\end{equation}
where $\widetilde{p}_m(y) = \frac{\sigma_m^2}{\sigma^2}p(y)$,
which reduces the problem to the Gaussian base case. In particular, by the triangle inequality and from the proof of \cref{pro:Gaussian-Case}, we see that there are $C,\lambda_1^{(m)}>0$ for $m=1,2,\dots,M$ such that
\begin{equation}
\left(\sum_{n=0}^{N-1} (\alpha_n^\epl - \alpha_n)^2\right)^{1/2} \le C\sum_{m=1}^M \theta_m e^{-\frac{\lambda_1^{(m)}}{\period^2\epl^2}}.
\end{equation}
Combining our estimates yields
\begin{equation}
\norm{G_{N_k}^{\epl_k} - G}_{L^2(\R)} \le C\sum_{m=1}^M \theta_m \left(e^{-\frac{\lambda_1^{(m)}}{\period^2\epl_k^2}} + e^{-\lambda_2^{(m)}N_k}\right) \le C\max_{m\in\{1,2,\dots,M\}} \left(e^{-\frac{\lambda_1^{(m)}}{\period^2\epl_k^2}} + e^{-\lambda_2^{(m)}N_k}\right),
\end{equation}
as desired.
\end{proof}

Using the Gaussian mixture approximation theory, in the next theorem we prove the convergence of $\rho^\epl_N$ to $\rho$ as $\epl \to 0$ and $N = N(\epl) \to \infty$.

\begin{proposition} \label{pro:General-Case}
Let $\{\epl_k\}_k$ and $N_k$ satisfy the assumptions of \cref{pro:Gaussian-Case}. Then 
\begin{equation}
\lim_{k\to\infty} \norm{\rho_{N_k}^{\epl_k} - \rho}_{L^2(\R)} = 0.
\end{equation}
\end{proposition}
\begin{proof}
Due to \cite[Theorem 5(c)]{NNC20}, for all $\delta > 0$ there exists a function $G_\delta$, which is a mixture of $M = M(\delta)$ Gaussians, of the form
\begin{equation}
G_\delta(x) = \sum_{m=1}^M \theta_m \frac1{\sqrt{2\pi \sigma_m^2}} e^{-\frac{(x - \mu_m)^2}{2\sigma_m^2}} \qquad \text{with} \qquad \sum_{m=1}^M \theta_m = 1,
\end{equation}
such that
\begin{equation}
\norm{\rho - G_\delta}_{L^2(\R)} < \delta.
\end{equation}
Let us now define the corresponding multiscale function $G_\delta^\epl$ as
\begin{equation}
G_\delta^\epl(x) = \sum_{m=1}^M \theta_m \frac1{\widetilde Z^\epl_m} e^{-\frac{(x - \mu_m)^2}{2\sigma_m^2} - \frac{1}{\sigma^2} p \left( \frac{x}\epl \right)} \qquad \text{with} \qquad \widetilde Z^\epl_m = \int_\R e^{-\frac{(x - \mu_m)^2}{2\sigma_m^2} - \frac{1}{\sigma^2} p \left( \frac{x}\epl \right)} \dd x.
\end{equation}
Then, we have
\begin{equation}
\norm{\rho^\epl - G_\delta^\epl}_{L^2(\R)}^2 = \int_\R (\rho^\epl(x) - G_\delta^\epl(x))^2 \dd x = \int_\R e^{-\frac{2}{\sigma^2} p \left( \frac{x}\epl \right)} \left( \frac1{Z^\epl} e^{-\frac{1}{\sigma^2} V(x)} - \sum_{m=1}^M \theta_m \frac1{\widetilde Z^\epl_m} e^{-\frac{(x - \mu_m)^2}{2\sigma_m^2}} \right)^2 \dd x,
\end{equation}
which implies
\begin{equation}
\begin{split}
\norm{\rho^\epl - G_\delta^\epl}_{L^2(\R)}^2 &\le 2 \left(\frac{Z}{Z^\epl} \right)^2 \int_\R e^{-\frac{2}{\sigma^2} p \left( \frac{x}\epl \right)} \left( \rho(x) - G_\delta(x) \right)^2 \dd x \\
&\quad + 2 \int_\R e^{-\frac{2}{\sigma^2} p \left( \frac{x}\epl \right)} \left( \sum_{m=1}^M \theta_m \left( \frac{Z}{Z^\epl} - \frac{\sqrt{2\pi\sigma_m^2}}{\widetilde Z^\epl_m} \right) \frac1{\sqrt{2\pi \sigma_m^2}} e^{-\frac{(x - \mu_m)^2}{2\sigma_m^2}} \right)^2 \dd x, \\
\end{split}
\end{equation}
which, due to the Jensen's inequality, gives
\begin{equation}
\begin{split}
\norm{\rho^\epl - G_\delta^\epl}_{L^2(\R)}^2 &\le 2 \left(\frac{Z}{Z^\epl} \right)^2 \int_\R e^{-\frac{2}{\sigma^2} p \left( \frac{x}\epl \right)} \left( \rho(x) - G_\delta(x) \right)^2 \dd x \\
&\quad + 2 \sum_{m=1}^M \theta_m \abs{\frac{Z}{Z^\epl} - \frac{\sqrt{2\pi\sigma_m^2}}{\widetilde Z^\epl_m}}^2 \frac1{2\pi \sigma_m^2} \int_\R e^{-\frac{(x - \mu_m)^2}{\sigma_m^2} - \frac{2}{\sigma^2} p \left( \frac{x}\epl \right)} \dd x.
\end{split}
\end{equation}
Recalling the definition of $Z^\epl$ and $\widetilde Z^\epl_m$
\begin{equation}
Z^\epl = \int_\R e^{- \frac{1}{\sigma^2} \left( V(x) + p \left( \frac{x}\epl \right) \right)} \dd x, \qquad \widetilde Z^\epl_m = \int_\R e^{-\frac{(x - \mu_m)^2}{2\sigma_m^2} - \frac{1}{\sigma^2} p \left( \frac{x}\epl \right)} \dd x,
\end{equation}
and using the notation 
\begin{equation}
\Pi_2 = \int_0^\period e^{- \frac{2}{\sigma^2} p(y)} \dd y,
\end{equation}
we deduce that
\begin{equation}
\lim_{\epl\to0} Z^\epl = \frac{Z\Pi}{\period}, \quad \lim_{\epl\to0} \widetilde Z^\epl_m = \frac{\sqrt{2\pi \sigma_m^2}\Pi}{\period}, \quad \lim_{\epl\to0} \int_\R e^{-\frac{2}{\sigma^2} p \left( \frac{x}\epl \right)} \left( \rho(x) - G_\delta(x) \right)^2 \dd x = \frac{\Pi_2}{\period} \norm{\rho - G_\delta}_{L^2(\R)}^2.
\end{equation}
Therefore, we obtain
\begin{equation}
\limsup_{k\to\infty} \norm{\rho^{\epl_k} - G_\delta^{\epl_k}}_{L^2(\R)}^2 \le \frac{2\Pi_2\period}{\Pi^2} \norm{\rho - G_\delta}_{L^2(\R)}^2,
\end{equation}
which implies
\begin{equation}
\limsup_{k\to\infty} \norm{\rho^{\epl_k} - G_\delta^{\epl_k}}_{L^2(\R)} \le \frac{\sqrt{2\Pi_2\period}}{\Pi} \norm{\rho - G_\delta}_{L^2(\R)} \le \frac{\sqrt{2\Pi_2\period}}{\Pi} \delta.
\end{equation}
Now let $G^{\epl_k}_{\delta,N_k}$ be the projection of $G^{\epl_k}_\delta$ onto the span of the first $N_k$ Hermite functions. Using the Gaussian mixture approximation, the triangle inequality, and the fact that the projection has smaller norm, we get for all $\delta > 0$
\begin{equation}
\begin{split}
\norm{\rho^{\epl_k}_{N_k} - \rho}_{L^2(\R)} &\le \norm{\rho^{\epl_k}_{N_k} - G^{\epl_k}_{\delta,N_k}}_{L^2(\R)} + \norm{G^{\epl_k}_{\delta,N_k} - G_\delta}_{L^2(\R)} + \norm{G_\delta - \rho}_{L^2(\R)} \\
&\le \norm{\rho^{\epl_k} - G^{\epl_k}_\delta}_{L^2(\R)} + \norm{G^{\epl_k}_{\delta,N_k} - G_\delta}_{L^2(\R)} + \delta,
\end{split}
\end{equation}
which implies
\begin{equation}
\begin{split}
\limsup_{k\to\infty} \norm{\rho^{\epl_k}_{N_k} - \rho}_{L^2(\R)} &\le \limsup_{k\to\infty} \norm{\rho^{\epl_k} - G_\delta^{\epl_k}}_{L^2(\R)} + \limsup_{k\to\infty} \norm{G^{\epl_k}_{\delta,N_k} - G_\delta}_{L^2(\R)} + \delta \\
&\le \limsup_{k\to\infty} \norm{G^{\epl_k}_{\delta,N_k} - G_\delta}_{L^2(\R)} + \left( \frac{\sqrt{2\Pi_2\period}}{\Pi} + 1 \right) \delta.
\end{split}
\end{equation}
Then by \cref{cor:Gaussian-mixture} we have for any $\delta > 0$
\begin{equation}
\lim_{k\to\infty} \norm{G^{\epl_k}_{\delta,N_k} - G_\delta}_{L^2(\R)} = 0,
\end{equation}
so we obtain
\begin{equation}
0 \le \liminf_{k\to\infty} \norm{\rho^{\epl_k}_{N_k} - \rho}_{L^2(\R)} \le \limsup_{k\to\infty} \norm{\rho^{\epl_k}_{N_k} - \rho}_{L^2(\R)} \le \left( \frac{\sqrt{2\Pi_2\period}}{\Pi} + 1 \right) \delta,
\end{equation}
which yields
\begin{equation}
\lim_{k\to\infty} \norm{\rho^{\epl_k}_{N_k} - \rho}_{L^2(\R)} = 0,
\end{equation}
which is the desired result in the general setting. 
\end{proof}

\begin{remark} \label{rem:subsequences}
The subsequences $\epl_k$ and $N_k$ in \cref{pro:General-Case} are introduced to present the most general form of the result, where $N_k$ is not required to follow a specific growth rate relative to $\epl_k$, but only needs to satisfy the upper bound given in equation \eqref{eq:bound_N}. In particular, as long as $\epl_k \to 0$, $N_k \to \infty$, and $N_k$ respects the constraint in \eqref{eq:bound_N}, \cref{pro:General-Case} holds for any such subsequence. An alternative formulation is to fix the number of Fourier modes $N$ to be as large as allowed by the constraint, i.e., 
\begin{equation}
\bar N = \bar N(\epl) =  \left\lfloor \frac{\pi^2}{\gamma\period^2\epl^2}\right\rfloor,
\end{equation}
which yields the convergence
\begin{equation}
\lim_{\epl\to0} \norm{\rho^\epl_{\bar N(\epl)} - \rho}_{L^2(\R)} = 0.
\end{equation}
Conceptually, \cref{pro:General-Case} means that ``$\epl$ must vanish before $N$ diverges''. Specifically, from the weak convergence of $\rho^\epl$ to $\rho$, it follows that
\begin{equation} \label{eq:iterated_limit_convergence}
\lim_{N\to\infty} \lim_{\epl\to0} \norm{\rho^\epl_N - \rho}_{L^2(\R)} = 0,
\end{equation}
but this result is weaker than \cref{pro:General-Case} and does not provide a quantitative guideline for selecting $N$. Moreover, an additional confirmation of the fact that $N$ should diverge more slowly than $\epl$ vanishes is that we cannot exchange the order of the two limits in equation \eqref{eq:iterated_limit_convergence} because
\begin{equation}
\lim_{\epl\to0} \lim_{N\to\infty} \norm{\rho^\epl_N - \rho}_{L^2(\R)} = \lim_{\epl\to0} \norm{\rho^\epl - \rho}_{L^2(\R)},
\end{equation}
and the last limit does not exist, since $\rho^\epl$ converges to $\rho$ only weakly in $L^2(\R)$, not strongly.
\end{remark}

\begin{remark}
The proof of \cref{pro:General-Case} shows that the approximation $\rho^\epl_N$ converges exponentially fast to the density $\rho$, with rate $\mathcal O(e^{- 1/\epl^2})$ given by \cref{cor:Gaussian-mixture}, up to an arbitrarily small error $\delta>0$. On the other hand, following up on \cref{rem:polynomial_rate}, and using the polynomial rate in \cite[Corollary A.4]{BKP25} and properties of Hermite functions, it is possible to prove that 
\begin{equation}
\norm{\rho_N^\epl - \rho}_{L^2(\R)} \le \epl^\ell N^{\frac\ell2 + 1}.
\end{equation}
Setting $N = \epl^{-\tau}$ with $\tau > 0$, this becomes
\begin{equation}
\norm{\rho_N^\epl - \rho}_{L^2(\R)} \le \epl^{\ell - \tau \left( \frac\ell2 + 1 \right)},
\end{equation}
where the right-hand side converges to zero as $\epl\to0$ provided $\tau < 2\ell/(\ell+2)$, with polynomial rate. Therefore, as $\ell \to \infty$, $\tau$ can be taken arbitrarily close to 2, but never equal to 2. In contrast, our result achieves convergence with $N = \epl^{-2}$, corresponding to a larger number of Fourier modes.
\end{remark}

\cref{pro:General-Case} shows that the deterministic term $\mathfrak q_N^\epl$ in equation \eqref{eq:deterministic_stochastic_terms} vanishes as $\epl\to0$, provided that $N \to \infty$ is chosen appropriately. In the next section, we turn our attention to the stochastic term $\mathcal Q_N^{T,\epl}$.

\subsection{Analysis of the stochastic component $\mathcal Q_N^{T,\epl}$}

Let us first introduce some quantities. Define the scale functions for $x \in \R$
\begin{equation} \label{eq_A-harmonic_function}
f_\epl(x) \defeq \int_0^x \exp \left( \frac{1}{\sigma^2} \left( V(z) + p \left( \frac{z}{\epl} \right) \right) \right) \dd z, \qquad
f(x) \defeq \int_0^x \exp \left( \frac{1}{\sigma^2} V(z) \right) \dd z,
\end{equation}
which belong to $C^2(\R)$ under \cref{as:potentials} and are harmonic with respect to the differential operators 
\begin{equation} \label{eq:generators}
\mathcal{A}_\epl \defeq \left[ -V' - \frac{1}{\epl} p' \left( \frac{\cdot}{\epl} \right) \right] \partial_x + \sigma^2 \partial_{xx}, \qquad \mathcal{A} \defeq - \mathcal K V' \partial_x + \mathcal K \sigma^2 \partial_{xx},
\end{equation}
respectively, i.e., $\mathcal{A}_\epl f_\epl = \mathcal{A} f = 0$. Setting $\xi^\epl \defeq f_\epl \circ X^\epl$ and $\xi \defeq f \circ X$ and applying Itô's formula we can eliminate the drift terms and obtain the transformed SDEs
\begin{align}
\label{eq:transformed_sde_eps}
\d \xi^\epl_t &= \frac{1}{m^\epl(\xi^\epl_t)} \dd W_t, \qquad t \in [0, T], \qquad \xi^\epl_0 = f_\epl(x_0), \\
\label{eq:transformed_sde}
\d \xi_t &= \frac{1}{m(\xi_t)} \dd W_t, \qquad t \in [0, T], \qquad \xi_0 = f(x_0),
\end{align}
where
\begin{align}
\label{eq:m_def_eps}
m^\epl(x) &\defeq \frac{1}{\sqrt{2 \sigma^2} f'_\epl(g_\epl(x))}, \qquad g_\epl(x) \defeq f_\epl^{-1}(x), \qquad x \in \R, \\
\label{eq:m_def}
m(x) &\defeq \frac{1}{\sqrt{2 \mathcal K \sigma^2} f'(g(x))}, \qquad g(x) \defeq f^{-1}(x), \qquad x \in \R.
\end{align}
Note that $f_\epl', f' > 0$, so that the inverse functions $g_\epl, g$ are well-defined. Moreover, the invariant densities $\rho^\epl, \rho$ in \eqref{eq:invariant_density} can be written in terms of $f_\epl, f$ as
\begin{equation}
\rho^\epl(x) = \frac1{Z^\epl f_\epl'(x)}, \qquad \rho(x) = \frac1{Z f'(x)}.
\end{equation}
In the following, $\Pr^{x_0}$ (and analogously for $\E^{x_0}$) denotes $\Pr^{x_0} \circ (X^\epl)^{-1} = \Pr((X^{\epl,x_0})^{-1}(\cdot))$, where $X^{\epl, x_0}$ is the unique strong solution to \eqref{eq:SDE_multiscale} with $X_0^\epl = x_0$.

As it will become apparent later on, the convergence of $\mathcal Q_N^{T,\epl}$ in \eqref{eq:deterministic_stochastic_terms} to zero builds on the validity of an ``$\epl$-dependent'' mean ergodic theorem, i.e., for any initial condition $x_0$ it must hold 
\begin{equation} \label{eq:eps_mean_ergodic_theorem}
\E^{x_0} \left[ \abs{\frac{1}{T_\epl} \int_0^{T_\epl} \varphi(X_t^\epl) \dd t - \int_\R \varphi(y) \rho^\epl(y) \dd y}^2 \right] \to 0, \qquad \text{as } \epl\to0,
\end{equation}
where $\varphi \colon \R \to \R$ is a bounded, measurable function and, crucially, $T_\epl = \mathcal{O}(\epl^{-\zeta})$, with $\zeta>0$, as $\epl\to0$; cf.~\cite{BK25}.

\subsubsection{An $\epl$-dependent mean ergodic theorem with rates}

In order to prove the ergodic theorem \eqref{eq:eps_mean_ergodic_theorem}, we will follow along the lines of the original proof in \cite[Section 18]{GiS72}, where the result is derived for $T \to \infty$ in the single-scale setting, i.e., it does not depend on $\epl$. Many of the main ideas are contained in the aforementioned work, but, since we require specific convergence rates in $\epl$ and the approximation of $X_\epl$ to $X$ only prevails in a weak sense, we need to modify the approach to our setting and extend the proofs.

First, notice that an integral transformation yields
\begin{equation} \label{eq:transformation_MET}
\frac{1}{T_\epl} \int_0^{T_\epl} \varphi(X_t^\epl) \dd t - \frac{1}{Z^\epl} \int_\R \varphi(x) \rho^\epl(x) \dd x = \frac{1}{T_\epl} \int_0^{T_\epl} \varphi(g_\epl(\xi^\epl_t))  \dd t - \frac{1}{C_{m^\epl}} \int_\R \varphi(g_\epl(x)) m^\epl(x)^2 \dd x,
\end{equation}
where $\xi_t^\epl, g_\epl, m^\epl$ are defined in equations \eqref{eq:transformed_sde_eps}, \eqref{eq:m_def_eps}, and $C_{m^\epl}$ is given by
\begin{equation}
C_{m^\epl} \defeq \int_\R m^\epl(x)^2 \dd x.
\end{equation}
Therefore, it is sufficient to prove \eqref{eq:eps_mean_ergodic_theorem} for the solution of equation \eqref{eq:transformed_sde_eps}.

We start with a lemma that summarizes some properties that will be used throughout what follows and is crucial to get the convergence rates.

\begin{lemma} \label{lem:g_properties}
Under \cref{as:potentials}, the following statements hold:
\begin{enumerate}[leftmargin=*,label=\roman*)]
\item There exist constants $\Cm{1}, \Cm{2}>0$ such that $\Cm{1} \le C_{m^\epl} \le \Cm{2}$ for all $\epl > 0$.
\item There exist constants $C_1^-, C_1^+, C_2^-, C_2^+ > 0$ such that for sufficiently large $\abs{x}$ with $\abs{x} \to \infty$
\begin{equation} \label{eq:f_eps_bounds}
\begin{aligned}
\frac{C_1^+}{\abs{x}} \exp \left(\frac{\beta}{2 \sigma^2} x^2 \right) &\le \inf_{\epl > 0} f_\epl(x) \le \sup_{\epl > 0} f_\epl(x) \le \frac{C_2^+}{\abs{x}} \exp \left( \frac{L_V + \abs{V'(0)}}{2 \sigma^2} x^2 \right), &x > 0, \\
\frac{C_1^-}{\abs{x}} \exp \left( \frac{L_V + \abs{V'(0)}}{2 \sigma^2} x^2 \right) &\le \inf_{\epl > 0} f_\epl(x) \le \sup_{\epl > 0} f_\epl(x) \le \frac{C_2^-}{\abs{x}} \exp \left(\frac{\beta}{2 \sigma^2} x^2 \right), &x < 0.
\end{aligned}
\end{equation}
\item Let $w(x, y) \defeq -W_{-1}(-x/y^2)$ for $x > 0$ and $y \in \R$ with $\abs{y} \ge \sqrt{ex}$, where $W_{-1}$ is the part of the principal branch of the Lambert W function defined on $[-1/e, 0)$, which is strictly decreasing and satisfies $W_{-1}(-1/e) = -1$ and $\lim_{y \to 0^+} W_{-1}(y) = -\infty$. Then, for sufficiently large $M$ with $M \to \infty$
\begin{equation} \label{eq:int_m_eps_bounds_BigO}
\begin{aligned}
\int_{\abs{y}>M} m^\epl(y)^2 \dd y &= \BigO{\frac{\exp \left(-\frac{r}{2 l} w(r^\wedge, M) \right)}{\sqrt{w(r^\wedge, M)}}}, &r^\wedge \defeq \frac{r}{C_1^- \wedge C_2^+}, \\
\int_{-M}^M \abs{y} m^\epl(y)^2 \dd y &= \BigO{\sqrt{w(r^\vee, M)}}, &r^\vee \defeq \frac{r}{C_1^+ \vee C_2^-},
\end{aligned}
\end{equation}
where $a \wedge b := \min\{a, b\}$, $a \vee b := \max\{a, b\}$ for $a, b \in \R$ and
\begin{equation} \label{eq:lr_def}
l \defeq \frac{L_V + \abs{V'(0)}}{\sigma^2}, \qquad r \defeq \frac{\beta}{\sigma^2}.
\end{equation}
\end{enumerate}
\end{lemma}
\begin{proof}
The proof is divided into three steps, corresponding to the three claims.

\textbf{Step i).} First recall that by \cref{as:potentials} iii) there exist $\beta > 0$ and $R \ge 1$ such that for all $x \in \R$ with $\abs{x} \ge R$
\begin{equation} \label{eq:recurrence_cond_mod}
- \sign(x) V'(x) \le -\beta \abs{x},
\end{equation}
which implies
\begin{equation}
0 < C_m \defeq \int_\R m(x)^2 \dd x = \frac{1}{2 \mathcal K \sigma^2} \int_\R \frac{1}{f'(x)} \dd x = \frac{1}{2 \mathcal K \sigma^2} \int_\R \exp \left( -\frac{1}{\sigma^2} \int_0^z V'(y) \dd y \right) \dd x < \infty.
\end{equation}
The continuity and periodicity of $p$ imply for every $x \in \R$
\begin{equation} \label{eq:f'_eps_bounds}
\Cf{1} f'(x) \le f_\epl'(x) \le \Cf{2} f'(x),
\end{equation}
where
\begin{equation}
\Cf{1} \defeq \exp \left( \frac{\min_{y \in [0, L]} p(y)}{\sigma^2} \right), \qquad \Cf{2} \defeq \exp \left( \frac{\max_{y \in [0, L]} p(y)}{\sigma^2} \right).
\end{equation}
Then, by making a change of variables, we have
\begin{equation}
C_{m^\epl} = \int_\R m^\epl(x)^2 \dd x = \int_\R \left( \frac{1}{\sqrt{2 \sigma^2} f'_\epl(g_\epl(x))} \right)^2 \dd x = \frac{1}{2 \sigma^2} \int_\R \frac{1}{f'_\epl(x)} \dd x,
\end{equation}
which, due to equation \eqref{eq:f'_eps_bounds}, gives
\begin{equation}
C_{m^\epl} \le \frac{1}{2\sigma^2 \Cf{1}} \int_\R \frac{1}{f'(x)} \dd x \eqdef \Cm{2}.
\end{equation}
A lower bound can be established in the same manner.

\textbf{Step ii).} For this asymptotic result we consider $x \in \R$ sufficiently large such that \eqref{eq:recurrence_cond_mod} holds. By \eqref{eq:f'_eps_bounds} it suffices to analyze $f$ for these $x$. Let $x \ge R$ first, then
\begin{equation}
f(x) = f(R) + f'(R) \int_R^x \exp \left( \int_R^y \frac{V'(z)}{\sigma^2} \dd z \right) \dd y.
\end{equation}
On the one hand, by the global Lipschitz assumption on $V'$ we have that
\begin{equation}
\int_R^y \frac{V'(z)}{\sigma^2} \dd z \le \frac{L_V + \abs{V'(0)}}{2 \sigma^2} (y^2 - R^2),
\end{equation}
and, on the other hand, the recurrence property described by \eqref{eq:recurrence_cond_mod} gives
\begin{equation}
\int_R^y \frac{V'(z)}{\sigma^2} \dd z \ge \frac{\beta}{2 \sigma^2} (y^2 - R^2).
\end{equation}
Consecutive applications of L'Hôpital's rule show that for any $x_0 \in \R$ and $\eta > 0$
\begin{equation} \label{eq:L'Hospital_appl_+}
\lim_{x \to \infty} \frac{2 \eta x \int_{x_0}^x \exp(\eta y^2) \dd y}{\exp(\eta x^2)} = 1.
\end{equation}
This gives us for sufficiently large $x \ge R$ and some constants $C_1, C_2 > 0$
\begin{equation}
\begin{aligned}
f(x) &\le f(R) + f'(R) \exp \left( -\frac{\beta}{2 \sigma^2} R^2 \right) \int_R^x \exp \left( \frac{L_V + |V'(0)|}{2 \sigma^2} y^2 \right) \dd y \\[0.25cm]
&\le C_1 + C_2 x^{-1} \exp \left( \frac{L_V + |V'(0)|}{2 \sigma^2} x^2 \right),
\end{aligned}
\end{equation}
and similarly for some constants $C_3, C_4 > 0$
\begin{equation}
\begin{aligned}
f(x) &\ge f(R) + f'(R) \exp \left( -\frac{\beta}{2 \sigma^2} R^2 \right) \int_R^x \exp \left(\frac{\beta}{2 \sigma^2} y^2 \right) \dd y \\[0.25cm]
&\ge C_3 + C_4 x^{-1} \exp \left(\frac{\beta}{2 \sigma^2} x^2 \right) \ge C_4 x^{-1} \exp \left(\frac{\beta}{2 \sigma^2} x^2 \right).
\end{aligned}
\end{equation}
Using equation \eqref{eq:f'_eps_bounds}, we thus have for sufficiently large $x$ with $x \to \infty$
\begin{equation} \label{eq:f_eps_estimate_+}
\frac{C_1^+}{x} \exp \left(\frac{\beta}{2 \sigma^2} x^2 \right) \le \inf_{\epl > 0} f_\epl(x) \le \sup_{\epl > 0} f_\epl(x) \le \frac{C_2^+}{x} \exp \left( \frac{L_V + |V'(0)|}{2 \sigma^2} x^2 \right).
\end{equation}
The case $x \le -R$ works almost the same due to the symmetry of the global Lipschitz and recurrence condition. One only needs to be careful with the signs, e.g., instead of equation \eqref{eq:L'Hospital_appl_+}, one proves
\begin{equation} \label{eq:L'Hospital_appl_-}
\lim_{x \to -\infty} \frac{- 2 \eta x \int_x^{x_0} \exp(\eta y^2) \dd y}{\exp(\eta x^2)} = 1.
\end{equation}
Eventually we arrive at the statement that for sufficiently large $|x|$ with $x \to -\infty$
\begin{equation}    \label{eq:f_eps_estimate_-}
\frac{C_1^-}{-x} \exp \left( \frac{L_V + |V'(0)|}{2 \sigma^2} x^2 \right) \le \inf_{\epl > 0} f_\epl(x) \le \sup_{\epl > 0} f_\epl(x) \le \frac{C_2^-}{-x} \exp \left(\frac{\beta}{2 \sigma^2} x^2 \right).
\end{equation}

\textbf{Step iii)}. As it will become evident later in the proof, we first need to invert the left-hand side and the right-hand side of the inequalities \eqref{eq:f_eps_bounds}, and this can be done in terms of the Lambert W function as follows, see \cite{Cha13} and the references therein for a definition and properties of this special function. For any $C, \eta > 0$ and $y \in \R$ with $\abs{y} \ge \sqrt{2e \eta/C}$ we have
\begin{equation}
y = \frac{C}{x} \exp(\eta x^2) \qquad \Longleftrightarrow \qquad -\frac{2 \eta}{C y^2} = -2 \eta x^2 \exp(-2 \eta x^2) \eqdef a \exp(a), 
\end{equation}
where $a \defeq -2 \eta x^2$. The solution to this last equation is exactly given by the Lambert W function $W_{-1}$, that is,
\begin{equation}
a = W_{-1}(-2\eta/C y^2).
\end{equation}
Hence, we obtain
\begin{equation}
x = \pm \sqrt{-\frac{1}{2 \eta} W_{-1}\left(-\frac{2\eta}{C y^2}\right)}.
\end{equation}
We can now invert the inequalities \eqref{eq:f_eps_bounds} for sufficiently large $|x|$ with $\abs{x} \to \infty$. Notice that the thresholds beyond which the inequalities \eqref{eq:f_eps_bounds} hold do not depend on $\epl$, and due to \eqref{eq:f'_eps_bounds} we can therefore also set $\epl$-dependent thresholds beyond which the inequalities for the inverse $g_\epl = f_\epl^{-1}$ hold. Thus, inverting the asymptotic bounds yields for sufficiently large $|y|$ with $\abs{y} \to \infty$
\begin{equation} \label{eq:g_eps_bounds}
\begin{aligned}
\sqrt{\frac{w(l/C_2^+, y)}{l}} &\le \inf_{\epl > 0} g_\epl(y) \le \sup_{\epl > 0} g_\epl(y) \le \sqrt{\frac{w(r/C_1^+, y)}{r}}, &y>0, \\
-\sqrt{\frac{w(r/C_2^-, y)}{r}} &\le \inf_{\epl > 0} g_\epl(y) \le \sup_{\epl > 0} g_\epl(y) \le -\sqrt{\frac{w(l/C_1^-, y)}{l}}, &y<0,
\end{aligned}
\end{equation}
where $l$ and $r$ are defined in equation \eqref{eq:lr_def}. We can now prove the two claims in \eqref{eq:int_m_eps_bounds_BigO}. Let $M>0$ be sufficiently large such that all the inequalities in equation \eqref{eq:g_eps_bounds} hold for $M$ and $-M$, respectively. Similarly as before, we can estimate
\begin{equation}
\int_{\abs{y} > M} m^\epl(y)^2 \dd y \le \frac{1}{2 \sigma^2 \Cf{1}} \left[ \int_{g_\epl(M)}^\infty \frac{\d y}{f'(y)} + \int_\infty^{g_\epl(-M)} \frac{\d y}{f'(y)} \right].
\end{equation}
The first integral can be upper bounded using equation \eqref{eq:recurrence_cond_mod} with some constant $C_1 > 0$ (which may change from inequality to inequality)
\begin{equation}
\int_{g_\epl(M)}^\infty \frac{\d y}{f'(y)} \le C_1  \int_{g_\epl(M)}^\infty \exp \left( -\frac{\beta}{2 \sigma^2} x^2 \right) \dd y \le C_1 \frac{\exp \left(-\frac{\beta}{2 \sigma^2} g_\epl(M)^2 \right)}{g_\epl(M)},
\end{equation}
where we employed a standard inequality for the tails of a Gaussian integral in the last step. Using equation \eqref{eq:g_eps_bounds}, for $M \to \infty$ we get
\begin{equation}
\int_{g_\epl(M)}^\infty \frac{\d y}{f'(y)} = \BigO{\frac{\exp \left(-\frac{r}{2 l} w(r/C_2^+, M) \right)}{\sqrt{w(r/C_2^+, M)}}}.
\end{equation}
Using the symmetry and equation \eqref{eq:g_eps_bounds} again, we similarly obtain
\begin{equation}
\int_\infty^{g_\epl(-M)} \frac{\d y}{f'(y)} = \BigO{\frac{\exp \left(-\frac{r}{2 l} w(r/C_1^-, M) \right)}{\sqrt{w(r/C_1^-, M)}}}.
\end{equation}
Combining these two estimates and due to the monotonicity of $W_{-1}$, we get the first claim. The second claim is obtained after noting that due to equation \eqref{eq:recurrence_cond_mod} we have
\begin{equation}
\lim_{x \to \infty} \frac{f(x)}{f'(x)} = - \lim_{x \to -\infty} \frac{f(x)}{f'(x)} = 0.
\end{equation}
This, in turn, gives for sufficiently large $M>0$
\begin{equation}
\int_{-M}^M \abs{y} m^\epl(y) \dd y = \BigO{\int_{g_\epl(-M)}^{g_\epl(M)} \frac{\abs{f(y)}}{f'(y)} \dd y} = \BigO{g_\epl(M) - g_\epl(-M)},
\end{equation}
which implies
\begin{equation}
\int_{-M}^M \abs{y} m^\epl(y) \dd y = \BigO{\sqrt{w(r/C_1^+, M)} + \sqrt{w(r/C_2^-, M)}} = \BigO{\sqrt{w(r/C_1^+ \vee C_2^-, M)}},
\end{equation}
where we used the monotonicity of $W_{-1}$ in the last step.
\end{proof}

\begin{remark} \label{rem:lambert_w_bounds}
In the following Lemma \ref{lem:chi_eps_properties} we need the bounds derived in \cite[Theorem 1]{Cha13}, namely for $u > 0$
\begin{equation}
-1 - \sqrt{2u} - u < W_{-1}(-\exp(-u-1)) < -1 - \sqrt{2u} - \frac{2}{3}u.
\end{equation}
Then, identifying $u \defeq \log(M^2/\Tilde{r})-1$ with $\Tilde{r} \in \{ r^\wedge, r^\vee \}$ and $M > 0$ sufficiently large, we can establish the following estimates for the right-hand sides of equation \eqref{eq:int_m_eps_bounds_BigO} 
\begin{equation}
\begin{aligned}
\frac{\exp \left(-\frac{r}{2 l} w(r^\wedge, M) \right)}{\sqrt{w(r^\wedge, M)}} = \BigO{\frac{1}{\log(M) M^{r/l}}}, \qquad \sqrt{w(r^\vee, M)} = \BigO{\sqrt{\log(M)}}.
\end{aligned}
\end{equation}
Therefore, for $t_0 \ge 2 \sqrt{e r^\wedge}$, we obtain
\begin{equation}
\int_{t_0}^{T} \frac{\exp \left(-\frac{r}{2 l} w(r^\wedge, t/2) \right)}{\sqrt{w(r^\wedge, t/2)}} \dd t =
\begin{cases}
\BigO{1},         & \text{if } r > l, \\
\BigO{\log(T)},   & \text{if } r = l, \\
\BigO{T^{1-r/l}}, & \text{if } r < l,
\end{cases}
\end{equation}
which will be used in the proof of the mean ergodic theorem in Proposition \ref{pro:eps_mean_ergodic_theorem}.
\end{remark}

The following two results are needed for establishing the mean ergodic theorem. Since $\epl > 0$ is fixed in this context, the analysis from \cite[Section 18]{GiS72} applies without modification, and therefore the proofs of \cref{lem:stationarity,lem:semigroup_ergodicity} below are omitted.

\begin{lemma} \label{lem:stationarity}
Let $\epl>0$. For any measurable function $\varphi \colon \R \to \R$ such that
\begin{equation}
\int_\R \abs{\varphi(y)} m^\epl(y)^2 \dd y < \infty,
\end{equation}
it holds for $t>0$
\begin{equation} \label{eq:stationarity}
\int_\R \E^x \left[ \varphi(\xi^x_\epl(t)) \right] m^\epl(y)^2 \dd y = \int_\R \varphi(y) m^\epl(y)^2 \dd y.
\end{equation}
\end{lemma}

\begin{lemma} \label{lem:semigroup_ergodicity}
It holds for any $x \in \R$ and $\epl > 0$
\begin{equation} \label{eq:semigroup_ergodicity}
\begin{aligned}
&\abs{ \frac{1}{T_\epl} \int_0^{T_\epl} \E^x \varphi(\xi^\epl_t) \dd t - \frac{1}{C_{m^\epl}} \int_\R \varphi(y) m^\epl(y)^2 \dd y} \\
&\hspace{2cm} \le \frac{2 \norm{\varphi}_\infty }{\Cm{1}} \left[ \int_{\abs{x-y}>T_{\epl}} m^\epl(y)^2 \dd y + \frac{2 \Cm{2}}{T_\epl} \int_{\abs{x-y} \le T_{\epl}} \abs{x-y} m^\epl(y)^2 \dd y \right].
\end{aligned}
\end{equation}
\end{lemma}

Let us now introduce the sequence of functions for $\epl > 0$ that appear in \cref{lem:semigroup_ergodicity}
\begin{equation} \label{eq:def_chi_eps}
\chi^\epl(t, x) \defeq \frac{2 \Cphi}{\Cm{1}} \left[ \int_{\abs{x-y}>t} m^\epl(y)^2 \dd y + \frac{2 \Cm{2}}{t} \int_{\abs{x-y} \le t} \abs{x-y} m^\epl(y)^2 \dd y \right], \quad t \in \R^+, \; x \in \R,
\end{equation}
which satisfy a certain bound as proved in the next lemma. 

\begin{lemma} \label{lem:chi_eps_properties}
The functions $\chi^\epl$ are uniformly bounded on $\R^+ \times \R$ and they satisfy the following estimate for $t, M > 0$ sufficiently large and $x \in [-M, M]$
\begin{equation}
\chi^\epl(t, x) \le \Cphi \left[ \frac{\abs{x}+\sqrt{\log(M)}}{t} + \frac{1}{\log(M) M^{r/l}} + \frac{\exp \left(-\frac{r}{2 l} w(r^\wedge, t/2) \right)}{\sqrt{w(r^\wedge, t/2)}} + \mathds{1}_{(t/2, \infty)}(\abs{x}) \right],
\end{equation}
with a constant $C > 0$ independent of $\epl$.
\end{lemma}
\begin{proof}
The appearing constant $C>0$ in this proof will be different from line to line, but always independent of $\epl$ and $\Cphi$. The uniform boundedness follows immediately from the definition of $\chi^\epl$. Then, from the triangle inequality we have the inclusion
\begin{equation}
\{ y \in \R \; | \; \abs{x-y} > t \} \subseteq \{ y \in \R \; | \; \abs{y} > t/2 \} \cup \{ y \in \R \; | \; \abs{y} \le t/2, \; \abs{x} > t/2 \}, 
\end{equation}
which, due to \cref{lem:g_properties,rem:lambert_w_bounds}, implies for sufficiently large $t>0$
\begin{equation}
\begin{aligned}
\int_{\abs{x-y}>t} m^\epl(y)^2 \dd y &\le \int_{\abs{y}>t/2} m^\epl(y)^2 \dd y + \Cm{2} \mathds{1}_{(t/2, \infty)}(\abs{x}) \\
&\le C \left[ \frac{\exp \left(-\frac{r}{2 l} w(r^\wedge, t/2) \right)}{\sqrt{w(r^\wedge, t/2)}} + \mathds{1}_{(t/2, \infty)}(\abs{x}) \right],
\end{aligned}
\end{equation}
which bounds the first integral in equation \eqref{eq:def_chi_eps}. Then, we again use \cref{lem:g_properties,rem:lambert_w_bounds} to obtain for sufficiently large $t, M > 0$
\begin{equation}
\begin{aligned}
\frac{1}{t} \int_{\abs{x-y} \le t} \abs{x-y} m^\epl(y)^2 \dd y &= \frac{1}{t} \left[ \int_{\substack{\abs{x-y} \le t, \\ \abs{y} \le M}} \abs{x-y} m^\epl(y)^2 \dd y + \int_{\substack{\abs{x-y} \le t, \\ \abs{y} > M}} \abs{x-y} m^\epl(y)^2 \dd y \right] \\
&\le C \left[ \frac{\abs{x}}{t} + \frac{1}{t} \int_{-M}^M \abs{y} m^\epl(y)^2 \dd y + \int_{\abs{y} > M} m^\epl(y)^2 \dd y \right] \\
&\le C \left[ \frac{\abs{x}+\sqrt{\log(M)}}{t} + \frac{1}{\log(M) M^{r/l}} \right],
\end{aligned}
\end{equation}
which bounds the second integral in equation \eqref{eq:def_chi_eps} and completes the proof.
\end{proof}

We now move to the process $\xi^\epl$ and consider its transition probability function $p^\epl$, which, in this one-dimensional setting, has the explicit form \cite[Section 13]{GiS72}:
\begin{equation} \label{eq:transition_probability_density_eps}
\begin{aligned} 
p^\epl(h, x, y) &= \frac{m^\epl(y)}{\sqrt{2\pi h}} \left(\frac{m^\epl(y)}{m^\epl(x)}\right)^{1/2} \exp\left( -\frac{1}{2h} \left( \int_x^y m^\epl(z) \dd z \right)^2 \right) \\
&\quad\times \E \left[ \exp \left( h \int_0^1 B^\epl (\mathscr N(h,g_\epl(x),g_\epl(y),u)) \dd u \right) \right], \qquad x, y \in \R, \; h > 0,
\end{aligned}
\end{equation}
where
\begin{equation}
B^\epl(x) = -\frac12 b^\epl(x)^2 - \frac12 (b^\epl)'(x), \qquad b^\epl(x) \defeq -V'(x) - \frac{1}{\epl} p'\left( \frac{x}{\epl}\right), \qquad x \in \R,
\end{equation}
and
\begin{equation}
\mathscr N(h,x,y,u) \sim \mathcal N(x + u(y-x), hu(1-u)), \qquad u \in (0,1).
\end{equation}
In the next result, we give bounds on the transition probability function $p^\epl$ when $h = \epl^2$. In fact, this is the point where we encounter a recurring, yet interesting, phenomenon when dealing with multiscale diffusion problems. As we will see in the proof, in order to get good estimates on the transition probability density function, the time step $h$ between the two transitional states $x$ and $y$ should be of order $\BigO{\epl^2}$.

\begin{lemma} \label{lem:transition_probability_density_eps_bound}
There exist constants $K_1, K_2 > 0$ independent of $\epl$ such that for all $x,y \in \R$ and all sufficiently small $\epl > 0$ satisfying $K_1\epl^4 < \min \{ 2, 1/(8\sigma^2) \}$, the following bound holds
\begin{equation} \label{eq:transition_probability_density_eps_bound}
\begin{aligned}
p^\epl(\epl^2, x, y) &\le e^{\frac{K_2}{2}} \sqrt{\frac{2}{2 - K_1 \epl^4}} \exp \left( 2 K_1 \epl^2 g_\epl(x)^2 \right) \\
&\quad \times \frac{m^\epl(y)^{3/2}}{m^\epl(x)^{1/2}} \frac{1}{\sqrt{2 \pi \epl^2}} \exp \left( -\frac{1 - 8K_1 \sigma^2 \epl^4}{4 \sigma^2 \epl^2} (g_\epl(y) - g_\epl(x))^2 \right).
\end{aligned}
\end{equation}
\end{lemma}
\begin{proof}
First, note that there exist constants $K_1, K_2 > 0$ such that for all $x \in \R$ and $\epl > 0$
\begin{equation} \label{eq:derivative_drift_assumption}
-(b^\epl)'(x) \le K_1 x^2 + \frac{K_2}{\epl^2}.
\end{equation}
Using Jensen's inequality and equation \eqref{eq:derivative_drift_assumption}, we can estimate the expectation appearing in equation \eqref{eq:transition_probability_density_eps} as follows for $h>0$:
\begin{equation}
\begin{aligned}
\E \left[ \exp \left( h \int_0^1 B^\epl (\mathscr N(h,g_\epl(x),g_\epl(y),u)) \dd u \right) \right]
&\le \int_0^1 \E \left[ \exp \left( h B^\epl (\mathscr N(h,g_\epl(x),g_\epl(y),u)) \right) \right] \dd u \\
&\le \int_0^1 \E \left[ \exp \left( -\frac{h}{2} (b^\epl)'(\mathscr N(h,g_\epl(x),g_\epl(y),u)) \right) \right] \dd u \\
&\le e^{\frac{K_2 h}{2 \epl^2}} \int_0^1 \E \left[ \exp \left( \frac{K_1 h}{2} \mathscr N(h,g_\epl(x),g_\epl(y),u)^2 \right) \right] \dd u.
\end{aligned}
\end{equation}
In order to leverage against the $1/\epl^2$ factor in the exponential of the preceding estimate, we choose $h = \epl^2$ and obtain
\begin{equation}
\begin{aligned}
&\E \left[ \exp \left( \epl^2 \int_0^1 B^\epl (\mathscr N(\epl^2,g_\epl(x),g_\epl(y),u)) \dd u \right) \right] \\
&\hspace{3cm}\le e^{\frac{K_2}{2}} \int_0^1 \E \left[ \exp \left( \frac{K_1 \epl^2}{2} \mathscr N(\epl^2,g_\epl(x),g_\epl(y),u)^2 \right) \right] \dd u, \\
&\hspace{3cm} \eqdef e^{\frac{K_2}{2}} \mathcal I_\epl(x,y),
\end{aligned}
\end{equation}
where we note that for $u \in (0,1)$
\begin{equation}
\mathscr N(\epl^2,g_\epl(x),g_\epl(y),u) \sim g_\epl(x) + u(g_\epl(y) - g_\epl(x)) + \sqrt{\epl^2 u(1-u)} \mathcal Z, \qquad \mathcal Z \sim \mathcal{N}(0, 1).
\end{equation}
We then get
\begin{equation}
\begin{aligned}
\mathcal I_\epl(x,y) &\le \E \left[ \exp \left( K_1 \epl^2 \left( \left[ g_\epl(x) + u(g_\epl(y) - g_\epl(x)) \right]^2 + \epl^2 u(1-u) \mathcal Z^2 \right) \right) \right] \\
&\le \exp \left( K_1 \epl^2 \left[ g_\epl(x) + u(g_\epl(y) - g_\epl(x)) \right]^2 \right) \E \left[ \exp \left(\frac{K_1 \epl^4}{4} \mathcal Z^2 \right) \right], \\
\end{aligned}
\end{equation}
which implies
\begin{equation}
\begin{aligned}
\mathcal I_\epl(x,y) &\le \exp \left( 2 K_1 \epl^2 \left[ g_\epl(x)^2 + (g_\epl(y) - g_\epl(x))^2 \right] \right) \E \left[ \exp \left(\frac{K_1 \epl^4}{4} \mathcal Z^2 \right) \right] \\
&= \exp \left( 2 K_1 \epl^2 \left[ g_\epl(x)^2 + (g_\epl(y) - g_\epl(x))^2 \right] \right) \sqrt{\frac{2}{2 - K_1 \epl^4}}.
\end{aligned}
\end{equation}
Hence, we have for sufficiently small $\epl>0$ and $x, y \in \R$ the following estimate 
\begin{equation} \label{eq:estimate_expectation_transition}
\begin{aligned}
&\E \left[ \exp \left( \epl^2 \int_0^1 B^\epl (N(\epl^2,g_\epl(x),g_\epl(y),u)) \dd u \right) \right] \\
&\hspace{2cm}\le \exp\left( \frac{K_2}{2} \right) \exp \left( 2 K_1 \epl^2 \left[ g_\epl(x)^2 + (g_\epl(y) - g_\epl(x))^2 \right] \right) \sqrt{\frac{2}{2 - K_1 \epl^4}}.
\end{aligned}
\end{equation}
We now investigate the other terms in equation \eqref{eq:transition_probability_density_eps}. Recalling the definition of $m^\epl$ and using an integral substitution give
\begin{equation}
\int_x^y m^\epl(z) \dd z = \frac{g_\epl(y) - g_\epl(x)}{\sqrt{2\sigma^2}},
\end{equation}
which implies
\begin{equation}
\begin{aligned}
&\frac{1}{\sqrt{2 \pi \epl^2}} \exp\left( -\frac{1}{2 \epl^2} \left( \int_x^y m^\epl(z) \dd z \right)^2 \right) \exp \left( 2 K_1 \epl^2 (g_\epl(y) - g_\epl(x))^2 \right) \\
&\hspace{4cm}= \frac{1}{\sqrt{2 \pi \epl^2}} \exp \left( -\left( \frac{1}{4 \sigma^2 \epl^2} - 2 K_1 \epl^2 \right) (g_\epl(y) - g_\epl(x))^2 \right) \\
&\hspace{4cm}= \frac{1}{\sqrt{2 \pi \epl^2}} \exp \left( -\frac{1 - 8K_1 \sigma^2 \epl^4}{4 \sigma^2 \epl^2} (g_\epl(y) - g_\epl(x))^2 \right).
\end{aligned}
\end{equation}
This, together with equation \eqref{eq:estimate_expectation_transition}, yields the desired result.
\end{proof}

The next result is a direct consequence of the just established \cref{lem:transition_probability_density_eps_bound}.

\begin{corollary} \label{cor:tpd_eps_estimate}
There exist constant $C, K_1 > 0$ independent of $\epl$ such that for all $x_0,y \in \R$ and all sufficiently small $\epl > 0$ satisfying $K_1\epl^4 < \min \{ 2, 1/(8\sigma^2) \}$, the following bound holds
\begin{equation} \label{eq:tpd_eps_estimate}
p^\epl(\epl^2,f_\epl(x_0),y) \le C m^\epl(y)^{3/2} \phi_\epl(g_\epl(y) - x_0),
\end{equation}
where $\phi_\epl$ is the probability density function of the centered Gaussian $\mathcal N(0, 2 \sigma^2 \epl^2/(1 - 8K_1 \sigma^2 \epl^4))$.
\end{corollary}
\begin{proof}
The desired result follows by substituting $x = f_\epl(x_0)$ into \cref{lem:transition_probability_density_eps_bound}.
\end{proof}

We now come to the main result of this section, which is the $\epl$-dependent mean ergodic theorem.

\begin{proposition} \label{pro:eps_mean_ergodic_theorem}
Let $x_0 \in \R$ and $x_0^\epl := f_\epl(x_0)$. For any bounded, measurable function $\varphi \colon \R \to \R$ we have as $\epl \to 0$
\begin{equation} \label{eq:eps_mean_ergodic_theorem_orig}
\E^{x_0^\epl} \left[ \abs{\frac{1}{T_\epl} \int_0^{T_\epl} \varphi(X^\epl_t)  \dd t - \int_\R \varphi(x) \rho^\epl(x) \dd x}^2 \right] = 
\begin{cases}
\BigO{\Cphi^2 \frac{\log(T_\epl)^{3/4}}{\sqrt{\epl T_\epl}}}, & \text{if } r \ge l, \\
\BigO{\Cphi^2 \frac{1}{\sqrt{\epl T_\epl^{r/l}}}},            & \text{if } r < l,
\end{cases}
\end{equation}
where $l$ and $r$ are defined in equation \eqref{eq:lr_def}.
\end{proposition}
\begin{proof}
The constant $C>0$ in this proof will be different from line to line, but always independent of $\epl$ and $\Cphi$. Let
\begin{equation}
\Bar{\varphi}(x) \defeq \varphi(x) - \frac{1}{C_{m^\epl}} \int_\R \varphi(y) m^\epl(y)^2 \dd y, \quad x \in \R.
\end{equation}
First, observe that 
\begin{equation}
\begin{aligned}
\E^{x} \left[ \abs{\frac{1}{T_\epl} \int_0^{T_\epl} \Bar{\varphi}(\xi^\epl_t) \dd t}^2 \right] &= \frac{2}{T_\epl^2} \E^{x} \left[ \int_0^{T_\epl} \Bar{\varphi}(\xi^\epl_t) \int_t^{T_\epl} \Bar{\varphi}(\xi^\epl_s) \dd s \dd t \right] \\
&= \frac{2}{T_\epl^2} \int_0^{T_\epl} \E^{x} \left[ \Bar{\varphi}(\xi^\epl_t) \E^{x} \left[ \left. \int_t^{T_\epl}  \Bar{\varphi}(\xi^\epl_s) \dd s \right| \mathcal{F}_t \right] \right] \dd t \\
&=\frac{2}{T_\epl^2} \int_0^{T_\epl} \E^{x} \left[ \Bar{\varphi}(\xi^\epl_t) \E^{x} \left[ \left. \int_0^{T_\epl-t} \Bar{\varphi}(\xi^\epl_{s+t}) \dd s \right| \mathcal{F}_t \right] \right] \dd t.
\end{aligned}
\end{equation}
Using the Markov property and \cref{lem:semigroup_ergodicity}, we then find that
\begin{equation}
\begin{aligned}
\E^{x} \left[ \abs{\frac{1}{T_\epl} \int_0^{T_\epl} \Bar{\varphi}(\xi^\epl_t) \dd t}^2 \right]^2 &= \frac{2}{T_\epl^2} \int_0^{T_\epl} \E^{x} \left[ \Bar{\varphi}(\xi^\epl_t) \int_0^{T_\epl-t} \E^{\xi^\epl_t} \left[ \Bar{\varphi}(\xi^\epl_s) \right] \dd s \right] \dd t \\
&\le \frac{4 \Cphi}{T_\epl} \int_0^{T_\epl} \E^{x} \left[ \chi^\epl(T_\epl-t, \xi^\epl_t) \right] \dd t,
\end{aligned}
\end{equation}
which, by applying \cref{lem:stationarity}, gives
\begin{equation} \label{eq:bound_expectation_chi}
\begin{aligned}
\int_\R \E^{x} \left[ \left| \frac{1}{T_\epl} \int_0^{T_\epl} \Bar{\varphi}(\xi^\epl_t) \dd t \right|^2 \right] m^\epl(x)^2 \dd x 
&\le \frac{4 \Cphi}{T_\epl} \int_0^{T_\epl} \int_\R \E^{x} \left[ \chi^\epl(T_\epl-t, \xi^\epl_t) \right] m^\epl(x)^2 \dd x \dd t \\
&= \frac{4 \Cphi}{T_\epl} \int_0^{T_\epl} \int_\R \chi^\epl(T_\epl-t, y) m^\epl(y)^2 \dd y \dd t.
\end{aligned}
\end{equation}
Then, by \cref{lem:g_properties,rem:lambert_w_bounds}, it follows for $M > T_\epl/2$ sufficiently large
\begin{equation} \label{eq:split_integral_space}
\begin{aligned}
&\frac{1}{T_\epl} \int_0^{T_\epl} \int_\R \chi^\epl(T_\epl-t, y) m^\epl(y)^2 \dd y \dd t \\
&\hspace{2cm}\le \frac{1}{T_\epl} \int_0^{T_\epl} \int_{-M}^{M} \chi^\epl(T_\epl-t, y) m^\epl(y)^2 \dd y \dd t + C \Cphi \int_{\abs{y}>M} m^\epl(y)^2 \dd y \\
&\hspace{2cm}\le \frac{1}{T_\epl} \int_0^{T_\epl} \int_{-M}^{M} \chi^\epl(t, y) m^\epl(y)^2 \dd y \dd t + \frac{C \Cphi}{\sqrt{\log(M)}M^{r/l}},
\end{aligned}
\end{equation}
and the first term in the right-hand side can be split by choosing a sufficiently large $0 < t_0 < T_\epl$
\begin{equation}
\begin{aligned}
\frac{1}{T_\epl} \int_0^{T_\epl} \int_{-M}^{M} \chi^\epl(t, y) m^\epl(y)^2 \dd y \dd t &= \frac{1}{T_\epl} \int_0^{t_0} \int_{-M}^{M} \chi^\epl(t, y) m^\epl(y)^2 \dd y \dd t \\
&\quad + \frac{1}{T_\epl} \int_{t_0}^{T_\epl} \int_{-M}^{M} \chi^\epl(t, y) m^\epl(y)^2 \dd y \dd t.
\end{aligned}
\end{equation}
Therefore, using \cref{lem:chi_eps_properties,lem:g_properties,rem:lambert_w_bounds}, we get 
\begin{equation} \label{eq:bound_integral_interior_space}
\begin{aligned}
&\frac{1}{T_\epl} \int_0^{T_\epl} \int_{-M}^{M} \chi^\epl(t, y) m^\epl(y)^2 \dd y \dd t \\
&\qquad\le C \Cphi \left[ \frac{1}{T_\epl} + \frac{1}{T_\epl} \int_{t_0}^{T_\epl} \int_{-M}^M \left( \frac{\abs{y}+\sqrt{\log(M)}}{t} + \frac{1}{\sqrt{\log(M)} M^{r/l}} \right) m^\epl(y)^2 \dd y \dd t \right. \\
&\quad\qquad \left. + \frac{1}{T_\epl} \int_{t_0}^{T_\epl} \int_{-M}^M \left( \frac{\exp \left(-\frac{r}{2 l} w(r^\wedge, t/2) \right)}{\sqrt{w(r^\wedge, t/2)}} + \mathds{1}_{(t/2, \infty)}(\abs{y}) \right) m^\epl(y)^2 \dd y \dd t \right] \\
&\qquad\le C \Cphi \left[ \frac{1}{T_\epl} + \frac{\sqrt{\log(M)} \log(T_\epl)}{T_\epl} + \frac{1}{\sqrt{\log(M)}M^{r/l}} + \frac{1}{T_\epl} \int_{t_0}^{T_\epl} \frac{\exp \left(-\frac{r}{2 l} w(r^\wedge, t/2) \right)}{\sqrt{w(r^\wedge, t/2)}} \dd t \right]\;.
\end{aligned}
\end{equation}
Using equations \eqref{eq:bound_expectation_chi}, \eqref{eq:split_integral_space}, \eqref{eq:bound_integral_interior_space}, recalling \cref{rem:lambert_w_bounds}, and choosing $M = T_\epl^{l/r}$, we obtain for sufficiently small $\epl > 0$
\begin{equation} \label{eq:estimates_BigO}
\int_\R \E^{x} \left[ \abs{\frac{1}{T_\epl} \int_0^{T_\epl} \Bar{\varphi}(\xi^\epl_t) \dd t}^2 \right] m^\epl(x)^2 \dd x =
\begin{cases}
\BigO{\Cphi^{2} \frac{\log(T_\epl)^{3/2}}{T_\epl}}, & \text{if } r \ge l, \\
\BigO{\Cphi^{2} \frac{1}{T_\epl^{r/l}}},            & \text{if } r < l.
\end{cases}
\end{equation}
Next, observe that by \cref{cor:tpd_eps_estimate} and the Cauchy-Schwarz inequality we have
\begin{equation} \label{eq:almost_final_estimate_MET}
\begin{aligned}
&\int_\R \E^y \left[ \abs{\frac{1}{T_\epl-\epl^2} \int_0^{T_\epl-\epl^2} \Bar{\varphi}(\xi^\epl_t) \dd t}^2 \right] p^\epl(\epl^2, x_0^\epl, y) \dd y \\
&\qquad\le C \int_\R \E^y \left[ \abs{\frac{1}{T_\epl-\epl^2} \int_0^{T_\epl-\epl^2} \Bar{\varphi}(\xi^\epl_t) \dd t}^2 \right] m^\epl(y)^{3/2} \phi_\epl(g_\epl(y) - x_0) \dd y \\
&\qquad\le C \left( \int_\R \E^y \left[ \abs{\frac{1}{T_\epl-\epl^2} \int_0^{T_\epl-\epl^2} \Bar{\varphi}(\xi^\epl_t) \dd t}^4 \right] m^\epl(y)^{2} \dd y \right)^{1/2} \left( \int_\R m^\epl(y) \phi_\epl(g_\epl(y) - x_0)^2 \dd y \right)^{1/2},
\end{aligned}
\end{equation}
which, due to the boundedness of $\varphi$, a change of variables, and estimates \eqref{eq:estimates_BigO}, implies
\begin{equation}
\begin{aligned}
&\int_\R \E^y \left[ \abs{\frac{1}{T_\epl-\epl^2} \int_0^{T_\epl-\epl^2} \Bar{\varphi}(\xi^\epl_t) \dd t}^2 \right] p^\epl(\epl^2, x_0^\epl, y) \dd y \\
&\qquad\le C \Cphi \left( \int_\R \E^y \left[ \abs{\frac{1}{T_\epl-\epl^2} \int_0^{T_\epl-\epl^2} \Bar{\varphi}(\xi^\epl_t) \dd t}^2 \right] m^\epl(y)^{2} \dd y \right)^{1/2} \left( \int_\R \phi_\epl(y - x_0)^2 \dd y \right)^{1/2} \\
&\qquad= \begin{cases}
\BigO{\Cphi^2 \frac{\log(T_\epl)^{3/4}}{\sqrt{\epl T_\epl}}}, & \text{if } r \ge l, \\
\BigO{\Cphi^2 \frac{1}{\sqrt{\epl T_\epl^{r/l}}}},            & \text{if } r < l.
\end{cases}
\end{aligned}
\end{equation}
Finally, notice that using the Markov property it follows that
\begin{equation}
\begin{aligned}
\E^{x_0^\epl} \left[ \abs{\frac{1}{T_\epl} \int_0^{T_\epl} \Bar{\varphi}(\xi^\epl_t) \dd t}^2 \right] &= \frac{1}{T_\epl^2} \E^{x_0^\epl} \left[ \abs{\int_0^{\epl^2} \Bar{\varphi}(\xi^\epl_t) \dd t}^2 \right] +\frac{1}{T_\epl^2} \E^{x_0^\epl} \left[ \abs{\int_{\epl^2}^{T_\epl} \Bar{\varphi}(\xi^\epl_t) \dd t}^2 \right] \\
&\quad+ \frac{2}{T_\epl^2} \E^{x_0^\epl} \left[ \left( \int_0^{\epl^2} \Bar{\varphi}(\xi^\epl_t) \dd t \right) \left( \int_{\epl^2}^{T_\epl} \Bar{\varphi}(\xi^\epl_t) \dd t \right) \right] \\
&\le \frac{4 \epl^4 \Cphi^2}{T_\epl^2} + \frac{8 \epl^2 \Cphi^2}{T_\epl} + \frac{1}{T_\epl^2} \int_\R \E^y \left[ \abs{\int_0^{T_\epl-\epl^2} \Bar{\varphi}(\xi^\epl_t) \dd t}^2 \right] p^\epl(\epl^2, x_0^\epl, y) \dd y,
\end{aligned}
\end{equation}
which, together with estimates \eqref{eq:almost_final_estimate_MET} and equation \eqref{eq:transformation_MET}, concludes the proof.
\end{proof}

\begin{remark}
The convergence rates provide valuable insight into the interplay between three key parameters: the Lipschitz constant $L_V$, which bounds the maximal slope of the drift, the factor $\abs{V'(0)}$, representing the offset of the drift from the origin, and the dissipativity constant $\beta$, which drives recurrence and ergodicity. These appear in the ratio $r/l = \beta/(L_V + \abs{V'(0)})$. This expression highlights that, for fixed $|V'(0)|$, increasing $\beta$, which corresponds to a stronger inward pull toward the origin at infinity, and decreasing $L_V$, indicating less variability in the drift, lead to faster convergence. Moreover, a larger value for $|V'(0)|$ makes the process ``work'' longer before pulling back to the origin through the dissipativity. However, it should be noted that a limitation of the current proof is that the convergence rate in the regime $r \ge l$ cannot be improved by increasing the ratio $r/l$.
\end{remark}

\subsubsection{Application of the mean ergodic theorem and final convergence result}

We now apply the $\epl$-dependent mean ergodic theorem derived in the previous section to analyze how the stochastic term $\mathcal Q_N^{T,\epl}$ in equation \eqref{eq:deterministic_stochastic_terms} depends on the parameters $N, T$, and $\epl$.  The next result provides explicit convergence rates.

\begin{lemma} \label{lem:rates_stochastic_term}
Under \cref{as:potentials}, it holds
\begin{equation}
\E \left[ \norm{\widehat \rho_N^{T,\epl} - \rho_N^\epl}_{L^2(\R)}^2 \right] = 
\begin{cases}
\BigO{\frac{N\log(T)^{3/4}}{\sqrt{\epl T}}}, & \text{if } r \ge l, \\
\BigO{\frac{N}{\sqrt{\epl T^{r/l}}}},        & \text{if } r < l,
\end{cases}
\end{equation}
where $r$ and $l$ are defined in equation \eqref{eq:lr_def}.
\end{lemma}
\begin{proof}
By definition of $\widehat \rho_N^{T,\epl}$ and $\rho_N^\epl$ in equations \eqref{eq:estimator_density} and \eqref{eq:rho_N_e_def}, respectively, and since the sequence $\{ \psi_n \}_{n=0}^\infty$ forms an orthonormal basis of $L^2(\R)$, we have
\begin{equation}
\begin{aligned}
\E \left[ \norm{\widehat \rho_N^{T,\epl} - \rho_N^\epl}_{L^2(\R)}^2 \right] &= \sum_{n=0}^{N-1} \E \left[ \left( \widehat \alpha_n^{T,\epl} - \alpha_n^\epl \right)^2 \right] = \sum_{n=0}^{N-1} \E \left[ \abs{\frac1{T} \int_0^T \psi_n(X_t^\epl) \dd t - \int_\R \psi_n(x) \rho^\epl(x) \dd x}^2 \right].
\end{aligned}
\end{equation}
The desired result then follows by applying \cref{pro:eps_mean_ergodic_theorem} and noting that $\norm{\psi_n}_\infty$ is uniformly bounded independently of $n$ due to the Cramér's inequality in equation \eqref{eq:Cramer}.
\end{proof}

Combining the results obtained in the previous sections, we are now ready to finalize the proof of \cref{thm:convergence}

\begin{proof}[Proof of \cref{thm:convergence}]
Recalling the triangle inequality from equation \eqref{eq:triangle}, we have
\begin{equation}
\E \left[ \norm{\widehat \rho_{N(\epl)}^{T(\epl),\epl} - \rho}_{L^2(\R)}^2 \right] \le 2 \left( \norm{\rho_{N(\epl)}^\epl - \rho}_{L^2(\R)}^2 + \E \left[ \norm{\widehat \rho_{N(\epl)}^{T(\epl),\epl} - \rho_{N(\epl)}^\epl}_{L^2(\R)}^2 \right] \right),
\end{equation}
which, due to \cref{lem:rates_stochastic_term} and using the specified choices for $N$ and $T$, implies that for some constant $C > 0$
\begin{equation}
\E \left[ \norm{\widehat \rho_{N(\epl)}^{T(\epl),\epl} - \rho}_{L^2(\R)}^2 \right] \le 2\norm{\rho_{N(\epl)}^\epl - \rho}_{L^2(\R)}^2 + \begin{cases}
C \abs{\log(\epl)}^{3/4} \epl^{\frac{\zeta - 5}{2}}, & \text{if } r \ge l, \\
C \epl^{\frac{r\zeta - 5l}{2l}}, & \text{if } r < l.
\end{cases}
\end{equation}
Then, the first term on the right-hand side vanishes by \cref{pro:General-Case,rem:subsequences}, while the second term tends to zero under the assumptions on $\zeta$. This completes the proof.
\end{proof}

\begin{remark}
\cref{lem:rates_stochastic_term} and the proof of \cref{thm:convergence} highlight an important difference between the number of Fourier modes $N$ and the observation time $T$. Unlike $N$, which must increase slowly relative to the inverse of the scale parameter $\epl$, the observation time $T$ must grow sufficiently fast. In other words, while $N$ is restricted and cannot grow too quickly, $T$ should be chosen large enough to ensure the convergence of the estimator.
\end{remark}

\section{Conclusion} \label{sec:conclusion}

In this work, we addressed the problem of learning the homogenized invariant measure from multiscale data. Our proposed estimator is based on a truncated Fourier series, using Hermite functions as the basis. The Fourier coefficients are estimated from the data by leveraging the ergodic theorem. To handle model apparent misspecification, we carefully selected the number of Fourier modes and the time of observation to ensure that fast-scale oscillations are not captured, and we rigorously demonstrated that this approach yields an asymptotically unbiased estimator of the invariant density.

The methodology introduced in this work can be extended in several directions and applied to a variety of inference problems. One natural extension is to the multidimensional setting, where the basis functions are constructed as tensor products of Hermite functions. Another promising direction is to establish a central limit theorem in $L^2(\R)$, which would complement the current results by proving not only asymptotic unbiasedness but also asymptotic normality of the estimator. With the machinery that is so far available to us, we are able to prove the asymptotic normality of the estimator to a Gaussian series element in a certain weak Sobolev-type Hilbert space with negative scale exponent, but proving tightness of the relevant measures in $L^2(\R)$, which would be the natural space to have here, remains unclear. Moreover, while our convergence analysis relies heavily on the specific properties of Hermite functions, the framework is, in principle, compatible with any orthonormal basis. We are therefore interested in exploring alternative bases in future work. Additionally, we aim to consider more general multiscale models, such as processes involving nonseparable fast-scale potentials, which lead to multiplicative noise, or driven by colored noise \cite{PRZ25}. We also believe that the strategy of discarding higher-order modes to mitigate model misspecification could be beneficial in other contexts where such issues arise.

Finally, this approach may also prove useful for nonparametric estimation of the drift term and/or the confining potential. In particular, the following identities hold
\begin{equation} \label{eq:estimation_drift}
\mathcal V'(x) = - \Sigma \frac{\rho'(x)}{\rho(x)} \qquad \text{and} \qquad \mathcal V(x) = - \Sigma \log(\rho(x)) - \Sigma \log(Z),
\end{equation}
which express both $\mathcal{V}'$ and $\mathcal{V}$ in terms of the invariant density $\rho$. Estimating $\mathcal{V}'$ requires the additional computation of $\rho'$, while estimating $\mathcal{V}$ necessitates an extra condition (e.g., zero mean with respect to $\rho$) to fix the additive constant. Moreover, knowledge of the homogenized diffusion coefficient $\Sigma$ is required in both cases and, therefore, must be estimated in advance. Nonparametric estimators for $\mathcal{V}'$ and $\mathcal{V}$ are obtained by substituting $\rho$ with its approximation $\widehat \rho_N^{T,\epl}$ in equation \eqref{eq:estimation_drift}. A detailed analysis of the performance and convergence of these estimators is left for future work.

\subsection*{Acknowledgements} 

The authors JIB and SK acknowledge funding by the Deutsche Forschungsgemeinschaft (DFG, German Research Foundation) - Project number 442047500 through the Collaborative Research Center ``Sparsity and Singular Structures'' (SFB 1481). AZ is supported by ``Centro di Ricerca Matematica Ennio De Giorgi'' and the ``Emma e Giovanni Sansone'' Foundation, and is a member of INdAM-GNCS. This material is based upon work supported by the National Science Foundation Graduate Research Fellowship Program under Grant No.\ DGE 2146752. Any opinions, findings, and conclusions or recommendations expressed in this material are those of the authors and do not necessarily reflect the views of the National Science Foundation.

\bibliographystyle{siamnodash}
\bibliography{biblio}

\end{document}

%% file: ex_shared.tex
\usepackage[T1]{fontenc}
\usepackage{lmodern}
\usepackage[utf8]{inputenc}
\usepackage{microtype}
\usepackage{framed}
\usepackage{listings}
\usepackage[margin=1in]{geometry}
\usepackage{setspace}
\usepackage{mathrsfs, mathenv}
\usepackage{amsmath, amsthm, amssymb, amsfonts, amscd}
\usepackage{graphicx}
\usepackage[svgnames]{xcolor}
\usepackage{hyperref}
\hypersetup{citecolor=blue, colorlinks=true, linkcolor=black}
\usepackage[ruled, vlined]{algorithm2e}
\usepackage[capitalise]{cleveref}
\usepackage{subcaption}
\usepackage{bbm}
\usepackage{cite}
\usepackage{verbatim}
\usepackage{pgfplots}
\usepackage{tikz}
\usetikzlibrary{arrows,decorations.pathmorphing,backgrounds,positioning,fit,matrix}
\usepackage{etoolbox}
\usepackage{color}
\usepackage{lipsum}
\usepackage{ifthen}
\usepackage[title]{appendix}
\usepackage{autonum}
\usepackage{accents}
\usepackage{xpatch}
\usepackage[]{algpseudocode}
\usepackage{multicol}
\usepackage[abs]{overpic}
\usepackage{eqnarray}
\usepackage{dsfont}

\crefname{assumption}{Assumption}{Assumptions}
\crefname{figure}{Figure}{Figures}

\theoremstyle{plain}
\newtheorem{theorem}{Theorem}[section]
\newtheorem{corollary}[theorem]{Corollary}
\newtheorem{lemma}[theorem]{Lemma}
\newtheorem{proposition}[theorem]{Proposition}

\numberwithin{equation}{section}

\theoremstyle{definition}

\theoremstyle{remark}
\newtheorem{remark}[theorem]{Remark}
\newtheorem{assumption}[theorem]{Assumption}

\ifpdf
  \DeclareGraphicsExtensions{.eps,.pdf,.png,.jpg}
\else
  \DeclareGraphicsExtensions{.eps}
\fi

\usepackage{mathtools}

\usepackage{booktabs}

\usepackage{pgfplots}
\usepackage{tikz}
\usetikzlibrary{patterns,arrows,decorations.pathmorphing,backgrounds,positioning,fit,matrix}
\usepackage[labelfont=bf]{caption}
\usepackage{here}
\usepackage[font=normal]{subcaption}

\usepackage{enumitem}
\setlist[itemize]{leftmargin=.5in}
\setlist[enumerate]{leftmargin=.5in,topsep=3pt,itemsep=3pt,label=(\roman*)}

\newcommand*\samethanks[1][\value{footnote}]{\footnotemark[#1]}

\newcommand{\email}[1]{\href{#1}{#1}}
\newcommand{\TheTitle}{Nonparametric estimation of homogenized invariant measures from multiscale data via Hermite expansion}
\newcommand{\TheAuthors}{}
\title{\TheTitle}
\author{Jaroslav I. Borodavka \thanks{Scientific Computing Center, Karlsruhe Institute of Technology, Karlsruhe, Germany, \email{jaroslav.borodavka@kit.edu}, \email{sebastian.krumscheid@kit.edu}.}
\and Max Hirsch \thanks{Department of Mathematics, University of California Berkeley, Berkeley, CA, USA, \email{mhirsch@berkeley.edu}.}
\and Sebastian Krumscheid \samethanks[1]
\and Andrea Zanoni \thanks{Centro di Ricerca Matematica Ennio De Giorgi, Scuola Normale Superiore, Pisa, Italy, \email{andrea.zanoni@sns.it}.}}
\date{}

\usepackage{amsopn}

\DeclareMathOperator{\sign}{sgn}

\newcommand{\abs}[1]{\left\lvert#1\right\rvert}
\newcommand{\norm}[1]{\left\|#1\right\|}
\renewcommand{\Pr}{\mathbb{P}}

\newcommand{\V}{\mathcal{V}}
\newcommand{\N}{\mathbb{N}}

\newcommand{\period}{\mathbb{L}}

\newcommand{\R}{\mathbb{R}}

\newcommand{\C}{\mathbb{C}}

\newcommand{\epl}{\varepsilon}

\newcommand{\defeq}{\coloneqq}
\newcommand{\eqdef}{\eqqcolon}

\newcommand{\E}{\operatorname{\mathbb{E}}}

\renewcommand{\d}{\mathrm{d}}
\newcommand{\dd}{\,\mathrm{d}}
\definecolor{shade}{RGB}{100, 100, 100}
\definecolor{bordeaux}{RGB}{128, 0, 50}

\usepackage[usestackEOL]{stackengine}

\newcommand{\BigO}[1]{\mathcal{O} \left( #1 \right)}
\newcommand{\Cphi}{\| \varphi \|_\infty}
\newcommand{\Cm}[1]{C_{m}^{(#1)}}
\newcommand{\Cf}[1]{C_f^{(#1)}}

\definecolor{leg1}{RGB}{0,114,189}
\definecolor{leg2}{RGB}{217,83,25}
\definecolor{leg3}{RGB}{237,177,32}
\definecolor{leg4}{RGB}{126,47,142}
\definecolor{leg5}{RGB}{119,172,48}

\definecolor{leg21}{RGB}{62,38,169}
\definecolor{leg22}{RGB}{46,135,247}
\definecolor{leg23}{RGB}{55,200,151}
\definecolor{leg24}{RGB}{254,195,56}

\ifpdf
\hypersetup{
	pdftitle={\TheTitle},
	pdfauthor={\TheAuthors}
}
\fi